\documentclass[10pt]{amsart}
\usepackage{amsfonts}
\usepackage{array}
\usepackage{tabularx}
\usepackage{arydshln}
\usepackage{amsmath}

\usepackage{amssymb}
\usepackage{physics}
\usepackage{graphicx}
\usepackage{caption}
\usepackage{subcaption}
\usepackage{multirow}
\usepackage{footmisc}
\usepackage{xcolor}

\providecommand{\U}[1]{\protect\rule{.1in}{.1in}}

\usepackage{xcolor}
\usepackage{latexsym}
\usepackage{amsmath}
\usepackage{amssymb}
\usepackage{mathrsfs}
\usepackage{graphicx}
\usepackage{color}
\usepackage{pgfpages}
\usepackage{ifthen}
\usepackage{leftidx,tensor}
\usepackage[T1]{fontenc}
\usepackage[latin1]{inputenc}
\usepackage{mathtools}
\usepackage{comment}
\usepackage{dsfont}
\usepackage[nocompress]{cite}
\usepackage[shortlabels]{enumitem}
\usepackage{aliascnt}
\usepackage[bookmarks=true,pdfstartview=FitH, pdfborder={0 0 0}, colorlinks=true,citecolor=red, linkcolor=blue]{hyperref}
\makeatother
\usepackage{nicefrac}
\newtheorem{thm}{Theorem}[section]
\newaliascnt{cor}{thm}
\newaliascnt{prop}{thm}
\newaliascnt{lem}{thm}
\newtheorem{cor}[cor]{Corollary}
\newtheorem{prop}[prop]{Proposition}
\newtheorem{lem}[lem]{Lemma}
\aliascntresetthe{cor}
\aliascntresetthe{prop}
\aliascntresetthe{lem}

\newaliascnt{defn}{thm}
\newaliascnt{asu}{thm}
\newaliascnt{con}{thm}
\newtheorem{defn}[defn]{Definition}
\newtheorem{asu}[asu]{Assumption}

\aliascntresetthe{defn}
\aliascntresetthe{asu}
\aliascntresetthe{con}

\newcounter{stp}
\newcounter{stpi}
\newcounter{stpci}
\newcounter{stpiii}

\newaliascnt{rem}{thm}
\newaliascnt{exa}{thm}
\newaliascnt{masu}{thm}
\newaliascnt{nota}{thm}
\newaliascnt{sett}{thm}
\newtheorem{rem}[rem]{Remark}

\aliascntresetthe{rem}
\aliascntresetthe{exa}
\aliascntresetthe{masu}
\aliascntresetthe{nota}
\aliascntresetthe{sett}
\numberwithin{equation}{section}

\setlist[enumerate]{font = \normalfont}

\renewcommand {\S}	{\mathbb{S}}

\newcommand {\R}	{\mathbb{R}}

\renewcommand{\d}{\, \mathrm{d}}

\renewcommand{\H}{\mathrm{H}}

	\renewcommand{\phi}{\varphi}
	
	\renewcommand{\bar}[1]{\overline{#1}}

	\renewcommand{\div}{\mathrm{div} \, }

	\newcommand{\rC}{\mathrm{C}}
	\newcommand{\rL}{\mathrm{L}}
	\newcommand{\rW}{\mathrm{W}}
	\newcommand{\rH}{\H}

\title[Variational Structure of EBMs with Heterogeneous Insolation]{Variational Structure of Energy-Balance Climate Models with Heterogeneous Insolation}

\author{Gianmarco Del Sarto}
\address{Technische Universit\"{a}t Darmstadt\\
Fachbereich Mathematik\\
	Schlossgartenstr.\ 7\\
	64289 Darmstadt\\
	Germany}
\email{delsarto@mathematik.tu-darmstadt.de}

\author{Giuseppe Floridia}
\address{Sapienza Universit\`a di Roma\\
Dipartimento di Scienze di Base e Applicate per l'Ingegneria\\
	Via Antonio Scarpa 16\\
	00161 Roma\\
	Italy}
\email{giuseppe.floridia@uniroma1.it}
\begin{document}

\maketitle

\begin{abstract}
We analyse stationary Sellers-type energy-balance climate models with heterogeneous insolation $Q_\varepsilon = Q_*+\varepsilon Q_1$, both on the interval $(-1,1)$ with degenerate diffusion $\kappa(x)=1-x^2$ and on $\mathbb{S}^2$ with uniformly elliptic diffusion. Assuming that the homogeneous model corresponding to $\varepsilon = 0$ admits three constant equilibria $\theta_S <\theta_M < \theta_W$ satisfying suitable non-degeneracy and spectral conditions, we prove that, for sufficiently small $\vert \varepsilon \vert $, these equilibria continue to three stationary branches $ u_S^\varepsilon$, $u_M^\varepsilon$, and $u_W^\varepsilon$.
The cold and warm branches, $u_S^\varepsilon$ and $u_W^\varepsilon$ respectively, are strict local minimisers of the associated
energy functional, while the intermediate branch $u_M^\varepsilon$ is non-degenerate and has
Morse index one. We further establish the Palais-Smale condition and obtain
a mountain-pass critical point $u^\varepsilon_{MP}$ at the minimax level between the two local minima. Under an
additional isolation assumption for the homogeneous problem at the minimax
level, this critical point coincides with $u_M^\varepsilon$. Lastly, a
Legendre-Galerkin discretisation of the one-dimensional model numerically
illustrates the three stationary branches, their discrete Morse indices, and
the corresponding energy landscape.
\end{abstract}

\tableofcontents

\section{Introduction and main results}
\subsection{Motivation and modelling background}
\noindent
Energy-balance models (EBMs) are a class of reduced climate models in which
the Earth's temperature evolves according to the balance between the radiation
absorbed and emitted by the planet. They were introduced independently in the late 1960s by Budyko and Sellers \cite{Budyko1969,Sellers1969}, and have been widely used for their ability to reproduce important qualitative features of the
Earth's climate, see \cite{North1981}. A classical feature of EBMs is the
possible coexistence of two different stable climate states: a warm state, very similar to the present climate of the Earth, and a \emph{Snowball Earth} state in which ice sheets extend over
all the planetary surface. These two states are separated by an
unstable climate
\cite{North1975,LucariniGhil2020}.
This bistable picture is further motivated by palaeoclimatic evidence for
severe low-latitude glaciations during the Neoproterozoic
\cite{Hoffmann1998,Hoffmann2002}.

On a spatial domain $\mathcal{M}$, an EBM for
the unknown Earth's absolute temperature $u=u(t,\xi)$ takes the form
\begin{equation}
\partial_t u =\div _{\mathcal{M}} (\kappa \nabla_{\mathcal{M}}u ) + Q_\varepsilon(\xi)\beta_{\mathrm{ph}}(u) +q
-\mathfrak e\, u|u|^3,
\qquad \xi \in\mathcal{M}.
\label{eq:introduction-parabolic-ebm}
\end{equation}
Here $\kappa$ is an effective heat-transport coefficient, $Q_\varepsilon$ is the
incoming solar radiation, and $\beta_{\mathrm{ph}}$ is the co-albedo, that is the fraction of the incoming radiation that is absorbed. In the terminology adopted throughout this paper, we consider a
Sellers-type EBM: the physical co-albedo
$\beta_{\mathrm{ph}}$ is assumed to be continuously differentiable, as specified below. This contrasts with
Budyko-type models, where $\beta_{\mathrm{ph}}$ is modelled by a discontinuous function. The term $q$ is an additive effective radiative forcing, which describes in particular the greenhouse-gas effect in the radiation balance, see \cite{NorthKim2017,Bastiaansen2022}. Finally, $\mathfrak e\,u|u|^3$ is the outgoing long-wave radiative
flux prescribed by the Stefan-Boltzmann law, where $
\mathfrak e :=\varepsilon_{\mathrm{em}}\sigma_{\mathrm{SB}}$ is the effective Stefan-Boltzmann coefficient, obtained as the product of
the emissivity $\varepsilon_{\mathrm{em}}$ and the Stefan-Boltzmann
constant $\sigma_{\mathrm{SB}}$. In this paper, the insolation is decomposed as
$$
Q_\varepsilon(\xi) =Q_*+\varepsilon Q_1(\xi).
$$
The constant $Q_*>0$ represents a spatially homogeneous reference
insolation, whereas $Q_1$ describes its spatial variation. The parameter $\varepsilon \in \R$ controls the amplitude of the heterogeneity.

A time-independent solution of \eqref{eq:introduction-parabolic-ebm}
represents a stationary climate configuration. Such a profile solves the
elliptic equation
\begin{equation}
\div_{\mathcal M}(\kappa\nabla_{\mathcal{M}}u ) + Q_\varepsilon(\xi)\beta_{\mathrm{ph}}(u)
+q - \mathfrak{e}\,u|u|^3= 0.
\label{eq:introduction-stationary-ebm}
\end{equation}
The previous stationary equation has a natural variational structure. Let
$B$ be a primitive of a suitable extension of the physical co-albedo to $\mathbb{R}$, introduced below. Define
\begin{equation}
G_\varepsilon^{\mathcal M}(u) = \int_{\mathcal M} \left( \mathfrak{e}\frac{|u|^5}{5} - Q_\varepsilon B(u) - qu \right) \d \mu_{\mathcal M} + \frac{1}{2}
\int_{\mathcal M} \kappa| \nabla_{\mathcal M}u |^2 \d \mu_{\mathcal M},
\label{eq:introduction-ebm-energy}
\end{equation}
where $d \mu_{\mathcal{M}}$ denotes the Riemannian volume measure on $\mathcal{M}.$ The first integral is the local radiative potential, while the second
penalises spatial temperature variations and represents the energetic cost of
heat transport. The Euler-Lagrange equation associated with
\eqref{eq:introduction-ebm-energy} is precisely \eqref{eq:introduction-stationary-ebm}. Thus, stationary climate configurations are critical points of the energy functional $G_{\varepsilon}^{\mathcal{M}}$, and the energy landscape provides a natural variational classification of these equilibria. Strict local minimisers represent variationally stable climate states, while saddle points represent unstable climate states. This variational viewpoint for
one-dimensional EBMs, and its relation to the time-dependent model, is
developed in \cite{North1979variational,DelSarto2024}.

The aim of this paper is to establish the existence and variational structure
of multiple stationary solutions of
\eqref{eq:introduction-stationary-ebm}. We consider the two most common types of EBM: a one-dimensional zonally averaged model, in which
$\mathcal{M}$ is an interval parametrising latitude, and a two-dimensional
model posed on the full sphere $\mathbb{S}^2$. Particular emphasis is placed
on spatially heterogeneous EBMs, for which the insolation $Q_\varepsilon(\xi)$ depends explicitly on position. Our objective is to
give a rigorous variational analysis of the persistence, under small spatial
heterogeneities, of the classical configuration consisting of cold, warm,
and intermediate stationary climates.

A one-dimensional energy-balance model (1D EBM) is obtained by imposing zonal symmetry. The resulting
temperature therefore depends only on time and latitude. Setting $
(-1,1) \ni
x=\sin \varphi,$ where $\varphi$ denotes latitude, 1D EBM has the form
\begin{equation}
\partial_t u = \partial_x \left((1-x^2)\partial_x u\right) + Q_\varepsilon(x) \beta_{\mathrm{ph}}(u) + q - \mathfrak{e} u |u|^3, \qquad x \in(-1,1).
\label{eq:introduction-parabolic-degenerate-ebm}
\end{equation}
The coefficient $
\kappa(x):=1-x^2 $ vanishes at $x=\pm 1$, which represent the two poles. Thus the diffusion is
degenerate at the endpoints, and the polar behaviour is encoded through the
degenerate operator and its natural weighted energy space rather than through
an independently imposed classical boundary condition.

In a two-dimensional energy-balance model (2D EBM), the temperature is allowed to vary over the full
surface of the Earth. Modelling the Earth as
the unit sphere, we consider, for $\omega \in \S^2$,
\begin{equation}
\partial_t u = \div_{\S^2} (\kappa \nabla_{\S^2}u )
+ Q_\varepsilon(\omega)\beta_{\mathrm{ph}}(u) + q - \mathfrak{e} u|u|^3, \qquad \omega \in \mathbb S^2,
\label{eq:introduction-parabolic-spherical-ebm}
\end{equation}
where $
\kappa\geq k_*>0.$ Note that for constant spherical diffusivity $\kappa$, the 1D EBM \eqref{eq:introduction-parabolic-degenerate-ebm} is the zonally symmetric reduction of the 2D EBM \eqref{eq:introduction-parabolic-ebm}, obtained by assuming that the temperature and the insolation are independent of longitude.

\subsection{Related literature and contribution}
\noindent 
EBMs have been studied since the classical works
of Budyko, Sellers, and North
\cite{Budyko1969,Sellers1969,North1975,North1981}. The possibility of
multiple stationary climates, with cold and warm states separated by an
intermediate unstable state, and the corresponding bifurcation and stability
questions have been investigated in \cite{Ghil1976,North90}. A variational formulation of
Budyko-Sellers models was introduced in
\cite{North1979variational}. More recently, \cite{DelSarto2024} established,
under sufficient conditions including a large-diffusion hypothesis, the
existence of two local minimisers and a mountain-pass saddle for a spatially
heterogeneous 1D EBM with \emph{non-degenerate} diffusion.

The present paper complements these results by giving a perturbative
Morse-theoretic description of three prescribed homogeneous equilibria. For
insolation of the form
$$
Q_\varepsilon=Q_*+\varepsilon Q_1,
$$
we continue the cold, middle, and warm equilibria in the operator-domain
topology and prove that their variational character persists: the cold and
warm branches are strict local minimisers, whereas the middle branch is
non-degenerate and has Morse index one. We also establish the
Palais-Smale condition and construct a mountain-pass critical point between
the two minima. Under an additional homogeneous isolation assumption at the
minimax level, the globally constructed mountain-pass point is shown to
coincide with the locally continued middle branch. The local Morse-index
property and the global minimax characterisation are thus kept logically
distinct.

The one-dimensional part is formulated for the strongly degenerate Legendre
diffusion
$$
-\partial_x\left((1-x^2)\partial_x u\right),
$$
whose weighted functional framework has previously appeared in the analysis
of degenerate and stochastic EBMs
\cite{Diaz1997,Floridia2014,Diaz2022}. Multiplicity, S-shaped branches, and
stability have also been studied for discontinuous Budyko models, mainly by D\'iaz et al., see
\cite{Diaz1997b,Diaz1998,Diaz2019}. In contrast, the present work
uses a smooth co-albedo and combines weighted compactness, spectral
perturbation, Morse-index persistence, and mountain-pass theory.

Budyko-type 2D EBMs  on compact Riemannian manifolds and on the sphere
have been considered in
\cite{Diaz1997b,DiazTello1999,Bermejo2009} by D\'iaz and collaborators. These works address stationary multiplicity, parabolic well-posedness, and
numerical approximation in global two-dimensional models. More
recently, Sellers-type EBMs have been coupled to the primitive equations,
with global strong well-posedness and time-periodic solutions established in
\cite{DelSartoHieberZochling2026,DelSartoHieberPalma}. In contrast, the contribution
of the spherical part of the present paper is to establish the same
perturbative Morse and minimax structure as in the degenerate interval, for a uniformly elliptic diffusivity on $\S^2$.

Finally, Legendre spectral expansions are classical in zonally averaged EBMs
\cite{North1975}, and convergent Galerkin-Legendre methods for degenerate
climate equations have been analysed in \cite{Plocinicz}. Our
Legendre-Galerkin computations are intended to illustrate the deformation
of the three stationary branches, their Morse-index, and the associated energy landscape.

\subsubsection*{Notation}

\noindent
Given a Banach space $X$, we denote its norm by
$ \| \cdot\|_X$. In particular, when the underlying domain is clear from the
context, $\|\cdot\|_p$ denotes the norm in $\rL^p$. For a sequence
$(u_n)_n\subset X$, the notation $u_n \to u$ in $X$ means strong
convergence, while $u_n \rightharpoonup u$ in $X$ means weak
convergence, see \cite{Brezis}.

For the 1D EBM, setting $\mathcal{I}:=(-1,1)$, $
\kappa(x):=1-x^2,$ we introduce the weighted Sobolev space
$$
\rH^1_\kappa(\mathcal{I}) :=\left\{ u \in  \rL^2(\mathcal{I}): u \text{ is locally absolutely continuous in } \mathcal{I}, \ \sqrt{\kappa}\,u'\in \rL^2(\mathcal{I}) \right\}.
$$
It is endowed with the norm
$$
\|u\|_{1,\kappa}^2 := \| u \|_2^2 + |u|_{1,\kappa}^2, \qquad |u|_{1,\kappa}^2 := \int_{\mathcal{I}}\kappa(x)|u'(x)|^2 \d x.
$$
This is the natural energy space associated with the degenerate quadratic
form $ u \mapsto \int_{\mathcal{I}}\kappa(x)|u'(x)|^2\d x$,
 see \cite{Kufner,Floridia2014} for more details.

For the 2D EBM, we denote by $ \mathbb{S}^2 :=\{ \omega \in \R^3 : |\omega|=1\}$ the unit sphere and, for $\omega\in\S^2$, by $
T_\omega\S^2 := \{ v \in \R^3 : v \cdot\omega=0\} $ its tangent plane. The Riemannian metric $g$ induced by the Euclidean
scalar product is given by
$$
g_\omega(v,w):=v\cdot w, \qquad v, w \in T_\omega \S^2.
$$
We denote by $ \d\omega$ the surface measure on $\S^2$. In spherical coordinates $( \theta , \phi)$, where
$\theta \in (0,\pi)$ is the colatitude and
$ \phi \in(0,2\pi)$ is the longitude, it holds $
 \d\omega= \sin \theta \d \theta \d \phi. $ Moreover, a tangent vector field $X$ satisfies
$X(\omega)\in T_\omega\S^2$ for every $\omega \in \S^2$.
For $u\in \rC^\infty(\S^2)$, the tangential gradient
$\nabla_{\S^2}u$ is characterised by
$$
g_\omega(\nabla_{\S^2}u(\omega),X(\omega)) = \d u_\omega [X(\omega)]
$$
for every smooth tangent vector field $X$. The spherical divergence is
defined by the integration-by-parts identity
$$
\int_{S^2} g(X,\nabla_{\S^2}\varphi )\d \omega = - \int_{\S^2} \div_{\S^2}X\, \varphi \d \omega, \qquad \varphi\in \rC^\infty(\S^2),
$$
and the Laplace-Beltrami operator is $
\Delta_{\S^2}u :=\div _{\S^2} (\nabla_{\mathbb S^2}u).$ We refer to the monograph \cite{Abate2012,Aubin1998} for more information on differential
geometry.

The Sobolev space $\rH^1(\S^2)$ is defined as the completion of
$\rC^\infty(\S^2)$ with respect to the norm
$$
\|u\|_{\rH^1(\S^2)}^2 := \int_{\S^2}|u|^2\d \omega
+ \int_{\S^2} |\nabla_{\S^2}u|^2\d \omega.
$$
When $\kappa\in \rL^\infty(\S^2)$ and $\kappa \geq k_*>0$, as in our framework, uniform ellipticity gives
$$
k_*
\int_{\S^2} |\nabla_{\S^2}u |^2\d \omega \leq \int_{\S^2}\kappa|\nabla_{\S^2}u|^2\d \omega \leq \|\kappa\|_{\infty }
\int_{\S^2} | \nabla_{\S^2}u|^2\d \omega.
$$
Hence the weighted and unweighted Dirichlet semi-norms are equivalent, see
\cite{Aubin1998} for more details on analysis on Riemannian manifolds.

Finally, let $X$ be a real Banach space and $\mathcal{E} \in \rC^2(X;\R)$. We write
$$
D\mathcal{E}(u)[v] \quad \text {and} \quad D^2\mathcal E(u)[v,w]
$$
for the first and second variations of $\mathcal{E}$ at $u$, evaluated
in the directions $v,w\in X$. We also use $\mathcal{E}'(u) = D\mathcal{E}(u)$ when the derivative is regarded as an
element of $X^*$. A point $u_*\in X$ is critical if
$$
D\mathcal{E}(u_*)[v]=0 \qquad \forall v \in X.
$$
Suppose now that $X$ is a real Hilbert space and that $u_*$ is a critical
point of $\mathcal{E}$. The point $u_*$ is non-degenerate if
$$
D^2\mathcal{E}(u_*)[v,w]=0 \quad \forall w \in X \quad  \Rightarrow \quad v=0. $$
The Morse index of $u_*$ is
$$
\mathrm{ind}_{\mathcal E}(u_*) := \sup \{ \dim (E): E \subset X \text{ is a linear subspace, }\,
D^2\mathcal{E}(u_*)[v,v]<0 \, \forall  v \in E\setminus\{0\} \} \in \mathbb{N}_0\cup\{+\infty\}.
$$
Thus Morse index zero means that the Hessian has no strictly negative
direction, whereas Morse index one means that its maximal negative subspace
is one-dimensional. A non-degenerate critical point of Morse index one is
therefore a variational saddle point with exactly one negative direction.
By contrast, Morse index zero alone does not imply strict local minimality. In the applications below, strict local minimality follows from coercivity of
the Hessian. We refer to \cite{Chang,Jabri2003} for more details on critical points in infinite dimensions and Morse index.

\subsection{Main results}
\label{subsec: main results}
\noindent
We use the geometry parameter $ \mathcal{M} \in\{ \mathcal{I}, \S^2\}$, with $
\mathcal{I}=(-1,1),$ and write $ \d\mu_{\mathcal{I}}:=\d x, \d \mu_{\S^2}:=\d\omega.$
The corresponding energy spaces are
\begin{equation}
X_{\mathcal M} :=\begin{cases} \rH^1_\kappa(-1,1),
& \mathcal M=\mathcal{I},\\
\rH^1(\S^2), & \mathcal{M}=\S^2.
\end{cases}
\label{eq: energy spaces def}
\end{equation}
In the interval case, $\kappa(x)=1-x^2$. In the spherical case, we assume $ \kappa \in \rW^{1,\infty}(\S^2)$ and $
\kappa\geq k_*>0.$ For $u,v\in X_{\mathcal M}$, define
$$
a_{\kappa,\mathcal M}(u,v) := \begin{cases}
\int_{\mathcal{I}}
\kappa(x)u'(x)v'(x)\d x, \qquad  & \mathcal{M}=\mathcal{I},\\
\int_{\S^2} \kappa \nabla_{\mathbb{S}^2}u\cdot \nabla_{\S^2}v\d \omega, & \mathcal{M}=\S^2.
\end{cases}
$$
Let $A_{\kappa,\mathcal M}$ be the non-negative self-adjoint operator on $\rL^2(\mathcal{M}, \d \mu_{\mathcal{M}})$ associated with this form. Its first positive eigenvalue is
\begin{equation}
\lambda_1^{\kappa,\mathcal{M}} := \inf_{\substack{ v \in X_{\mathcal M}\setminus\{0\}\\
\int_{\mathcal M}v\d \mu_{\mathcal M}=0}} \frac{a_{ \kappa, \mathcal{M}}(v,v)} {\| v \|_{\rL^2(\mathcal{M})}^2} >0.
\label{eq:main-results-first-positive-eigenvalue}
\end{equation}
When the geometry is fixed, we write simply $A_\kappa$ and
$\lambda_1^\kappa$. Observe that the physical co-albedo $\beta_{\mathrm{ph}}$ is defined only for
non-negative (absolute) temperatures. However, the variational functional $G_\varepsilon^{\mathcal{M}}$ in \eqref{eq:introduction-ebm-energy} is considered
on the whole energy space $X_\mathcal{M}$ in \eqref{eq: energy spaces def}, whose elements are not a priori constrained to be
non-negative. We therefore assume that $\beta_{\mathrm{ph}}$ admits an extension
\begin{equation}
\beta \in \rC_b^1(\R), \qquad \beta|_{[0,\infty)} = \beta_{\mathrm{ph}}, \qquad 0 \leq \beta \leq 1, \qquad \beta' \geq 0,
\label{eq:extension coalbedo}
\end{equation}
and we fix one such extension. Let $B$ be a primitive of $\beta$, so that the energy functional is given by
$$
G_\varepsilon^{\mathcal{M}}(u) = \int_{\mathcal M} \left( \mathfrak{e} \frac{|u|^5}{5} - Q_\varepsilon B(u) - qu
\right)\d \mu_{\mathcal{M}} + \frac{1}{2} a_{\kappa, \mathcal{M}}(u,u).
$$
When the geometry is fixed, we omit the superscript $\mathcal{M}$ in $G_\varepsilon^\mathcal{M}$. Recalling that $Q_\varepsilon = Q_* + \varepsilon Q_1$, define the homogeneous scalar
potential
\begin{equation}
R(\theta ) := \mathfrak e\frac{|\theta |^5}{5}-Q_*B(\theta )-q\theta , \qquad \theta \in \R.
\label{eq:main-results-homogeneous-potential}
\end{equation}

Our first assumption captures the classical double-well structure of homogeneous energy-balance models, corresponding to the coexistence of cold and warm climate states separated by an intermediate transition state. This bistable picture is well established in the EBM literature and is also supported by physical and palaeoclimatic considerations; see \cite{Marotzke,Hoffmann1998,Hoffmann2002}. The curvature conditions below express this geometry variationally: the cold and warm equilibria are local minima, while the Hessian at the intermediate equilibrium has one negative constant direction and is positive on the mean-zero subspace.

\begin{asu}[Homogeneous non-degenerate three-equilibrium structure]
\label{ass:homogeneous-three-well}
There exist three numbers $ 0 < \theta_S < \theta_M<\theta_W$ such that
\begin{equation}
R'(\theta_i)=0, \qquad i \in\{S , M , W\}.
\label{eq:main-results-three-critical-points}
\end{equation}
Setting
\begin{equation}
\mu_i := R''(\theta_i) = 4\mathfrak{e} \theta_i^3- Q_*\beta'(\theta_i), \qquad i \in \{ S, M ,W\},
\label{eq:mu-i-definition-degenerate}
\end{equation}
we assume, for every geometric setting $\mathcal M$ under consideration, that
\begin{equation}
\mu_S>0, \qquad \mu_W>0, \qquad -\lambda_1^{\kappa,\mathcal M}< \mu_M < 0.
\label{eq:main-results-unified-morse-condition}
\end{equation}
\end{asu}
\noindent 
In climatic terms, the condition $\mu_M<0$ means that the constant
perturbation mode is energetically unstable: a small spatially uniform
warming or cooling decreases the energy in either direction away from the
middle equilibrium. On the other hand,
$-\lambda_1^{\kappa,\mathcal M}<\mu_M$ implies that the Hessian is positive
on the mean-zero subspace, since
$\lambda_1^{\kappa,\mathcal M}+\mu_M>0$. Thus, perturbations that
redistribute temperature without changing its spatial average increase the
energy. Hence the Hessian at the homogeneous middle equilibrium has a
one-dimensional negative subspace, corresponding to variations of the global
mean temperature.

The previous assumption is local and determines the Morse structure of the
three homogeneous equilibria. To relate the intermediate equilibrium to the
global energy barrier between the cold and warm states, we impose the
following additional assumption on the homogeneous problem.

\begin{asu}[Homogeneous mountain-pass isolation]
\label{ass:homogeneous-mountain-pass-isolation}
Let $ u_i^0\equiv\theta_i,$ for $
i\in \{S,M,W\}.$ For every geometric setting $\mathcal M$ under consideration, define $
\Gamma_0^{\mathcal M} := \{ \gamma\in \rC([0,1];X_{\mathcal M}): \gamma(0)=u_S^0,\  \gamma(1)=u_W^0 \}$ and
\begin{equation}
c_0^{\mathcal M} := \inf_{\gamma\in\Gamma_0^{\mathcal M}} \sup_{t\in[0,1]} G_0^{\mathcal M}(\gamma(t)).
\label{eq:main-results-unified-homogeneous-minimax-level}
\end{equation}
We assume that the homogeneous middle equilibrium realises this minimax
level,
\begin{equation}
c_0^{\mathcal M} = G_0^{\mathcal M}(u_M^0),
\label{eq:main-results-unified-middle-realises-level}
\end{equation}
and that it is the unique critical point at that level, that is
\begin{equation}
(G_0^{\mathcal M})'(u)=0, \qquad
G_0^{\mathcal M}(u)=c_0^{\mathcal M} \quad \Longrightarrow \quad u=u_M^0.
\label{eq:main-results-unified-level-isolation}
\end{equation}
\end{asu}
\noindent 
We can now state the main result for the one-dimensional degenerate model.

\begin{thm}
\label{thm:main-result-degenerate}
Assume \autoref{ass:homogeneous-three-well} for
$\mathcal M=\mathcal I$ and $Q_1 \in \rC([-1,1])$. Set $
Q_\varepsilon (x) = Q_* + \varepsilon Q_1 (x),$ with $ Q_* >0.$ Then there exists $\varepsilon_0>0$ such that, for every $|\varepsilon|<\varepsilon_0$, the stationary equation
$$
-(\kappa u')' + \mathfrak{e} u|u|^3 - Q_\varepsilon\beta(u) - q = 0, \qquad \kappa(x)=1-x^2,
$$
admits three solutions $
u_S^\varepsilon, u_M^\varepsilon, u_W^\varepsilon \in D(A_\kappa).$ For every $i\in\{S,M,W\}$,
$$
u_i^\varepsilon \to \theta_i \qquad \text{in }D(A_\kappa) \quad \text{as } \varepsilon \to 0.
$$
Moreover, $u_S^\varepsilon$ and $u_W^\varepsilon$ are strict local
minimisers of $G_\varepsilon$ in $\rH^1_\kappa(\mathcal{I})$, while
\(u_M^\varepsilon\) is non-degenerate and has Morse index one.

Define $ \Gamma_\varepsilon^{\mathcal I} :=\left\{ \gamma \in \rC\bigl([0,1], \rH^1_\kappa( \mathcal{I})\bigr):
\gamma(0)=u_S^\varepsilon,\  \gamma(1)=u_W^\varepsilon
\right\}$ and $ c_\varepsilon^{\mathcal I}:=
\inf_{\gamma\in\Gamma_\varepsilon^{\mathcal I}} \sup_{t\in[0,1]}
G_\varepsilon(\gamma(t)).$ Then there exists a mountain-pass critical point
$u_{MP}^\varepsilon\in \rH^1_\kappa(\mathcal{I})$ such that
$$
G_\varepsilon'(u_{MP}^\varepsilon)=0, \qquad G_\varepsilon(u_{MP}^\varepsilon) = c_\varepsilon^{\mathcal I}, \qquad  c_\varepsilon^{\mathcal I}
> \max\left\{ G_\varepsilon(u_S^\varepsilon), G_\varepsilon(u_W^\varepsilon) \right\}.
$$

If, in addition \autoref{ass:homogeneous-mountain-pass-isolation} holds for $\mathcal{M}=\mathcal{I}$, then, after possibly decreasing
$\varepsilon_0$, every mountain-pass critical point
$u_{\mathrm{MP}}^\varepsilon$ at the minimax level
$c_\varepsilon^{\mathcal I}$ coincides with the perturbative middle branch, that is $
u_{\mathrm{MP}}^\varepsilon=u_M^\varepsilon.$

\end{thm}

\begin{proof}
The existence of $u_S^{\varepsilon}, u_M^\varepsilon, u_W^\varepsilon$, their convergence, and the Morse index assertions follow from
\autoref{thm: existence three climate degenerate small eps}. The part regarding the existence of the mountain-pass point $u_{MP}^\varepsilon$ follows from
\autoref{cor:mountain-pass-between-branches-degenerate}. Under the additional
isolation assumption, the identification follows from
\autoref{prop:identification-mp-morse-degenerate}.
\end{proof}
\noindent 
We next state the corresponding result for the spherical model.

\begin{thm}
\label{thm:main-result-sphere}
Let $ \kappa\in \rW^{1,\infty}(\S^2),$ $\kappa \geq k_*>0.$ Assume \autoref{ass:homogeneous-three-well} for
$\mathcal M=\S^2$. Let $Q_1\in \rC(\S^2)$, and set $
Q_\varepsilon(\omega)=Q_*+\varepsilon Q_1(\omega),$ with  $Q_* >0$. Then there exists $\varepsilon_0>0$ such that, for every $|\varepsilon|<\varepsilon_0$, the stationary equation
$$
-\div_{\S^2} (\kappa\nabla_{\S^2}u) +\mathfrak{e} u|u|^3-Q_\varepsilon \beta(u) -q=0
$$
admits three solutions $
u_S^\varepsilon, u_M^\varepsilon, u_W^\varepsilon \in D(A_\kappa)=\rH^2(\S^2).$ For every $i\in\{S,M,W\}$,
$$
u_i^\varepsilon \to \theta_i \qquad \text{in }D(A_\kappa)=\rH^2(\S^2) \quad\text{as } \varepsilon \to 0.
$$
Moreover, $u_S^\varepsilon$ and $u_W^\varepsilon$ are strict local
minimisers of $G_\varepsilon$ in $\rH^1(\S^2)$, while
$u_M^\varepsilon$ is non-degenerate and has Morse index one.

Define $ \Gamma_\varepsilon^{\S^2} :=\{ \gamma\in \rC ([0,1]; H^1(\mathbb{S}^2) ): \gamma(0) = u_S^\varepsilon, \  \gamma(1)=u_W^\varepsilon \}$ and
$ c_\varepsilon^{\S^2} := \inf_{\gamma \in \Gamma_\varepsilon^{\S^2}}
\sup_{t\in[0,1]} G_\varepsilon(\gamma(t)).$ Then there exists a mountain-pass critical point $u_{MP}^\varepsilon\in \rH^1(\S^2)$ such that
$$
G_\varepsilon'(u_{MP}^\varepsilon)=0, \qquad G_\varepsilon(u_{MP}^\varepsilon) =c_\varepsilon^{\mathbb S^2}, \qquad  c_\varepsilon^{\S^2} > \max\{
G_\varepsilon(u_S^\varepsilon), G_\varepsilon (u_W^\varepsilon)
\}.
$$

If, in addition
\autoref{ass:homogeneous-mountain-pass-isolation} holds for $\mathcal M=\mathbb S^2$, then, after possibly decreasing $\varepsilon_0$, the mountain-pass critical point $u_{\mathrm{MP}}^\varepsilon$ provided above coincides with the
perturbative middle branch, that is$
u_{\mathrm{MP}}^\varepsilon=u_M^\varepsilon.$
\end{thm}

\begin{proof}
The persistence, convergence, and Morse index assertions follow from
\autoref{thm: three climate sphere}. The mountain-pass assertion follows from
\autoref{cor:MP point between min sphere}. Under the additional isolation
assumption, the identification follows from
\autoref{prop:identification-mp-morse-sphere}.
\end{proof}
\noindent
Finally, the numerical experiments in \autoref{sec:numerical section} complement the
analytical results for the 1D EBM. We introduce a
Legendre-Galerkin discretisation, for which the degenerate diffusion
operator is diagonal in coefficient space, and use parameter continuation in $\varepsilon$ to approximate the cold, middle, and warm stationary
branches. We then compute the discrete energies and the eigenvalues of the
discrete Hessian in order to illustrate the persistence of the Morse-index
pattern and the corresponding energy landscape. These computations thus provide
a numerical illustration of the analytical theory.

\subsubsection*{Organisation of the paper}
\noindent 
The remainder of the paper is organised as follows. In
\autoref{sec: 1D degenerate}, we study the 1D EBM with
degenerate diffusion. We prove the persistence of the cold, middle, and
warm stationary branches under small heterogeneous perturbations, establish
the Palais-Smale condition for $G_\varepsilon^\mathcal{I}$, and construct a mountain-pass critical point $u^\varepsilon_{MP}$
between the two strict local minimisers $u_S^\varepsilon, u_W^{\varepsilon}$. Under the additional isolation
assumption in \autoref{ass:homogeneous-mountain-pass-isolation}, the point $u^\varepsilon_{MP}$ is identified with the non-degenerate
Morse-index-one middle branch $u^\varepsilon_M$. In \autoref{sec: spherical model}, we develop the analogous
theory on $\S^2$, adapting the argument to the uniformly elliptic
spherical setting. Lastly, \autoref{sec:numerical section} illustrates the
one-dimensional theory through a Legendre-Galerkin discretisation, numerical
continuation of the stationary branches $u_i^\varepsilon$, for $i \in \{S, M, W \}$, and computations of the discrete
energy $G_\varepsilon^{\mathcal{I}}$ and Hessian spectrum.

\section{The one-dimensional degenerate model}
\label{sec: 1D degenerate}
\noindent 
Throughout this section, the geometry is fixed to $\mathcal M=\mathcal I=(-1,1)$. To simplify the notation, we suppress the
superscript $\mathcal I$ and write just $ G_\varepsilon,
\Gamma_\varepsilon, c_\varepsilon, a_\kappa, A_\kappa, \lambda_1^\kappa.
$
\subsection{Variational setting and preliminary properties}
\label{sec: functional definition}
Let $\beta$ the extension to $\R$ of the physical coalbedo $\beta_{\mathrm{ph}}$ satisfying \eqref{eq:extension coalbedo} Denote by $B(\theta ):= B_0 + \int_0 ^\theta \beta (r) \d r$ a primitive of $\beta$, and note that
\begin{equation}
0 \leq  \abs{B(\theta )} \leq \abs{B_0} + \abs{\theta } \quad \forall \theta .
\label{eq: inequality B(u)}
\end{equation}
We start by recalling a classical result regarding the embeddings for $H^1_\kappa (\mathcal{I})$.
\begin{lem}
For every $p\in[2,\infty)$, the embedding $
\rH^1_\kappa(\mathcal I)\hookrightarrow \rL^p(\mathcal I)$ is continuous. Moreover, the embedding $
\rH^1_\kappa(\mathcal I)\hookrightarrow \rL^2(\mathcal I)$ is compact.
\label{lemma: compact embedding weighted sobolev}
\end{lem}
\begin{proof}
   The continuous embedding into $\rL^p(\mathcal I)$, for every
$p\in[2,\infty)$, follows from \cite[Lemma 3.2]{Floridia2014}.
The compact embedding into $\rL^2(\mathcal I)$ follows from
\cite[Proposition 3.8]{Floridia2014}.
\end{proof}
\noindent 
We now recall the weighted Poincar\'e inequality, which will be used in the following
to identify the spectral gap of the degenerate diffusion operator.
\begin{lem}[Weighted Poincar\'e inequality]
\label{lem: weighted poincare}
For $\kappa (x) = 1-x^2$, define $
V:=\left \lbrace v\in \rH^1_\kappa(\mathcal{I})\; :\; \int_{\mathcal{I}} v(x)\d x = 0\right \rbrace .$ Then there exists $C>0$ such that
\begin{equation}\label{eq: weighted poincare}
\|v\|_{2}^2 \leq C \int_{\mathcal{I}} \kappa(x)\, \vert v'(x) \vert ^2\d x, \qquad\forall v\in V.
\end{equation}
\end{lem}
\begin{proof}
    First, let's observe that $V$ is a closed linear subspace of $\rH^1_\kappa (\mathcal{I})$. Indeed, it can be checked that the map $p \colon \rH^1_\kappa (\mathcal{I}) \to \mathbb{R}$, with $p(u):= \int_{\mathcal{I}} u(x) \d x $ is linear and bounded, and $V= p^{-1}(\{0\}).$

    Now, assume by contradiction that the proof of the lemma is false. Thus, there exists a sequence $(u_k)_k \subseteq V$ such that
    $$
    \norm{u_k}_2^2 > k \abs{u_k}_{1,\kappa}^2.
    $$
    Set $v_k:= u_k / \norm{u_k}_2$, and observe that: (i) $\norm{v_k}_2 = 1$, (ii) $\abs{v_k}_{1,\kappa} = \frac{\abs{u_k}_{1,\kappa}}{\norm{u_k}_2} < \frac{1}{\sqrt{k}}$, and (iii) $\norm{v_k}^2_{1,\kappa} = 1 + \abs{v_k}_{1,\kappa}^2$. From (ii) and (iii) it follows that $\norm{v_k}_{1,\kappa} \to 1$. Hence $(v_k)_k \subset V$ is a bounded sequence and thus, up to subsequences, there exists $v \in V$ such that $v_k\rightharpoonup v \quad\text{weakly in }\rH^1_\kappa$. On the other hand, by the boundedness of $(v_k)_k$ in $V$, and the compact embedding $\rH^1_\kappa \hookrightarrow \rL^2$ from \autoref{lemma: compact embedding weighted sobolev}, it also holds, up to subsequences, that $v_k \to v$ in $\rL^2$. Since $\norm{v_k}_2 =1$ for all $k$, then also $\norm{v}_2 = 1$. Further, since $\abs{v_k}_{1,\kappa}^2 = \int_{\mathcal{I}} (1-x^2) (v'_k(x))^2 \d x \to 0$, then $v'_k \to 0$ in $\rL^2(K)$, for any compact $K \subset \subset \mathcal{I}.$ Thus $v' = 0$ a.e., and hence $v \equiv c$, for some constant $c$. But $v \in V$, and thus $c=0$. This is a contradiction with the fact that $\norm{v}_2 = 1.$
\end{proof}
\noindent 
We can now apply the direct method of the calculus of variations. In particular, the
fifth-order radiative term provides coercivity, while the compact embedding
from \autoref{lemma: compact embedding weighted sobolev} allows us to pass to the limit in the lower-order terms. The monotonicity
of $B$ then permits a truncation argument showing that, for $q \geq 0$, a
global minimiser is non-negative.
\begin{prop}
\label{prop: existence minim}
Assume that $q \geq 0$ and that $Q_\varepsilon\in \rC([-1,1])$ satisfies  $Q_\varepsilon \geq 0$. Then the functional
$$
G_\varepsilon(u)=\int_{\mathcal{I}} \Big(    \mathfrak{e}\frac{\abs{u}^5}{5} - Q_\varepsilon(x)B(u)-q u\Big)\d x
+\tfrac12\int_{\mathcal{I}} \kappa (x)(u'(x))^2\d x 
$$
admits a global minimiser $u_0 = u_0 (\varepsilon,q) \in \rH^1_\kappa (\mathcal{I})$. Moreover $
u_0\in \rC^2(\mathcal{I})$ and $ u_0\geq 0$ in $\mathcal{I}.
$
\end{prop}

\begin{proof}   
Assume for simplicity that $\mathfrak{e} = 1$. We divide the proof into four steps.

\textit{Step 1: coercivity and compactness.} Since $B$, thanks to \eqref{eq: inequality B(u)}, has at most linear growth and
$Q_\varepsilon\in \rL^\infty(\mathcal{I})$, there exist constants
$c_0,C_0>0$ such that
$$
\frac{|s|^5}{5} - Q_\varepsilon(x)B(s) - qs \geq c_0 |s|^5 - C_0
$$
for every $s \in\R$ and every $x\in[-1,1]$. Thus
\begin{equation}
G_\varepsilon(u) \geq \frac{1}{2} \| \sqrt{\kappa} u' \|_{2}^2 + c_0 \| u \|_{5}^5 - 2C_0.
\label{eq: coercivity energy minimiser}
\end{equation}
Since $\rL^5(\mathcal{I})\hookrightarrow \rL^2(\mathcal{I})$, it follows that
$G_\varepsilon$ is coercive on $\rH^1_\kappa(\mathcal{I})$.

Let $(u_n)_n\subset \rH^1_\kappa(\mathcal{I})$ be a minimising sequence. By
\eqref{eq: coercivity energy minimiser}, the sequence is bounded in
$\rH^1_\kappa(\mathcal{I})$ and in $\rL^5(\mathcal{I})$. Hence, after passing to a
subsequence, there exists $u_0\in \rH^1_\kappa(\mathcal{I})$ such that
$$
u_n \rightharpoonup u_0 \qquad \text{in }\rH^1_\kappa(\mathcal{I}).
$$
By the compact embedding in
\autoref{lemma: compact embedding weighted sobolev}, up to a subsequence,
$$
u_n\to u_0 \qquad \text{in }\rL^2(\mathcal{I})
$$
and, after extracting a further subsequence,
$$
u_n(x)\to u_0(x) \qquad \text{for a.e. }x\in \mathcal{I}.
$$

\textit{Step 2: lower semicontinuity.} By Fatou's lemma,
$$
\int_{\mathcal{I}} \frac{ |u_0|^5}{5}\d x
\leq \liminf_{n\to\infty} \int_{\mathcal{I}} \frac{|u_n|^5}{5}\d x.
$$
Since $0 \leq \beta \leq1$, the primitive $B$ is globally Lipschitz, and
therefore
$$
\|B(u_n)-B(u_0)\|_{1} \leq \| u_n-u_0\|_{1}
\to 0.
$$
Using $Q_\varepsilon\in \rL^\infty(\mathcal{I})$, we obtain
$$
\int_{\mathcal{I}} Q_\varepsilon(x)B(u_n)\d x
\to \int_{\mathcal{I}} Q_\varepsilon(x)B(u_0) \d x.
$$
Similarly, it is possible to show that
$
\int_{\mathcal{I}}u_n\d x \to \int_{\mathcal{I}}u_0\d x.$ Finally, weak lower semicontinuity gives
$$
\int_{\mathcal{I}} \kappa(x) |u_0'(x)|^2 \d x
\leq \liminf_{n\to\infty} \int_{\mathcal{I}} \kappa(x) |u_n'(x)|^2\d x.
$$
Combining these convergences, we obtain
$$
G_\varepsilon(u_0) \leq \liminf_{n\to\infty}G_\varepsilon(u_n) = \inf_{\rH^1_\kappa(\mathcal{I}) }G_\varepsilon.
$$
Thus $u_0$ is a global minimiser.

\textit{Step $3$: regularity for $u_0$.} The first variation of $G_\varepsilon$ in the point $u$ in direction $h$ is given by
$$
D G_\varepsilon(u)[h] = \int_{\mathcal{I}}
\left( u \abs{u}^3 - Q_\varepsilon(x) \beta(u)-q  \right )h \d x + \int_{\mathcal{I}} \kappa(x)u'h'\d x.
$$
Choosing $h \in \rC^\infty_c(\mathcal{I})$ and setting $\phi(t):= G_\varepsilon(u_0 + th)$, it holds $\phi'(0)=0$. So
$$
0=\phi'(0)=\delta G_\varepsilon(u_0,h),
$$
from which it follows
\begin{equation}
\int_{\mathcal{I}} \left(u_0 \abs{u_0}^3- Q_\varepsilon(x)\beta(u_0)- q \right)h \d x = - \int_{\mathcal{I}} \kappa(x) u_0 ' h ' \d x.
\label{eq: 1}
\end{equation}
Then
\begin{equation}
(\kappa(x) u_0 ')'= u_0 \abs{u_0}^3- Q_\varepsilon(x)\beta(u_0)- q
\label{eq: ELE debole}
\end{equation}
in the weak sense. We now prove that the right-hand side is continuous in the interior. Let $K\Subset \mathcal{I}$. Since $\kappa$ is strictly positive on $K$ and
$u_0 \in \rH^1_\kappa(\mathcal{I})$, we have $u_0\in \rH^1(K)$. By the one-dimensional
Sobolev embedding, it holds
$
u_0\in \rC^0(K).
$
Since $Q_\varepsilon\in \rC([-1,1])$ and $\beta $ is continuous, it follows that
$$
f(x):=
u_0(x)\abs{u_0(x)}^3-Q_\varepsilon(x) \beta(u_0(x))-q \in \rC^0(K).
$$
By \eqref{eq: ELE debole}, $(\kappa u_0')'=f$ in the sense of distributions
on $K$. Since $f\in \rC^0(K)$, we get $\kappa u_0'\in \rC^1(K).$
Because $\kappa(x)>0$ on $K$ and $\kappa$ is smooth, this implies $u_0'\in \rC^1(K),$
and hence $ u_0\in \rC^2(K).$ Since $K\Subset \mathcal{I}$ was arbitrary, we conclude that $ u_0\in \rC^2(\mathcal{I}).$

\textit{Step 4: $u_0 \geq 0$.}
We use a global truncation argument. Define
$$
u_0^+(x):=\max\{u_0(x),0\}.
$$
Since the map $s\mapsto s^+$ is Lipschitz, we have
$u_0^+\in \rH^1_\kappa(\mathcal{I})$ and
$$
(u_0^+)'=\mathds{1}_{\{u_0>0\}}u_0'
\qquad\text{a.e.}
$$
Therefore
$$
\int_{\mathcal{I}}\kappa(x)\bigl((u_0^+)'\bigr)^2\d x
\leq
\int_{\mathcal{I}}\kappa(x)(u_0')^2\d x.
$$
Let$
E:=\{x\in \mathcal{I}: u_0(x)<0\}.
$
On $E$ one has $u_0^+=0$. Hence
$$
G_\varepsilon(u_0)-G_\varepsilon(u_0^+)
= \int_E \left[
\frac{\abs{u_0}^5}{5} -Q_\varepsilon(x)B(u_0) -q u_0 +Q_\varepsilon(x)B(0) \right]\d x 
+\frac{1}{2} \int_E\kappa(x)(u_0')^2\d x .
$$
Since $\beta = B' \geq 0$, the function $B$ is nondecreasing. Therefore,
on $E$ we have
$$
B(0)-B (u_0 ( x)) \geq 0.
$$
Moreover, since $q\geq 0$ and $u_0<0$ on $E$, one has $-q u_0 \geq 0$ on $E$. Moreover, $\vert u_0 \vert^5/5 >0$ a.e.\ on $E$. Consequently, if $\vert E \vert >0$, then 
$$
G_ \varepsilon( u_0 ) - G_\varepsilon(u_0^+)
> 0,
$$
contradicting the minimality of $u_0$.

\end{proof}
\noindent
The preceding truncation argument excludes negative values of a minimiser. A
second truncation, now from above, gives a uniform upper bound.
\begin{lem}
\label{lem: C_0 boundedness minimiser}
Let $b>0$ and $\varepsilon_*>0$, and assume that $
\overline Q := \sup_{ |\varepsilon| \leq \varepsilon_*} \| Q_\varepsilon \|_{\infty } < \infty.$ Then there exists $M=M(b,\overline{Q})>0$ such that, for every $q \in[0,b]$, every $|\varepsilon| \leq \varepsilon_*$, and every global
minimiser $u_0=u_0(\varepsilon,q)$ of $G_\varepsilon$, it holds
$$
0 \leq u_0 \leq M \qquad\text{a.e. in }(-1,1).
$$
\end{lem}
\begin{proof}
Assume for simplicity that $\mathfrak{e} = 1$. Fix $q \in [0,b]$ and $\varepsilon$ such that $\vert \varepsilon \vert \leq \varepsilon_*$. By \autoref{prop: existence minim}, $u_0 \geq 0$ a.e. in $\mathcal{I}$. Define
$$
\mathcal{R}(x,s) := \mathfrak{e}\frac{s^5}{5}-Q_\varepsilon(x) B(s)-q s,
\qquad s \geq 0.
$$
Since $0 \leq \beta \leq 1$ and $0<q<b$, we have, for $s \geq 0$,
$$
\partial_s\mathcal{R}(x,s) = s^4-Q_\varepsilon(x) \beta(s)-q \geq s^4- \| Q_\varepsilon \|_{\infty }-b \geq s^4 - \overline{Q} - b.
$$
Choose $M>0$ so large that
$$
M^4> \overline Q +b.
$$
Then
$$
\partial_s  \mathcal{R}(x,s)>0    \qquad \text{for every } s \geq M  \text{ and every } x \in[-1,1].
$$
Thus $\mathcal{R}(x,\cdot)$ is strictly increasing on $[M,\infty)$, uniformly
in $x$.

Define the truncation
$$
T_M(s):=\min\{s,M\}, \qquad T_M(u_0)(x):=T_M(u_0(x)).
$$
Since $T_M$ is Lipschitz, $T_M(u_0)\in H^1_\kappa(-1,1)$ and
$$
(T_M(u_0))' = \mathds 1_{\{u_0<M\}}u_0' \qquad \text{a.e.}
$$
Therefore
$$
\int_{\mathcal{I}}\kappa(x)((T_M(u_0))')^2\,dx \leq \int_{\mathcal{I}}\kappa(x)(u_0')^2\,dx.
$$
Let
$
E_M:=\{x\in \mathcal{I}:u_0(x)>M\}.
$
On $E_M$ one has $T_M(u_0)=M$. Since $\mathcal{R}(x,\cdot)$ is strictly
increasing on $[M,\infty)$, we obtain
$$
\mathcal{R}(x,u_0(x))>\mathcal{R}(x,M) = \mathcal R(x,T_M(u_0)(x)) \qquad    \text{for a.e. } x \in E_M.
$$
Outside $E_M$, $T_M(u_0)=u_0$. Hence $
G_\varepsilon(T_M(u_0))\leq  G_\varepsilon(u_0).$ Moreover, if $|E_M|>0$, then the inequality is strict
$$
G_\varepsilon(T_M(u_0))<G_\varepsilon(u_0),
$$
contradicting the minimality of $u_0$. Therefore $|E_M|=0$, and thus we can conclude that $
u_0\leq M$ a.e. in $\mathcal{I}.$
\end{proof}

\subsection{Three-climate structure in the degenerate model}
We now turn from the existence of a single global minimiser to the
variational three-climate structure. Our analysis has three components.
First, we establish the Palais-Smale compactness condition required for the
mountain-pass construction. Second, in the spatially homogeneous problem we
identify two strict local minimisers and one non-degenerate critical point of
Morse index one. Finally, we show that these three critical points persist
under sufficiently small heterogeneous perturbations of the insolation.

\subsubsection{Palais-Smale compactness}
The existence of a critical point separating two local minima is obtained by
a mountain-pass argument. The essential compactness input is the Palais-Smale condition, which ensures that approximate critical points with bounded energy have a convergent subsequence. We refer to the monograph
\cite{Jabri2003} for more details on the mountain pass theory.

Let $(X, \norm{\cdot})$ be a reflexive Banach space, $ \Phi \in \rC^1( X; \mathbb{R})$ be a functional and $\Phi '$ denote the first variation of $\Phi$.
\begin{defn}
The functional $\Phi$ satisfies the \emph{Palais-Smale condition} ((PS)-condition) if any sequence $\{u_n\}_n \subset X$ such that
$$
\Phi(u_n) \text{ is bounded and } \Phi ' (u_n) \rightarrow 0 \quad \text{in }X^*,
$$
admits a convergent subsequence.
\end{defn}
The previous one is a compactness condition needed in order to use the Mountain pass theorem, which loosely speaking affirms the existence of a mountain pass between two valleys. See \cite{Jabri2003} for more details.
\begin{thm}[Mountain pass]
\label{thm: mountain pass}
Assume that $\Phi$ satisfies the $(PS)$-condition, that $\Phi(0)=0$, and that
there exist $\rho,\alpha>0$ such that
$$
\Phi(x) \geq  \alpha \qquad \text{for every }x \in X \text{ with }\|x\|=\rho.
$$
Assume moreover that there exists $x_1\in X$ such that
$$
\|x_1\|>\rho, \qquad
\Phi(x_1)\leq 0.
$$
Let $ \Gamma := \left\lbrace  \gamma\in \rC([0,1],X):
\gamma(0)=0,\ \gamma(1)=x_1 \right\rbrace ,$ and define $
c := \inf_{\gamma\in\Gamma} \max_{t\in[0,1]}\Phi(\gamma(t)).$ Then $c \geq \alpha$ and there exists $x_2\in X$ such that
$$
\Phi'(x_2)=0, \qquad \Phi(x_2)=c.
$$
\end{thm}
\noindent 
For the applications below, it is convenient to reformulate the theorem for
paths connecting two strict local minimisers. 

\begin{cor}[Mountain pass between two local minimisers]
\label{cor:two-local-minima-mpt}
Let $X$ be a reflexive Banach space and let $\Phi\in C^1(X,\mathbb R)$ satisfy
the PS-condition. Assume that $u_0, u_1\in X$ are two distinct local
minimisers of $\Phi$ with a positive energy barrier in the following sense:
there exist $r_0,r_1>0$ and $\alpha_0,\alpha_1>0$ such that
$$
B_{r_0}(u_0)\cap B_{r_1}(u_1)=\emptyset,
$$
and
$$
\Phi(u) \geq \Phi(u_0)+\alpha_0 \qquad \text{for every }u \in\partial B_{r_0}(u_0),
$$
while
$$
\Phi(u)\geq \Phi(u_1)+\alpha_1 \qquad \text{for every } u \in \partial B_{r_1}(u_1).
$$
Let $ \Gamma := \left\lbrace  \gamma\in \rC([0,1],X): \gamma(0)=u_0,\ \gamma(1)=u_1 \right \rbrace,$ and define $ c := \inf_{\gamma\in\Gamma} \max_{t\in[0,1]}\Phi(\gamma(t)).$ Then $
c> \max\{\Phi(u_0),\Phi(u_1)\},$ and there exists $u_{MP}\in X$ such that
$$
\Phi'(u_{MP})=0, \qquad \Phi(u_{MP})=c.
$$
In particular, $u_{MP}\notin\{u_0,u_1\}$.
\end{cor} 
\begin{proof}
First, every path $\gamma\in\Gamma$ must leave $B_{r_0}(u_0)$ and enter
$B_{r_1}(u_1)$. Therefore it crosses both boundaries
$\partial B_{r_0}(u_0)$ and $\partial B_{r_1}(u_1)$. Hence
$$
\max_{t\in[0,1]}\Phi(\gamma(t)) \geq \max \{ \Phi(u_0)+\alpha_0,\, \Phi(u_1)+\alpha_1\}.
$$
Taking the infimum over $\gamma\in\Gamma$ gives $ c> \max\{\Phi(u_0),\Phi(u_1)\}.$

It remains to prove that $c$ is a critical level. Without loss of generality, we can assume $
\Phi(u_1)\geq \Phi(u_0).$ Set $
\Psi(v):=\Phi(u_1+v)-\Phi(u_1),$ for any $ v\in X.$ Since $\Phi$ satisfies the PS-condition, the same holds for $\Psi$. Moreover,
$\Psi(0)=0$. By the barrier assumption at $u_1$, we have
$$
\Psi(v)\geq \alpha_1 \qquad \text{for every }\|v\|=r_1.
$$
Since $B_{r_1}(u_1)$ does not contain $u_0$, we have $
\|u_0-u_1\|>r_1.$ Furthermore,
$$
\Psi(u_0-u_1) = \Phi(u_0)-\Phi(u_1) \leq 0.
$$
Thus the classical Mountain Pass Theorem (i.e. \autoref{thm: mountain pass}) applies to $\Psi$ with endpoint
$u_0-u_1$. It gives a critical point $v_{MP}\in X$ of $\Psi$ at the level $
\inf_{\eta} \max_{t\in[0,1]}\Psi(\eta(t)),$ where the infimum is taken over all continuous paths $\eta$ joining $0$ to
$u_0-u_1$.

Setting $ u_{MP}:=u_1+v_{MP},$ we obtain $ \Phi'(u_{MP})=0.$ Moreover, paths $\eta$ joining $0$ to $u_0-u_1$ are in one-to-one correspondence with paths $\gamma(t)=u_1+\eta(t)$ connecting $u_1$ to $u_0$.
Since reversing a path does not change its maximum energy, the corresponding
minimax level is exactly $c-\Phi(u_1)$. Hence $
\Phi(u_{MP})=c.$ Finally, because $
c>\max\{\Phi(u_0),\Phi(u_1)\},$ the point $u_{MP}$ is distinct from both local minimisers.
\end{proof}
\noindent
We now verify that the energy functional $G_\varepsilon$ of the degenerate one-dimensional
model satisfies the Palais-Smale condition.
\begin{prop}\label{prop:PS-degenerate}
Assume that $Q_\varepsilon\in \rL^\infty (\mathcal{I})$, $\kappa(x)=1-x^2$. Let $\beta$ and its primitive $B$ satisfy \eqref{eq:extension coalbedo} and \eqref{eq: inequality B(u)} respectively. Then the functional
$$
 G_\varepsilon(u)=\int_{\mathcal{I}}\left(    \mathfrak{e}\frac{|u|^5}{5}-Q_\varepsilon(x)B(u)-q u\right)\d x +\frac{1}{2}\int_{\mathcal{I}}\kappa(x)(u')^2\d x
$$
satisfies the Palais-Smale ((PS)) condition on $\rH^1_\kappa(\mathcal{I})$.
\end{prop}
\begin{proof}
Assume for simplicity that $\mathfrak{e} = 1$. Let $(u_n)_n\subset \rH^1_\kappa(\mathcal{I})$ be a Palais-Smale sequence, i.e.
$$
G_\varepsilon(u_n)\ \text{is bounded},\qquad  G_\varepsilon'(u_n)\to0\ \text{in }(\rH^1_\kappa)^*.
$$

\medskip\noindent\textit{Step 1. Uniform bounds.}
From the growth bound on $B$ there exist constants $C_0,C_1>0$ such that for every $x\in[-1,1]$ and $s\in\mathbb R$
\begin{equation}
    \frac{|s|^5}{5}-Q_\varepsilon(x)B(s)-q s \ge -C_0 - C_1|s|.
\label{eq: growth integrand}
\end{equation}
Moreover, since $\frac{|s|^5}{5}$ dominates the linear term for large $|s|$, there exist constants $C_2,C_3>0$ such that
\begin{equation}
\frac{|s|^5}{5} - Q_\varepsilon(x)B(s) - q s \geq C_2|s|^5-C_3
\qquad \forall s \in \mathbb{R},\ \forall x\in[-1,1].
\label{eq:coercive-integrand-bound-degenerate}
\end{equation}
Thus,
\begin{equation}
G_\varepsilon(u) \geq \frac{1}{2} \int_{\mathcal{I}}\kappa(x)(u'(x))^2\d x + C_2 \| u \|_{ 5}^5
- 2 C_3 \qquad \forall u \in \rH^1_\kappa(\mathcal{I}).
\label{eq: coercive energy bound degenerate}
\end{equation}
Integrating and using that $G_\varepsilon(u_n)$ is bounded from above yields a
uniform $L^5$-bound
\begin{equation}\label{eq: L5 bound}
\|u_n\|_{5}\le M
\end{equation}
for some $M>0$ independent of $n$.

Next, write $u_n=m_n+v_n$ with $m_n=\tfrac12\int_{\mathcal{I}} u_n$ and $\int_{\mathcal{I}} v_n=0$. By \autoref{lem: weighted poincare} (weighted Poincar\'e)
$$
\|v_n\|_{2}^2 \le C\int_{\mathcal{I}}\kappa(x)(v_n')^2\d x \le C\int_{\mathcal{I}}\kappa(x)(u_n')^2\d x.
$$
Hence
$$
\|u_n\|_{2}^2 \le 2\|v_n\|_{2}^2 + 2m_n^2 \leq 2C\|\sqrt\kappa\,u_n'\|_{2}^2 + 2m_n^2.
$$
The uniform $\rL^5$-bound \eqref{eq: L5 bound} implies that $|m_n|$ is uniformly bounded (because $\|u_n\|_{1}\le C\|u_n\|_{5}$ in 1D). Combining this with the coercive gradient term in $G_\varepsilon$ (see \eqref{eq: growth integrand}) yields a quadratic-type inequality in $\|u_n\|_{1,\kappa}$ and therefore the sequence $(u_n)$ is bounded in $\rH^1_\kappa(\mathcal{I})$. Consequently (up to a subsequence)
\begin{equation}\label{ weak conv}
u_n \rightharpoonup u \quad\text{in }\rH^1_\kappa(\mathcal{I}),\qquad
u_n\to u \ \text{in }\rL^2(\mathcal{I})\text{ and a.e. in } \mathcal{I}.
\end{equation}
Note that here we have used the compact embedding $\rH^1_\kappa\hookrightarrow \rL^2$ from \autoref{lemma: compact embedding weighted sobolev}.

\textit{Step 2. Identification of the derivative.}
The G\^ateaux derivative of $G_\varepsilon$ is given, for $w,v\in \rH^1_\kappa(\mathcal{I})$, by
$$
\langle  G_\varepsilon'(w),v\rangle
= \int_{\mathcal{I}}\left(
|w|^3w-Q_\varepsilon(x)\beta(w)-q
\right) v\d x
+
\int_{\mathcal{I}}\kappa(x)w'v'\d x.
$$
Setting $
A(w):= -(\kappa w')' + |w|^3w$ in weak form, and $
L(w):=Q_\varepsilon(x)\beta(w)+q,$ the Palais-Smale condition says
$$
A(u_n)-L(u_n)\to0 \qquad \text{in }(\rH^1_\kappa)^*.
$$

\textit{Step 3. Monotonicity argument.} We prove that $u_n\to u$ strongly in $\rH^1_\kappa(\mathcal{I})$. Since $(u_n)_n$ is
bounded in $\rH^1_\kappa(\mathcal{I})$ (see the end of Step $1$) and in $\rL^5(\mathcal{I})$ thanks to \eqref{eq: L5 bound}, and since
$u_n\to u$ a.e. thanks to \eqref{ weak conv}, we have, up to a subsequence, that
$$
u_n\rightharpoonup u \qquad
\text{in }\rL^5(\mathcal{I}).
$$
We test the difference $A(u_n)-A(u)$ against $u_n-u$. We have
$$
\langle A(u_n)-A(u),u_n-u\rangle = \int_{\mathcal{I}}\kappa(x)(u_n'-u')^2\d x  
+ \int_{\mathcal{I}} (|u_n|^3u_n-|u|^3u ) (u_n - u) \d x.
$$
The second integral is non-negative, because the map
$ s \mapsto |s|^3s$ is monotone increasing on $\mathbb{R}$. Hence
\begin{equation}
\langle A(u_n)-A(u),u_n-u\rangle \geq
\int_{\mathcal{I}}\kappa(x)(u_n'-u')^2\d x.
\label{eq:monotonicity-controls-gradient-degenerate}
\end{equation}

We now show that the left-hand side tends to zero. Since
$G_\varepsilon'(u_n)=A(u_n)-L(u_n) \to 0$ in $(\rH^1_\kappa)^*$ and
$(u_n-u)_n$ is bounded in $\rH^1_\kappa(\mathcal{I})$, we have
$$
\langle A(u_n)-L(u_n),u_n-u\rangle=o(1).
$$
Therefore
\begin{equation*}
\begin{split}
\langle A(u_n)-A(u),u_n-u\rangle
&= \langle A(u_n)-L(u_n),u_n-u\rangle  +
\int_{\mathcal{I}} \left(Q_\varepsilon(x) \beta(u_n)+q\right)(u_n-u)\d x  \\
&\quad
- \langle A(u),u_n-u\rangle .
\end{split}
\end{equation*}
The first term is $o(1)$. The second term tends to zero because
$Q_\varepsilon \beta(u_n)+q$ is bounded in $\rL^\infty(\mathcal{I})$ and
$u_n\to u$ strongly in $\rL^1(\mathcal{I})$. Finally,
$$
\langle A(u),u_n-u\rangle
= \int_{\mathcal{I}}\kappa(x)u'(u_n'-u')\d x
+ \int_{\mathcal{I}} |u|^3u\,(u_n-u)\d x
\to 0.
$$
Indeed, the first term tends to zero by the weak convergence in
$\rH^1_\kappa(\mathcal{I})$, while the second tends to zero because
$|u|^3u\in \rL^{5/4}(\mathcal{I})$ and $u_n\rightharpoonup u$ weakly in
$\rL^5(\mathcal{I})$.

Consequently,
$$
\langle A(u_n)-A(u),u_n -u \rangle \to 0.
$$
By \eqref{eq:monotonicity-controls-gradient-degenerate},
$$
\int_{\mathcal{I}}\kappa(x)(u_n'-u')^2\d x\to0.
$$
Together with the strong convergence $u_n \to u$ in
$\rL^2(\mathcal{I})$ from \eqref{ weak conv}, this gives
$$
u_n \to u \qquad \text{strongly in } \rH^1_\kappa(\mathcal{I}).
$$
Hence every Palais-Smale sequence admits a strongly convergent subsequence. This concludes the proof.

\end{proof}
\subsubsection{Perturbative saddle and mountain-pass critical point}

The Palais-Smale result in \autoref{prop:PS-degenerate} applies to spatially heterogeneous insolation.
To construct the reference three-climate configuration, however, we begin
with a homogeneous insolation $Q_*>0$. In this case constant functions are
natural candidates for equilibria, and their variational character is
completely determined by a scalar potential. More precisely, a constant function $u \equiv \theta$ is a critical point of
the homogeneous energy if and only if
$$
\mathfrak{e}\theta|\theta|^3-Q_*\beta(\theta)-q=0.
$$

We recall that we denote by $A_\kappa$ the self-adjoint operator in $\rL^2(\mathcal{I})$ associated
with the closed quadratic form
$$
a_\kappa(u,v)=
\int_{\mathcal{I}}\kappa(x)u'(x)v'(x)\d x,
\qquad
D(a_\kappa)=\rH^1_\kappa(\mathcal{I}).
$$
Equivalently, $A_\kappa u=-(\kappa u')'$ in the weak sense, with the natural
degenerate no-flux condition encoded by the form domain. Since the embedding
$\rH^1_\kappa(\mathcal{I})\hookrightarrow \rL^2(\mathcal{I})$ is compact, $A_\kappa$ has compact
resolvent. Its first eigenvalue is $0$, corresponding to the constant
functions. Moreover, $\lambda_1^\kappa$ denotes the first positive eigenvalue,
namely
\begin{equation*}
\lambda_1^\kappa := \inf \left\{ \frac{\int_{\mathcal{I}}\kappa(x)(v'(x))^2\d x}
{\int_{\mathcal{I}}v(x)^2\d x} : v \in \rH^1_\kappa(\mathcal{I}),\ v\not\equiv0,\
\int_{\mathcal{I}}v(x)\d x=0 \right\}.
\end{equation*}
By the weighted Poincare inequality, $\lambda_1^\kappa>0$.

The next result determines the variational character and Morse indices of the
three homogeneous equilibria.
\begin{lem}[Homogeneous equilibria and Morse indices in 1D]
\label{lem:homogeneous-equilibria-morse-degenerate}
Assume \autoref{ass:homogeneous-three-well}. Then the constant functions
$$
u_S^0 \equiv \theta_S,\qquad u_M^0 \equiv \theta_M, \qquad u_W^0 \equiv \theta_W
$$
are critical points of $G_{0}$. Moreover, $u_S^0$ and $u_W^0$ are
strict local minimisers, while $u_M^0$ is non-degenerate and has Morse index
one.
\end{lem}
\begin{proof}
Assume for simplicity that $\mathfrak{e } = 1.$ Since $R'(\theta_i)=0$, we have
$$
\theta_i \abs{\theta_i}^3 - Q_* \beta( \theta_i) - q = 0.
$$
Because the functions $u_i^0 \equiv \theta_i$ are constant, their weighted
gradients vanish. Hence each $u_i^0$ satisfies the stationary equation
$$
-(\kappa (u_i^0)')' + u_i^0\abs{u_i^0}^3 - Q_*\beta(u_i^0) - q =0
$$
in the weak sense. Thus each $u_i^0$ is a critical point of
$G_{0}$.

The second variation at $u_i^0$ is
$$
D^2G_{0}(u_i^0)[v,v] =\int_{\mathcal{I}} \kappa ( x ) (v'( x ) ) ^2\d x + \mu_i \int_{\mathcal{I}} v( x ) ^2\d x.
$$
If $i=S$ or $i=W$, then $\mu_i>0$. Hence
$$
D^2G_{0}(u_i^0)[v,v] \geq c_i\|v\|_{\rH^1_\kappa(\mathcal{I})}^2
$$
for some $c_i>0$. Therefore $u_S^0$ and $u_W^0$ are strict local minimisers.

It remains to analyse $u_M^0$. Consider the decomposition 
$$
v= \bar{v}+v_0, \qquad \bar{v}: = \frac{1}{2} \int_{\mathcal{I}} v(x)\d x, \qquad \int_{\mathcal{I}} v_0(x) \d x=0.
$$
On the constant direction $v \equiv1$, the second variation is negative,
because
$$
D^2 G_{0}(u_M^0)[1,1] = 2 \mu_M < 0.
$$
On the other hand, if $\int_{\mathcal{I}} v(x)\d x=0$, then by the definition of
$\lambda_1^\kappa$,
$$
\int_{\mathcal{I}}\kappa(x )( v' ( x ) )^2\d x \geq \lambda_1^\kappa \int_{\mathcal{I}} v ( x ) ^2 \d x.$$
Therefore
$$
D^2  G_{0} ( u_M^0) [v,v] \geq (\lambda_1^\kappa + \mu_M) \int_{\mathcal{I}} v( x )^2\d x.
$$
Since $-\lambda_1^\kappa < \mu_M < 0$, the right-hand side is strictly positive
for every non-zero mean-zero $v$. Hence the negative spectral subspace of the quadratic form is one-dimensional
and is generated by the constant functions (in other words, the quadratic form
has exactly one negative eigenvalue). It follows that $u_M^0$ has Morse index
one.

Lastly, to prove non-degeneracy, we consider the linearisation of the Euler-Lagrange
operator at $u_M^0$. This linearised operator is
$$
L_M v:=-(\kappa v')' + \mu_M v .
$$
Let $v\in \rH^1_\kappa(\mathcal{I})$ belong to its kernel in the weak sense, which means that
$$
\int_{\mathcal{I}} \kappa  v' \phi'\d x + \mu_M \int_{\mathcal{I}} v\phi\d x =0 \qquad
\text{for every } \phi \in \rH^1_\kappa(\mathcal{I}).
$$
Taking $\phi \equiv1 $, and using $\mu_M\neq0$, gives
$$
\int_{\mathcal{I}} v \d x = 0.
$$
Thus $v$ has zero mean. Taking now $\phi=v$, and using the weighted
Poincar\'e inequality, we obtain
$$
0 = \int_{\mathcal{I}} \kappa |v'|^2\d x + \mu_M \int_{\mathcal{I}} v^2\d x \geq (\lambda_1^\kappa + \mu_M ) \int_{\mathcal{I}} v^2\d x.
$$
Since $\lambda_1^\kappa + \mu_M >0 $, it follows that $v=0$. Hence
$\ker{L}_M=\{0\}$, and therefore $u_M^0$ is non-degenerate.
\end{proof}
\noindent
The following estimate provides the pointwise control needed for the implicit
function theorem.

\begin{lem}[Boundedness of the operator domain]
\label{lem:domain-embeds-Linfty-degenerate}
Let $\kappa(x)=1-x^2$, and let $A_\kappa$ be the self-adjoint operator in
$\rL^2(\mathcal{I})$ associated with the closed form
$$
a_\kappa(u,v) = \int_{\mathcal{I}}\kappa(x)u'(x)v'(x)\d x, \qquad D(a_\kappa)=\rH^1_\kappa(\mathcal{I}).
$$
Then $
D(A_\kappa)\hookrightarrow \rL^\infty(\mathcal{I})$ continuously. More precisely, there exists a constant $C>0$ such that
$$
\|u\|_{\infty } \leq C\left( \|u\|_{2}+\|A_\kappa u\|_{2}
\right) \qquad \forall u \in D(A_\kappa).
$$
\end{lem}

\begin{proof}
Let $u\in D(A_\kappa)$ and set $f:=A_\kappa u$. By the definition of the
operator associated with the form,
$$
\int_{\mathcal{I}}\kappa(x)u'(x)v'(x)\d x = \int_{\mathcal{I}} f(x)v(x)\d x \qquad \forall v\in \rH^1_\kappa(\mathcal{I}).
$$
Thus, in the sense of distributions on $\mathcal{I}$,
\begin{equation}
    -(\kappa u')'=f.
    \label{eq:identity k u prime and f}
\end{equation}
In particular, $\kappa u'\in \rH^1_{\mathrm{loc}}(\mathcal{I})$. Since $\kappa u'\in \rL^2(\mathcal{I})$ and $(\kappa u')'=-f\in \rL^2(\mathcal{I})$ in the
sense of distributions, we have $\kappa u'\in \rH^1(\mathcal{I})$. Testing the weak
identity against smooth functions with arbitrary endpoint values gives the
natural no-flux traces
$$
(\kappa u')(-1)=(\kappa u')(1)=0.
$$
Set $ y(x):=\kappa(x)u'(x).$ Since $y(1)=0$, for $x\in(0,1)$ we have, using Cauchy-Schwarz inequality and \eqref{eq:identity k u prime and f},
$$
|y(x)| = \left|\int_x^1 y'(s)\d s\right| \leq (1-x)^{1/2}\|y'\|_{2} = (1-x)^{1/2}\|f\|_{2}.
$$
Because $\kappa(x)=(1-x)(1+x)$, this gives, for $x\in(0,1)$,
$$
|u'(x)| = \frac{|y(x)|}{\kappa(x)} \leq \frac{C\|f\|_{2}}{(1-x)^{1/2}}.
$$
The right-hand side is integrable near $x=1$. The same argument at the left
endpoint, using $y(-1)=0$, gives
$$
|u'(x)|\leq \frac{C\|f\|_{\rL^2(\mathcal{I})}}{(1+x)^{1/2}} \qquad \text{for }x\in(-1,0).
$$
Hence $u'\in \rL^1(\mathcal{I})$ and $ \|u'\|_{1} \leq C\|A_\kappa u\|_{2}.$ Choose $x_0\in \mathcal{I}$ such that $ |u(x_0)|\leq C\|u\|_{2}.$ Then, for every $x\in \mathcal{I}$,
$$
\vert  u (x) \vert  \leq \abs{u(x_0)}+\int_{\mathcal{I}} \vert u'(s) \vert \d s \leq C\left( \| u  \|_{2} + \| A_\kappa u\|_{2} \right).
$$
This concludes the proof.
\end{proof}
\noindent
Having identified the homogeneous three-climate configuration, we now ask
whether it persists when the insolation is allowed to vary weakly in space. For $Q_1\in \rC([-1,1])$, our analysis is based on the spatial heterogeneous insolation of the form
\begin{equation*}
Q_\varepsilon(x) = Q_*+\varepsilon Q_1(x).
\end{equation*}
For sufficiently small $|\varepsilon|$, the function $Q_\varepsilon$
remains strictly positive on $[-1,1]$, which gives physical meaning to the insolation.

We now use the implicit function theorem and spectral continuity to continue
the three homogeneous equilibria under small perturbations.
\begin{thm}[Persistence of the three-climate structure in the degenerate case]
\label{thm: existence three climate degenerate small eps}
Assume \autoref{ass:homogeneous-three-well}. Let
$Q_1\in \rC([-1,1])$, and let $Q_\varepsilon=Q_*+\varepsilon Q_1$. Then there
exists $\varepsilon_0 > 0$ such that, for every
$\abs{\varepsilon} < \varepsilon_0$, the stationary degenerate EBM
$$
-(\kappa(x)u')' +\mathfrak{e}u\abs{u}^3- Q_\varepsilon(x) \beta(u)-q=0
$$
admits three solutions
$$
u_S^\varepsilon,\qquad u_M^\varepsilon,\qquad u_W^\varepsilon
$$
in the operator domain $D(A_\kappa)$. Moreover,
\begin{equation}
u_i^\varepsilon\to \theta_i \qquad \text{in }D(A_\kappa) \quad\text{as }\varepsilon\to0, \qquad i\in\{S,M,W\}.
\label{eq: convergence degenerate eps small}
\end{equation}
The equilibria $u_S^\varepsilon$ and $u_W^\varepsilon$ are strict local minimisers of $G_{  \varepsilon}$ in $\rH^1_\kappa(\mathcal{I})$, while
$u_M^\varepsilon$ is non-degenerate and has Morse index one. In particular, for every sufficiently small $\abs{\varepsilon}$, the degenerate one-dimensional
EBM has two stable variational equilibria and one unstable Morse-index-one
transition equilibrium.
\end{thm}
\begin{proof}
Assume for simplicity that $\mathfrak{e} = 1$. We use the Hilbert-space formulation associated with the self-adjoint operator
$A_\kappa$. Define
$$
\mathcal{F}: \mathcal{U} \subset \mathbb{R} \times D(A_{\kappa}) \to \rL^2(\mathcal{I}),\qquad \mathcal{F}(\varepsilon,u) :=A_\kappa u+ u\abs{u}^3 -Q_\varepsilon(x) \beta(u)
-q,
$$
where $\mathcal{U}$ is an open a neighbourhood of $(0,\theta_i)$. By \autoref{lem:domain-embeds-Linfty-degenerate}, it holds
$D(A_\kappa) \hookrightarrow \rL^\infty (\mathcal{I})$. Consequently,
if $u$ remains in a bounded subset of $D(A_\kappa)$, then $u$ remains bounded
in $\rL^\infty(\mathcal{I})$. Since the maps
$$
s \mapsto s\abs{s}^3 \qquad\text{and} \qquad s \mapsto \beta(s)
$$
are $\rC^1$ on bounded intervals, the corresponding Nemytskii maps are $\rC^1$
from bounded subsets of $D(A_\kappa)$ into $\rL^2(\mathcal{I})$. Hence
$\mathcal{F}$ is $\rC^1$ near each $(0,\theta_i)$. Moreover,
$$
\mathcal{F}(0,\theta_i)=0, \qquad i\in\{S,M,W\}.
$$
The derivative with respect to $u$ at $(0,\theta_i)$ is
$$
D_u\mathcal{F}(0,\theta_i)v = A_\kappa v+\mu_i v,
$$
where $\mu_i$ is defined in \eqref{eq:mu-i-definition-degenerate}.

We claim that $A_\kappa+\mu_i$ is an isomorphism from $D(A_\kappa)$ to
$\rL^2(\mathcal{I})$. For $i=S,W$, this follows from $\mu_i>0$ and the nonnegativity of
$A_\kappa$. For $i=M$, the spectrum of $A_\kappa$ starts with $0$ and its first
positive eigenvalue is $\lambda_1^\kappa$. Since
$$
-\lambda_1^\kappa<\mu_M<0,
$$
we have $- \mu_M \in (0 , \lambda_1^\kappa)$, and therefore $-\mu_M$ is not an
eigenvalue of $A_\kappa$. Hence $A_\kappa+\mu_M$ has trivial kernel and is
invertible.

The implicit function theorem gives, for each $i\in\{S,M,W\}$, a unique
$C^1$ branch $
\varepsilon \mapsto u_i^\varepsilon \in D(A_\kappa)$ defined for $\abs{\varepsilon}$ sufficiently small, such that
$$
u_i^0\equiv\theta_i, \qquad \mathcal{F}( \varepsilon , u_i^\varepsilon)=0.
$$
This proves the existence of the three equilibria and the convergence
\eqref{eq: convergence degenerate eps small}.

It remains to prove the variational character of the three branches. The
second variation of $G_{ \varepsilon}$ at $ u_i^\varepsilon$ is
$$
D^2 G_{ \varepsilon} ( u_i^\varepsilon )[ v , v ] = \int_{\mathcal{I}}\kappa(x)(v'(x))^2\d x + \int_{\mathcal{I}}
a_i^\varepsilon ( x ) v ( x )^2 \d x,
$$
where $
a_i^\varepsilon(x) := 4\vert u_i^\varepsilon( x ) \vert ^3 - Q_\varepsilon(x) \beta ' (u_i^\varepsilon(x)).$ By \eqref{eq: convergence degenerate eps small},
$$
a_i^\varepsilon\to \mu_i \qquad \text{in }\rL^\infty(\mathcal{I}) \quad \text{as } \varepsilon \to 0.
$$
Therefore the spectra of the self-adjoint operators $
A_\kappa+a_i^\varepsilon(x)$ converge to the spectra of $A_\kappa+\mu_i$, in the sense of convergence of
isolated eigenvalues with multiplicity.

For $i=S,W$, the operator $A_\kappa+\mu_i$ is strictly positive. Hence, for
$\abs{\varepsilon}$ sufficiently small, $A_\kappa+a_i^\varepsilon(x)$ remains
strictly positive. Consequently, there exists $c_i>0$ such that
$$
D^2 G_{\varepsilon}(u_i^\varepsilon)[v,v] \geq c_i \| v \|_{\rH^1_\kappa(\mathcal{I})}^2 \qquad \forall v \in \rH^1_\kappa(\mathcal{I}).
$$
Thus $u_S^\varepsilon$ and $u_W^\varepsilon$ are strict local minimisers.

For $i=M$, the operator $A_\kappa+\mu_M$ has exactly one negative eigenvalue
and no kernel, by \autoref{lem:homogeneous-equilibria-morse-degenerate}.
Since the spectrum varies continuously under the perturbation
$a_M^\varepsilon - \mu_M \to 0$ in $\rL^\infty(\mathcal{I})$, the perturbed operator
$A_\kappa + a_M^\varepsilon (x ) $ also has exactly one negative eigenvalue and no
kernel, for all sufficiently small $\abs{\varepsilon}$. Hence
$u_M^\varepsilon$ is non-degenerate and has Morse index one.
\end{proof}

\begin{cor}[Mountain-pass critical point between the stable branches]
\label{cor:mountain-pass-between-branches-degenerate}
Assume the hypotheses of \autoref{thm: existence three climate degenerate small eps}. Then, for every sufficiently
small $\abs{\varepsilon}$, the functional $G_{\varepsilon}$ admits a
mountain-pass critical point $
u_{MP}^\varepsilon \in \rH^1_\kappa(\mathcal{I})$ between the two strict local minimisers $u_S^\varepsilon$ and
$u_W^\varepsilon$.

More precisely, let $
\Gamma_\varepsilon := \left\lbrace  \gamma \in \rC( [0,1] ; \rH^1_\kappa(\mathcal{I})): \gamma(0) = u_S^\varepsilon,\  \gamma(1) = u_W^\varepsilon \right\rbrace,$ and define the minimax level
\begin{equation}
c_\varepsilon :=\inf_{\gamma\in\Gamma_\varepsilon} \sup_{ t \in [0,1]} G_{ \varepsilon}( \gamma( t ) ).
\label{eq: MP level degenerate}
\end{equation}
Then there exists $ u_{MP}^\varepsilon \in \rH^1_\kappa(\mathcal{I})$ such that
$$
G_{\varepsilon}'( u_{ MP }^\varepsilon)=0, \qquad G_{\varepsilon}(u_{M P }^\varepsilon) = c_\varepsilon.
$$
Moreover $
c_\varepsilon> \max \left\lbrace  G_{\varepsilon} (u_S^\varepsilon ) , G_{ \varepsilon }( u_W^\varepsilon )
 \right \rbrace ,$ and therefore $
u_{MP}^\varepsilon \notin \{ u_S^\varepsilon, u_W^\varepsilon\}.$
\end{cor}
\begin{proof}
By \autoref{thm: existence three climate degenerate small eps}, the critical points
$ u_S^\varepsilon $ and $ u_W^\varepsilon $ are distinct and have coercive second
variation. Hence, by the $\rC^2$ regularity of $G_{\varepsilon}$ in
neighbourhoods of these two points, they are strict local minimisers with an energy gap on sufficiently small spheres. Therefore there exist
$r_S,r_W>0$, chosen so that
$$
B_{r_S}(u_S^\varepsilon)\cap B_{r_W}(u_W^\varepsilon)=\emptyset,
$$
and constants $\alpha_S,\alpha_W>0$ such that
$$
G_{\varepsilon}(u) \geq G_{\varepsilon}( u_S^\varepsilon ) + \alpha_S \qquad \text{for every }u \in \partial B_{r_S}(u_S^\varepsilon ),
$$
and
$$
G_{ \varepsilon}(u) \geq G_{\varepsilon}( u_W^\varepsilon) + \alpha_W \qquad
\text{for every } u \in \partial B_{r_W}(u_W^\varepsilon).
$$
Every continuous path connecting $u_S^\varepsilon$ to $u_W^\varepsilon$ must cross
both boundaries $\partial B_{r_S}(u_S^\varepsilon)$ and
$\partial B_{r_W}(u_W^\varepsilon)$. Therefore $
c_\varepsilon> \max \left\lbrace 
G_{  \varepsilon}( u_S^\varepsilon ) , G_{ \varepsilon} (u_W^\varepsilon)
\right\rbrace .$

Since $G_{ \varepsilon}$ satisfies the PS condition thanks to
\autoref{prop:PS-degenerate}, \autoref{cor:two-local-minima-mpt} applies to the two local minimisers
$u_S^\varepsilon$ and $u_W^\varepsilon$. Hence there exists a critical point
$ u_{MP}^\varepsilon \in \rH^1_\kappa(\mathcal{I})$ at the minimax level
$c_\varepsilon$. The strict inequality above shows that
$u_{MP}^\varepsilon$ is distinct from the two local minimisers.
\end{proof}

\subsection{Identification of the perturbative saddle with the mountain-pass point}

The previous perturbative argument produces a critical point
$u_M^\varepsilon$ by a local implicit-function-theorem construction near the
homogeneous middle equilibrium $u_M^0$. This critical point is non-degenerate
and has Morse index one. Independently, the Mountain-Pass Theorem applied
between the two perturbative local minimisers
$u_S^\varepsilon$ and $u_W^\varepsilon$ produces a critical point
$u_{MP}^\varepsilon$ at the minimax level $c_\varepsilon$.

These two critical points need not coincide for purely local reasons. Indeed,
the Morse-index-one property is a local spectral property, whereas the
mountain-pass characterisation is global and depends on the minimax geometry
between the two wells. A functional may have several Morse-index-one critical
points, and the mountain-pass level may in principle be realized by a different
saddle. The purpose of this subsection is to show, under the homogeneous
mountain-pass isolation condition in \autoref{ass:homogeneous-mountain-pass-isolation}, that for sufficiently small perturbations
the mountain-pass saddle is precisely the perturbative continuation of the
homogeneous middle equilibrium, that is
$$
u_{MP}^\varepsilon=u_M^\varepsilon.
$$

The identification argument is based on three auxiliary steps. First, we
prove that the perturbed minimax levels $c_\varepsilon$ converge to the
homogeneous level $c_0$. Second, we show that any sequence of perturbed
mountain-pass points is a Palais-Smale sequence for $G_0$, and therefore
converges, up to a subsequence, to a homogeneous critical point at level
$c_0$. The isolation assumption then identifies this limit with $u_M^0$.
Finally, we upgrade the resulting convergence from
$\rH^1_\kappa(\mathcal{I})$ to $D(A_\kappa)$, where the local uniqueness supplied
by the implicit function theorem can be applied.
\begin{lem}[Continuity of the mountain-pass level]
\label{lem:mp-level-continuity-degenerate}
Assume the hypotheses of \autoref{thm: existence three climate degenerate small eps}. Let $c_\varepsilon$ be the
mountain-pass level defined in
\eqref{eq: MP level degenerate}, and let $c_0$ be the homogeneous
mountain-pass level. Then
$$
c_\varepsilon\to c_0 \qquad \text{as }\varepsilon \to0.
$$
\end{lem}
\begin{proof}
Assume for simplicity that $\mathfrak{e} = 1$. We first prove the upper bound. Let $\eta>0$, and choose
$\gamma_0\in\Gamma_0$ such that
$$
\sup_{t\in[0,1]} G_{0}(\gamma_0(t)) \leq c_0+\eta.
$$
Since $u_S^\varepsilon\to u_S^0$ and
$u_W^\varepsilon\to u_W^0$ in $H^1_\kappa(\mathcal{I})$, we can concatenate
$\gamma_0$ with the short line segments joining $u_S^\varepsilon$ to $u_S^0$
and $u_W^0$ to $u_W^\varepsilon$. This gives a path
$\gamma_\varepsilon\in\Gamma_\varepsilon$. Along these short segments, the
energy converges to the energy of the corresponding endpoint. Since
$u_S^0$ and $u_W^0$ are strict local minimisers and
$c_0>\max\{G_{0}(u_S^0), G_{0}(u_W^0)\}$, the maximum on the
short segments is strictly below $c_0+\eta$ for $\abs{\varepsilon}$ small.

Moreover,
$$
G_{\varepsilon}(u)-G_{0}(u) = - \varepsilon \int_{ \mathcal{I}}Q_1(x) B(u(x))\d x.
$$
Since $B$ has at most linear growth and $\gamma_0([0,1])$ is compact in
$\rH^1_\kappa(\mathcal{I})$, the right-hand side is $o(1)$ uniformly for
$u\in\gamma_0([0,1])$. Therefore $
\limsup_{\varepsilon\to0}c_\varepsilon
\leq c_0+\eta.$ Letting $\eta\downarrow0$ gives
\begin{equation}
\label{eq: bound limsup paths}
    \limsup_{\varepsilon\to0}c_\varepsilon\leq c_0.
\end{equation}

We now prove the lower bound. Let $\varepsilon_n\to0$, and choose
$\gamma_n\in\Gamma_{\varepsilon_n}$ such that $
\sup_{t\in[0,1]} G_{\varepsilon_n}(\gamma_n(t))
\leq c_{\varepsilon_n}+\frac{1}{n}.$ The upper bound just proved implies that $(c_{\varepsilon_n})_n$ is bounded
from above. For every $t\in[0,1]$ we have
$$
G_{\varepsilon_n}( \gamma_n ( t ) ) \leq c_{\varepsilon_n}+\frac{1}{n}.
$$
The bound in \eqref{eq: bound limsup paths} implies that the right-hand side is bounded
uniformly in $n$. Moreover, since $Q_{\varepsilon}=Q_*+\varepsilon Q_1$, the quantities
$\|Q_{\varepsilon}\|_{\infty }$ are uniformly bounded for $|\varepsilon|$ small. Therefore the proof of the coercive bound
\eqref{eq: coercive energy bound degenerate} applies uniformly to $G_{ \varepsilon}$, and thus there exist constants $c,C>0$, independent of
$\varepsilon$, such that
\begin{equation}
G_{ \varepsilon}(u) \geq \frac{1}{2} \int_{\mathcal{I}} \kappa(x) ( u' ( x ) )^2\d x + c \| u \|_{ 5 }^5 - C \qquad  \forall u \in \rH^1_\kappa(\mathcal{I}),
\label{eq:uniform-coercive-energy-bound-epsilon-degenerate}
\end{equation}
for all sufficiently small \(|\varepsilon|\). Hence an upper bound on $G_{\varepsilon_n}(\gamma_n(t))$, uniform in $n$ and $t$, implies uniform bounds on
$$
\| \sqrt{\kappa} \, \gamma_n'(t) \|_{2 } \qquad \text{and} \qquad \| \gamma_n( t ) \|_{5}.
$$
Thanks to the embedding $\rL^5(\mathcal{I}) \hookrightarrow \rL^2(\mathcal{I})$, we have that there exists \(R>0\) such that
$$
\| \gamma_n ( t ) \|_{ \rH^1_\kappa(\mathcal{I}) } \leq R \qquad \text{for every } t \in[0,1] \text{ and every } n .
$$
Therefore
$$
\sup_{t\in[0,1]} \vert  G_{  \varepsilon_n }( \gamma_n(t ) ) - G_{  0 }( \gamma_n(t)) \vert  \to 0.
$$

As before, concatenate $\gamma_n$ with the short line segments joining
$u_S^0$ to $u_S^{\varepsilon_n}$ and
$u_W^{\varepsilon_n}$ to $u_W^0$. This produces a path
$\tilde\gamma_n\in\Gamma_0$. The contribution of the short endpoint segments
has maximum strictly below $c_0$ for $n$ sufficiently large, because the
endpoints converge to the strict local minimisers $u_S^0$ and $u_W^0$. Hence
$$
c_0 \leq \sup_{t\in[0,1]} G_{0}(\tilde\gamma_n(t)) \leq c_{\varepsilon_n}+o(1).
$$
Thus $ c_0 \leq \liminf_{n\to\infty}c_{\varepsilon_n}.$ Since the sequence $\varepsilon_n\to0$ was arbitrary, this proves $ \liminf_{\varepsilon\to0}c_\varepsilon\geq c_0.$ Combining the two bounds gives $c_\varepsilon\to c_0$.
\end{proof}
\noindent
We next show that perturbed mountain-pass points form a Palais-Smale sequence
for the homogeneous functional.

\begin{lem}[Perturbative mountain-pass points are homogeneous Palais-Smale sequences]
\label{lem:mp-points-homogeneous-ps-degenerate}
Let $\varepsilon_n\to0$, and let
$u_n:=u_{MP}^{\varepsilon_n}$ be mountain-pass critical points of
$G_{\varepsilon_n}$ at level $c_{\varepsilon_n}$. Then
$(u_n)_n$ is a Palais-Smale sequence for the homogeneous functional
$ G_{0}$ at level $c_0$. In particular, up to a subsequence, it holds
$$
u_n \to u_* \qquad \text{in } \rH^1_\kappa( \mathcal{I}),
$$
where $u_*$ is a critical point of $ G_{0}$ and $
G_{0}(u_*) = c_0.$
\end{lem} 
\begin{proof}
By \autoref{lem:mp-level-continuity-degenerate}, it holds $
c_{\varepsilon_n}\to c_0. $ Since $
G_{ \varepsilon_n}(u_n)=c_{\varepsilon_n},$ the energy levels of $u_n$ are bounded. The same argument applied in the proof of \autoref{lem:mp-level-continuity-degenerate} after the coercive estimate \eqref{eq:uniform-coercive-energy-bound-epsilon-degenerate}, gives that $(u_n)_n$ is bounded in
$\rH^1_\kappa(\mathcal{I})$. Moreover,
$$
G_{0}(u_n) - G_{\varepsilon_n}(u_n) = \varepsilon_n\int_{\mathcal{I} }Q_1(x)B(u_n(x))\d x.
$$
Since $B$ has at most linear growth and $(u_n)_n$ is bounded in
$\rH^1_\kappa(\mathcal{I})$, the right-hand side tends to zero. Hence $
 G_{0}(u_n)\to c_0.$

It remains to check the derivative. Since $u_n$ is a critical point of
$ G_{\varepsilon_n}$, then $
G_{\varepsilon_n}'(u_n)=0.$ For every $v\in \rH^1_\kappa(\mathcal{I})$,
$$
\left\langle G_{ 0 }'( u_n )-G_{\varepsilon_n}'(u_n), v \right\rangle = \varepsilon_n \int_{\mathcal{I}} Q_1 ( x ) \beta( u_n (x) ) v( x )\d x.
$$
Since $\beta$ is bounded and $Q_1\in \rL^\infty(\mathcal{I})$, we have
$$
\vert  \left \langle G_{0}'(u_n), v \right \rangle \vert  \leq C\abs{ \varepsilon_n } \| v \|_{2} \leq C \abs{ \varepsilon_n} \| v \|_{ \rH^1_\kappa(\mathcal{I})}.
$$
Thus $ G_{0}'(u_n) \to 0 $ in $(\rH^1_\kappa)^*.$ Therefore $(u_n)_n$ is a Palais-Smale sequence for $ G_{0}$ at level
$c_0$. Since $G_{0}$ satisfy the PS-condition thanks to \autoref{prop:PS-degenerate}, after passing to a
subsequence, we have
$$
u_n\to u_* \qquad  \text{strongly in } \rH^1_\kappa(\mathcal{I}).
$$
Since $ G_{0} \in \rC^1(\rH^1_\kappa; \mathbb{R})$ and $u_k \to u_*$ in $\rH^1_\kappa$, we have $ G_{0}' (u_n) \to  0  =G_{0}'(u_*) $ and $G_{0}(u_n) \to c_0 = G_{0}(u_*) $.
\end{proof}
\noindent
To use the local uniqueness given by the implicit function theorem, we upgrade
the convergence from $\rH^1_\kappa(\mathcal{I})$ to $D(A_\kappa)$.

\begin{lem}[Upgrade of critical-point convergence]
\label{lem:critical-convergence-upgrade-degenerate}
Let $ \varepsilon_n \to0$, and let $ u_n $ be critical points of
$G_{ \varepsilon_n}$. Assume that
$$
u_n \to u_M^0 \qquad \text{strongly in } \rH^1_\kappa ( \mathcal{I} ).
$$
Then
$$
u_n\to u_M^0 \qquad \text{strongly in } D(A_\kappa).
$$
\end{lem}
\begin{proof}
Assume for simplicity that $\mathfrak{e}=1$. We first prove a uniform $\rL^\infty$ bound. Since $u_n$ is a critical point of
$G_{  \varepsilon_n}$, it satisfies
$$
\int_{\mathcal{I}}\kappa u_n'v'\d x + \int_{\mathcal{I}} ( u_n \abs{u_n}^3 - Q_{ \varepsilon_n}  \beta(u_n)
- q )v \d x = 0
$$
for every $v\in H^1_\kappa(\mathcal{I})$. Since
$Q_{ \varepsilon_n } = Q_* + \varepsilon_n Q_1$ is uniformly bounded in
$\rL^\infty ( -\mathcal{I})$ and $0 \leq \beta \leq 1$, choose $M>0$ so large that
$$
M^4> \sup_n \| Q_{ \varepsilon_n }\|_{ \infty  } + \abs{q}.
$$

Testing the equation with $v = (u_n-M)^+$ leads to
$$
\int_{\{u_n>M\}}\kappa |u_n'|^2\d x + \int_{\{u_n>M\}} ( u_n^4 -Q_{\varepsilon_n} \beta( u_n) - q )( u_n - M)\d x
=0.
$$
By the choice of $M$, the second integrand is non-negative on
$\{u_n > M \}$. Hence both terms vanish, and therefore
$(u_n - M )^+=0$. Thus $u_n\leq M$ a.e.

Similarly, testing the equation with $v = - (u_n+M)^-$ gives
$$
\int_{ \{ u_n < -M\} } \kappa  |u_n' |^2 \d x + \int_{ \{ u_n < - M \}}
( u_n\abs{u_n}^3 - Q_{\varepsilon_n} \beta(u_n)-q
)(u_n+M)\d x =0.
$$
On $\{ u_n < -M\}$, we have $u_n+M<0$ and
$$
u_n \abs{ u_n }^3 - Q_{ \varepsilon_n } \beta ( u_n ) - q \leq -M^4+ \sup_n \| Q_{ \varepsilon_n} \|_{\infty} + \abs{q}
<0.
$$
Therefore the second integrand is non-negative. Hence both terms are zero, and
$(u_n+M)^-=0$. Thus $u_n \geq -M$ a.e. We thus have proved that
$$
\| u_n \|_{ \infty }\leq M \qquad \text{uniformly in } n.
$$
In particular, the right-hand side
$$
-u_n\abs{u_n}^3+Q_{\varepsilon_n} \beta(u_n) +q
$$
belongs to $L^2(\mathcal{I})$, uniformly in $n$. Hence, by the definition of the
operator associated with the form, $ u_n \in D(A_\kappa)$ and
$$
A_\kappa u_n = -u_n \abs{u_n}^3 + Q_{\varepsilon_n} \beta(u_n) + q
$$
in $\rL^2(\mathcal{I})$.

Since $u_n \to u_M^0$ in $H^1_\kappa(\mathcal{I})$, in particular
$u_n \to u_M^0$ in $\rL^2(\mathcal{I})$. Together with the uniform $\rL^\infty$ bound, this implies $
u_n \abs{u_n}^3
\to u_M^0\abs{u_M^0}^3$ in $\rL^2(\mathcal{I}).$ Moreover, since $Q_{\varepsilon_n}\to Q_*$ in $\rL^\infty(\mathcal{I})$ and
$\beta$ is Lipschitz on bounded intervals, we also have
$$
Q_{ \varepsilon_n} \beta(u_n) \to Q_* \beta(u_M^0) \qquad \text{in }\rL^2(\mathcal{I}).
$$

The critical point equation may be written as
$$
A_\kappa u_n = -u_n\abs{u_n}^3 + Q_{\varepsilon_n} {\beta}(u_n) + q.
$$
Similarly,
$$
A_\kappa u_M^0 = -u_M^0\abs{u_M^0}^3 + Q_*\beta (u_M^0) + q.
$$
Therefore $
A_\kappa u_n \to A_\kappa u_M^0$ in $L^2 (\mathcal{I}).
$ Since also $u_n\to u_M^0$ in $\rL^2(\mathcal{I})$, we conclude that 
$$
u_n\to u_M^0 \qquad \text{in }D(A_\kappa),
$$
with respect to the graph norm of $A_\kappa$.
\end{proof}
\noindent
The preceding lemmas now allow us to identify the mountain-pass point with the
perturbative middle branch.

\begin{prop}[Identification of the perturbative saddle with the mountain-pass saddle]
\label{prop:identification-mp-morse-degenerate}
Assume the hypotheses of \autoref{thm: existence three climate degenerate small eps} and \autoref{ass:homogeneous-mountain-pass-isolation}. Then, for every
sufficiently small $\abs{\varepsilon}$, every mountain-pass critical point
$u_{MP}^\varepsilon$ produced by \autoref{cor:mountain-pass-between-branches-degenerate} coincides with the
perturbative Morse-index-one equilibrium $u_M^\varepsilon$ from \autoref{thm: existence three climate degenerate small eps}, i.e. $
u_{MP}^\varepsilon=u_M^\varepsilon.$
\end{prop}
\begin{proof}
We argue by contradiction. Suppose that there exist $\varepsilon_n \to 0$ and
mountain-pass critical points $u_{MP}^{\varepsilon_n}$ at the minimax levels
$c_{\varepsilon_n}$ such that
$u_{MP}^{\varepsilon_n}\neq u_M^{\varepsilon_n}$
for every $n$. Set $ u_n:=u_{MP}^{\varepsilon_n}.$ By \autoref{lem:mp-points-homogeneous-ps-degenerate}, after passing to a
subsequence,
$$
u_n \to u_* \qquad \text{strongly in }\rH^1_\kappa(\mathcal{I}),
$$
where $
G_{ 0 }'( u_ *) = 0$ and $  G_{0}(u_*) = c_0.$ By \autoref{ass:homogeneous-mountain-pass-isolation}, the
only critical point of the homogeneous functional $G_{0}$ at level
$c_0$ is $u_M^0$. Therefore $
u_*=u_M^0.$ Hence $
u_n\to u_M^0$  in $\rH^1_\kappa(\mathcal{I})$. Since each $u_n$ is a critical point of $ G_{\varepsilon_n}$,
\autoref{lem:critical-convergence-upgrade-degenerate} gives the stronger
convergence $u_n\to u_M^0$ in $D(A_\kappa).$

On the other hand, by the implicit-function-theorem construction in the proof
of \autoref{thm: existence three climate degenerate small eps}, there exist
a neighbourhood $\mathcal U$ of $u_M^0$ in $D(A_\kappa)$ and a number
$ \varepsilon_1>0$ such that, for every $|\varepsilon|< \varepsilon_1$, the
only critical point of $ G_{\varepsilon}$ in $\mathcal{U}$ is
$ u_M^\varepsilon$. Since $u_n\to u_M^0$ in $D( A_\kappa)$, we have
$u_n \in \mathcal{U}$ for all sufficiently large $n$. Therefore
$ u_{ MP }^{\varepsilon_n} = u_M^{\varepsilon_n} $
for all sufficiently large $n$, contradicting the choice of the sequence.
\end{proof}

\begin{rem}
\autoref{ass:homogeneous-mountain-pass-isolation} is a
global assumption. The Morse-index-one property of $u_M^0$ is local and does
not, by itself, imply that $u_M^0$ is the unique critical point at the
homogeneous mountain-pass level. Thus the identification
$u_{MP}^\varepsilon=u_M^\varepsilon$ requires the homogeneous
mountain-pass isolation input stated above.
\end{rem}

\begin{cor}[Perturbative tri-stability in the degenerate case]
\label{cor: corllary 1D tri-stability}
Under the assumptions of \autoref{thm: existence three climate degenerate small eps}, for every sufficiently small
$\abs{\varepsilon}$ the functional $ G_{\varepsilon}$ has at least
three distinct critical points in $H^1_\kappa(-1,1)$. More precisely, it has
two strict local minimisers $ u_S^\varepsilon, u_W^\varepsilon, $ corresponding to stable cold and warm climates, and one non-degenerate Morse-index-one critical point $
u_M^\varepsilon,$ corresponding to an unstable transition climate. In addition, by \autoref{cor:mountain-pass-between-branches-degenerate}, there exists a
mountain-pass critical point $ u_{MP}^\varepsilon$ between $u_S^\varepsilon$ and $u_W^\varepsilon$. If \autoref{ass:homogeneous-mountain-pass-isolation} also holds, then, for
all sufficiently small $\abs{\varepsilon}$, $
u_{MP}^\varepsilon=u_M^\varepsilon.$
\end{cor}

\section{The spherical model}
\label{sec: spherical model}
\noindent 
Throughout this section, the geometry is fixed to $\mathcal M=\S^2$. To simplify the notation, we suppress the
superscript $\S^2$ and write just $G_\varepsilon,
\Gamma_\varepsilon, c_\varepsilon, a_\kappa, A_\kappa,\lambda_1^\kappa.$
\subsection{Geometric and variational framework on $\mathbb S^2$}
\noindent 
We start by recalling the Poincar\'e inequality and the Rellich-Kondrachov theorem on
the compact manifold $\mathbb S^2$, see \cite[Corollary 4.3]{Aubin1998}  and
\cite[Theorem 2.34]{Aubin1998} respectively.

\begin{lem}[Poincar\'e inequality in $\mathbb{S}^2$]
    There exists $C>0$ such that for all $u \in \rH^1 (\mathbb{S}^2)$
    $$
    \| u - \bar u \|_{2}^2 \leq C \| \nabla_{\mathbb{S}^2} u \|_{2}^2, \quad \bar u := \frac{1}{\abs{\mathbb{S}^2}} \int_{\mathbb{S}^2} u \d \omega .
    $$
\end{lem}
\begin{lem}[Rellich-Kondrachov embeddings in $\mathbb{S}^2$]
    For any $p \in [1, \infty)$, the embedding $\rH^1 (\mathbb{S}^2) \hookrightarrow \rL^p ( \mathbb{S}^2)$ is compact.
\end{lem}

\begin{rem}[Identification of the operator domain]
\label{rem: H2 = dom A_k sphere}
Let $A_\kappa$ denote the $\rL^2(\mathbb{S}^2)$-realisation of
$$
A_\kappa u = -\div _{\mathbb{S}^2} (\kappa \nabla_{\mathbb{S}^2} u )
$$
associated with the closed symmetric form
$$
a_\kappa(u,v) = \int_{\mathbb{S}^2}  \kappa \nabla_{\mathbb{S}^2}u \cdot \nabla_{ \mathbb{S}^2}v\d \omega,
\qquad u,v\in \rH^1(\mathbb{S}^2).
$$
Since $\kappa \in \rW^{1,\infty}(\mathbb{S}^2)$ and
$ \kappa \geq k_*>0$, the operator is uniformly elliptic with Lipschitz
coefficients. Hence the standard elliptic $\rH^2$ regularity theorem on the
compact manifold $\mathbb{S}^2$ gives
$$
u \in D(A_\kappa) \quad \Longrightarrow \quad
u\in \rH^2(\mathbb{S}^2).
$$
Conversely, if $u\in \rH^2(\mathbb S^2)$, then
$$
A_\kappa u = -\kappa \Delta_{\mathbb{S}^2 } u - \nabla_{\mathbb{S}^2} \kappa \cdot \nabla_{\mathbb{S}^2}u
\in \rL^2(\mathbb{S}^2),
$$
because $\kappa,\nabla_{\mathbb{S}^2} \kappa \in \rL^\infty$. Therefore
$ u\in D(A_\kappa)$. Thus
$$
D(A_\kappa) = \rH^2 (\mathbb{S}^2),
$$
with equivalence between the graph norm of $A_\kappa$ and the usual
$\rH^2$-norm.
\end{rem}
\noindent 
We are now ready to prove the existence and the properties of the minimisers of the energy functional $G_\varepsilon$.
\begin{prop}[Existence and regularity of minimisers on $\S^2$]
\label{prop:minimiser-sphere-regularity}
Let
$$
\kappa \in \rW^{1,\infty}(\mathbb S^2), \qquad Q_\varepsilon\in \rC(\mathbb S^2), \qquad \kappa\ge k_*>0,
\qquad Q_\varepsilon \geq 0,
$$
and let $q \geq0 $. Then the energy functional
$$
G _\varepsilon (u) := \int_{\mathbb{S}^2} \left(     \mathfrak{e}  \frac{\abs{u}^5}{5}- Q_\varepsilon(\omega) B(u) - q u \right) \d \omega + \frac{1}{2} \int_{\mathbb{S}^2} \kappa (\omega) \abs{ \nabla_{\mathbb{S}^2} u}^2 \d \omega, 
$$
admits a minimiser
$u_0 \in \rH^1(\mathbb{S}^2)$. Moreover, $u_0 \geq 0$ a.e. on $\mathbb S^2$, and $u_0$
satisfies
$$
\int_{\mathbb{S}^2} \kappa \nabla_{\mathbb{S}^2} u_0\cdot\nabla_{\mathbb{S}^2} \phi \d\omega
+ \int_{\mathbb{S}^2} ( \mathfrak{e}|u_0|^3u_0 - Q_\varepsilon \beta(u_0)-q
)\phi\d\omega =0
$$
for every $\phi\in \rH^1(\mathbb{S}^2)$. If $A_\kappa$ denotes the self-adjoint operator associated with the form
$$
a_\kappa(u,v) := \int_{\mathbb{S}^2} \kappa \nabla_{\mathbb{S}^2} u \cdot\nabla_{\mathbb{S}^2}v\,d\omega,
\qquad u, v\in \rH^1(\mathbb S^2),
$$
then $
u_0\in D(A_\kappa) = \rH^2(\mathbb{S}^2)$ and $
A_\kappa u_0 = Q_\varepsilon \beta(u_0)+q-\mathfrak{e}|u_0|^3u_0 $ in  $\rL^2(\mathbb S^2).$
\end{prop}
\begin{proof}
Assume for simplicity that $\mathfrak{e} =1$. From \eqref{eq: inequality B(u)} and $Q_\varepsilon\in \rL^\infty(\mathbb{S}^2)$, there exist constants $C_0,C_1>0$ such that for all $s\in\mathbb R$ and $\omega\in\mathbb S^2$
$$
\frac{\abs{s}^5}{5}- Q_\varepsilon(\omega) B(s) - qs \geq \frac{\abs{s}^5}{5}-C_0-C_1\abs{s}.
$$
Since the term $\abs{s}^5$ dominates the linear term for large $\abs{s}$, there exist $c_2,c_3>0$ such that
$$
\frac{\abs{s}^5}{5}- Q_\varepsilon(\omega) B(s) - qs \geq c_2 \abs{s}^5- c_3.
$$
Integrating the previous inequality, we get
\begin{equation}
 G_\varepsilon (u) \geq c_2 \| u \|^5_{5}- c_3 \abs{\mathbb{S}^2}+ \frac{1}{2} \int_{\mathbb{S}^2} \kappa \abs{ \nabla_{\mathbb{S}^2}u}^2 \d \omega .
\label{eq: spehere var problem bound integral}
\end{equation}

\textit{Step 1: Existence of the minimiser in $\rH^1$}.\\ Let $(u_n)_n \subset \rH^1(\mathbb{S}^2)$ a minimising sequence, i.e. $ G _\varepsilon (u_n) \rightarrow \inf  {G}_\varepsilon $.  Then, $\sup_n  {G}_\varepsilon (u_n) < \infty $ and by \eqref{eq: spehere var problem bound integral} we deduce the uniform bounds
$$
\sup_n \|u_n\|_{5} < \infty, \qquad \sup_n \|\nabla _{\mathbb{S}^2}u_n\|_{2}<\infty.
$$
Let $\bar u_n:=\frac{1}{\abs{\mathbb{S}^2}} \int_{\mathbb{S}^2} u_n\d\omega $. Since $\rL^5(\mathbb{S}^2)\hookrightarrow \rL^1(\mathbb{S}^2)$, we have
$$
\abs{\bar{u}_n}\leq\frac{1}{\abs{\mathbb{S}^2}} \|u_n\|_{1} \leq C\|u_n \|_{5},
$$
so $(\bar u_n)_n$ is bounded. By Poincar\'e inequality
$$
\|u_n \|_2\leq \|u_n -\bar{u}_n\|_2 + \abs{\bar{u}_n} \abs{\mathbb{S}^2}^{1/2}\leq C(\| \nabla_{\mathbb{S}^2} u_n\|_2 +1  ),
$$
hence $(u_n)_n$ is bounded in $\rH^1 (\mathbb{S}^2).$ By boundedness, up to a subsequence, there exists $u_0 \in \rH^1(\mathbb{S}^2)$ such that $
u_n \rightharpoonup u_0 $ in $  \rH^1 (\mathbb{S}^2).$ By Rellich-Kondrachov Theorem, we also have
\begin{equation}
u_n  \xrightarrow{} u_0 \; \text{ in } \rL^1(\mathbb{S}^2),\quad  u_n \xrightarrow{} u_0 \; \text{ in } \rL^5(\mathbb{S}^2), \quad u_n \xrightarrow{} u_0 \; \text{ a.e. in }  \mathbb{S}^2.
\label{eq: Lp convergence and pointwise sphere variational problem}
\end{equation}
Consider the decomposition $G _\varepsilon = G^1 +    G^2,$ with
$$
G^1(u) =\int_{\mathbb{S}^2} \left(       \frac{\abs{u}^5}{5}- Q_\varepsilon(\omega) B(u) - q u \right) \d \omega, \qquad G^2 (u) = \frac{1}{2} \int_{\mathbb{S}^2} \kappa (\omega) \abs{ \nabla_{\mathbb{S}^2} u}^2 \d \omega.
$$
Note that, using that $u_n \xrightarrow[]{} u_0 $ in $\rL^5$, we have $ \int \abs{ u_n}^5 \xrightarrow[]{} \int \abs {u_0}^5$. Moreover, since $\beta$ is bounded and $B'=\beta$, the primitive $B$ is globally Lipschitz. Hence, using the strong convergence
$u_n\to u_0$ in $\rL^1(\mathbb S^2)$, we obtain
$$
\|B(u_n)-B(u_0)\|_{1} \leq \|\beta\|_{\infty} \|u_n-u_0\|_{1} \to 0.
$$
Since $Q_\varepsilon\in \rL^\infty(\mathbb S^2)$, this gives $
Q_\varepsilon B(u_n)\to Q_\varepsilon B(u_0)$ in $
\rL^1(\mathbb S^2).$ Thus
$
G^1(u_n) \xrightarrow[]{} G^1 (u_0).$ For $G^2$, since $\nabla_{\mathbb S^2}u_n\rightharpoonup
\nabla_{\mathbb S^2}u_0$ in $\rL^2(\mathbb S^2)$ and multiplication by $\sqrt\kappa\in \rL^\infty(\mathbb S^2)$ is continuous on $\rL^2$, we have
$$
\sqrt \kappa \nabla_{\mathbb{S}^2}u_n \rightharpoonup \sqrt\kappa \nabla_{\mathbb{S}^2}u_0 \quad \text{in }\rL^2(\mathbb S^2).
$$
Therefore, by weak lower semicontinuity of the $\rL^2$-norm,
$$
G^2(u_0) \leq \liminf_{ n\to\infty} G^2(u_n).
$$
Therefore
$$
 G_\varepsilon (u_0) \leq \liminf_{n \to \infty}  G_\varepsilon (u_n) = \inf_{ \rH^1 (\mathbb{S}^2)}  G _\varepsilon.
$$
Thus $u_0$ is a minimiser.

\textit{Step 2: $u_0\ge 0$.}\\
Set $u_0^+:=\max \{u_0,0\}$ and $u_0^-:=\max\{-u_0,0\}$, so that $u_0=u_0^+-u_0^-$ and $u_0^+u_0^-=0$ a.e.
Since the map $T(s):=s_+=\max\{s,0\}$ is $1$-Lipschitz, the Sobolev chain rule gives
$u_0^+\in \rH^1(\mathbb S^2)$ and
$$
\nabla_{\mathbb{S}^2}u_0^+=\nabla_{\mathbb S^2}u_0 \, \mathds{ 1}_{\{u_0>0\}} \qquad\text{a.e. on }\mathbb S^2,
$$
hence, using $\kappa \geq k_*>0$,
\begin{equation}\label{eq:grad-trunc-sphere}
\int_{\mathbb {S}^2}\kappa |\nabla_{\mathbb S^2}u_0^+|^2\d\omega \leq \int_{\mathbb {S}^2}\kappa \vert \nabla_{\mathbb{S}^2}u_0 \vert ^2 \d \omega .
\end{equation}

For the potential part, define for fixed $\omega\in\mathbb S^2$ the function
$$
F_\omega(s):=\frac{|s|^5}{5}-Q_\varepsilon(\omega)B(s)-qs .
$$
Since $\beta \geq 0$, we have $B'= \beta\geq 0$, so $B$ is non-decreasing and in particular
$B(s)\leq  B(0)$ for all $s\leq 0$. Therefore, for $s\leq 0$ and $q\geq 0$,
$$
F_\omega(s)-F_\omega(0) =\frac{\vert s \vert^5}{5}- Q_\varepsilon(  \omega)\left(B(s)-B(0)\right)-qs
=\frac{\vert s\vert^5}{5}+Q_\varepsilon(\omega)\left(B(0)-B(s)\right)+q \vert s \vert  \geq\ 0,
$$
with strict inequality if $s<0$. Hence pointwise a.e.
$$
F_\omega(u_0^+(\omega)) \leq F_\omega(u_0(\omega)),
$$
and integrating gives
\begin{equation}\label{eq:pot-trunc-sphere}
\int_{\S^2}\left(\frac{|u_0^+|^5}{5}-Q_\varepsilon B(u_0^+)-q u_0^+\right)\d\omega \leq \int_{\S^2}\left(\frac{|u_0|^5}{5}-Q_\varepsilon B(u_0)-q u_0\right)\d\omega,
\end{equation}
again with strict inequality if $\{u_0<0\}$ has positive measure.

Combining \eqref{eq:grad-trunc-sphere} and \eqref{eq:pot-trunc-sphere} we obtain
$$
G_\varepsilon(u_0^+)\leq G_\varepsilon(u_0).
$$
Since $u_0$ is a minimiser, equality must hold, which forces $\vert \{u_0<0 \} \vert =  0$.
Therefore $u_0\geq 0$ a.e. on $\mathbb {S}^2$.

\textit{Step 3: Euler-Lagrange equation (weak stationary EBM).}\\
We claim that $G_\varepsilon$ is Gateaux differentiable on $\rH^1(\mathbb S^2)$ and for every $u,\phi\in \rH^1(\mathbb{S}^2)$,
\begin{equation}\label{eq:EL-derivative-sphere}
\left\langle  G_\varepsilon'(u),\phi \right\rangle =\int_{\mathbb{S}^2}\left(|u|^3u - Q_\varepsilon(\omega)\beta(u) - q\right) \phi \d \omega +\int_{\mathbb{S}^2}\kappa(\omega) \nabla_{\mathbb{S}^2}u\cdot \nabla_{\S^2}\phi\d\omega .
\end{equation}
Indeed, on the compact manifold $\mathbb S^2$ we have the continuous embeddings
$\rH^1(\mathbb{S}^2)\hookrightarrow \rL^p(\mathbb{S}^2)$ for every $p<\infty$, hence all terms above are integrable; differentiating
$t\mapsto \int \frac{\vert u+t\phi\vert^5}{5}$ gives $\int |u|^3u\,\phi$, differentiating $t\mapsto \int Q_\varepsilon B(u+t\phi)$ gives
$\int Q_\varepsilon \beta(u)\phi$, and the quadratic gradient term yields
$\int \kappa\,\nabla u\cdot\nabla\phi$.

Now fix $\phi\in \rH^1(\mathbb S^2)$ and consider $\varphi(t):= G_\varepsilon(u_0+t\phi)$.
Since $u_0$ is a global minimiser, $t=0$ is a minimum of $\varphi$, hence $\varphi'(0)=0$.
Using \eqref{eq:EL-derivative-sphere} we conclude that for every $\phi\in \rH^1(\mathbb S^2)$
$$
\int_{\mathbb {S}^2}\kappa(\omega)\nabla_{\mathbb {S}^2}u_0\cdot \nabla_{\mathbb {S}^2}\phi\d\omega
+\int_{\mathbb{S}^2}\left( \vert u_0 \vert ^3u_0 - Q_\varepsilon(\omega) \beta(u_0) - q\right) \phi\d\omega=0.
$$
This is exactly the weak formulation of the stationary EBM on $\mathbb S^2$.

\textit{Step 4: Regularity bootstrap.}\\
We now improve the regularity of $u_0$. Let $A_\kappa$ be the operator associated
with the closed form
$$
a_\kappa(u,v) := \int_{\S^2} \kappa\nabla_{\mathbb S^2}u\cdot\nabla_{\mathbb S^2}v \d\omega,\qquad u, v\in \rH^1(\mathbb S^2).
$$
The weak Euler-Lagrange equation obtained above can be rewritten as
$$
A_\kappa u_0 = Q_\varepsilon \beta(u_0)+q-|u_0|^3u_0
$$
in the sense of distributions.

We claim that the right-hand side belongs to $\rL^2(\mathbb S^2)$. Indeed, since
$\mathbb S^2$ is two-dimensional and compact, we have
$$
\rH^1(\mathbb S^2)\hookrightarrow \rL^p(\mathbb S^2) \qquad
\text{for every }1\le p<\infty.
$$
In particular $u_0\in \rL^8(\mathbb S^2)$, and therefore
$$
|u_0|^3u_0\in \rL^2(\mathbb S^2).
$$
Moreover, $Q_\varepsilon \in \rC(\mathbb S^2)\subset \rL^\infty(\mathbb S^2)$, and
$\beta$ is bounded. Hence
$$
Q_\varepsilon\beta(u_0)+q-|u_0|^3u_0\in \rL^2(\mathbb S^2).
$$
By the definition of the operator associated with the form, this implies
$$
u_0\in D(A_\kappa), \qquad A_\kappa u_0 = Q_\varepsilon\beta(u_0)+q-|u_0|^3u_0
\quad \text{in }\rL^2(\mathbb S^2).
$$

Since $\kappa\in \rW^{1,\infty}(\mathbb S^2)$ and $\kappa\ge k_*>0$, the operator
$A_\kappa$ is uniformly elliptic with Lipschitz coefficients. Standard elliptic $\rH^2$-regularity for uniformly elliptic divergence-form operators with Lipschitz coefficients gives $
u_0\in \rH^2(\mathbb S^2).$
\end{proof}

\subsection{Three-climate structure on $\mathbb S^2$}
\noindent 
We now develop on $\mathbb S^2$ the same variational programme as in the
degenerate one-dimensional case: Palais-Smale compactness, persistence of
the three homogeneous equilibria, mountain-pass existence, and identification
under an additional isolation assumption.

\subsubsection{Palais-Smale compactness}
\noindent
We first verify the compactness property required for the mountain-pass
argument.

\begin{prop}[Palais-Smale condition on $\mathbb{S}^2$]
\label{prop: Palais Smale sphere}
Let $ q \in \mathbb{R}$ be fixed. Assume $ \kappa \in \rL^{\infty}(\mathbb{S}^2),$ with $ \kappa ( \omega ) \geq k_*>0$ and $ Q_\varepsilon \in  \rL^\infty (\mathbb{S}^2)$. Let $\beta$ and its primitive $B$ satisfy \eqref{eq:extension coalbedo} and \eqref{eq: inequality B(u)} respectively. Then the functional
$$
G_\varepsilon(u )= \int_{\mathbb{S}^2 } \left(\mathfrak{e} \frac{|u|^5}{5} - Q_\varepsilon(\omega) B(u)-q u \right) \d \omega + \frac{1}{2} \int_{\mathbb {S}^2} \kappa( \omega ) \, \vert \nabla_{\mathbb{S}^2} u \vert ^2 \d \omega
$$
satisfies the Palais-Smale condition on $\rH^1(\mathbb S^2)$.
\end{prop}

\begin{proof}
Assume for simplicity that $\mathfrak{e} =1$. Let $(u_n)_n\subset \rH^1(\mathbb S^2)$ be a Palais-Smale sequence, i.e.
$$
G_\varepsilon(u_n)\ \text{is bounded},\qquad G_\varepsilon'(u_n)\to 0\ \; \text{in }(\rH^1(\mathbb S^2))^*.
$$
We prove that $(u_n)$ admits a strongly convergent subsequence in $\rH^1(\mathbb{S}^2)$.

\textit{Step 1: uniform bounds in $\rH^1$.}\\
Since $ Q_\varepsilon \in \rC(\mathbb S^2)$, we have
$ \| Q_\varepsilon \|_{\infty}<\infty$. Hence, from
$ \vert B(s) \vert  \leq | B_0 | + | s |$, there exist constants $ C_0 , C_1 > 0$
such that, for every $\omega \in \mathbb{S}^2$ and $s \in \mathbb R$,
\begin{equation}
\frac{ |s|^5 }{ 5 }- Q_\varepsilon( \omega) B ( s ) - qs \geq - C_0 - C_1 |s |.
\label{eq: bound with L1 potential sphere}
\end{equation}
Since $\frac{|s|^5}{5}$ dominates the linear term for large $|s|$, there also exist $C_2,C_3>0$ with
$$
\frac{\vert s \vert ^5}{5}-Q_\varepsilon(\omega)B(s)-q s \geq  C_2 \vert s \vert ^5 - C_3\qquad\forall s, \omega.
$$
Integrating and using that $G_\varepsilon(u_n)$ is bounded from above yields a uniform $\rL^5$ bound
\begin{equation}\label{eq:Palais smale sphere L5 bound}
\|u_n\|_{5}\leq M.
\end{equation}
In particular, since $\mathbb{S}^2$ has finite measure, $\rL^5(\mathbb S^2)\hookrightarrow \rL^1(\mathbb S^2)$ and thus
$$
\|u_n\|_{1}\leq C \|u_n\|_{5}\leq C M.
$$

Define the mean
$$
m_n: = \frac{1}{ \vert \mathbb{S}^2 \vert } \int_{\mathbb{S}^2} u_n \d \omega, \qquad v_n : = u_n - m_n,
\qquad \int_{\mathbb{S}^2} v_n \d \omega = 0.
$$
Then $ \vert m_n \vert \leq \frac{1}{ \vert \mathbb{S}^2 \vert} \| u_n \|_{1} \leq C$ uniformly.
By the Poincar\'e inequality on $\mathbb{S}^2$,
$$
\| v_n \|_{ 2}^2 \leq   C_P \| \nabla_{ \mathbb{S}^2} v_n\|_{2}^2
= C_P \| \nabla_{\mathbb{S}^2} u_n\|_{ 2}^2.
$$
Using $\kappa\geq k_*>0$ we have
$$
\| \nabla_{ \mathbb{S}^2 } u_n \|_2^2 \leq \frac{1}{k_*} \int_{ \mathbb{S}^2 } \kappa | \nabla_{\mathbb {S}^2} u_n |^2 \d\omega.
$$
Hence 
\begin{equation}
    \| u_n \|_2^2 \leq 2 \| v_n\|_2^2 + 2| \mathbb{S}^2| \,m_n^2
\leq \frac{2C_P}{k_*} \int_{\mathbb{S}^2} \kappa\,|\nabla_{ \mathbb {S}^2}u_n|^2\d\omega + C.
\label{eq: bound u_n L2 norm sphere}
\end{equation}
Introduce the shorthand
$$
A_n:=\int_{\mathbb S^2}\kappa(\omega)\,|\nabla_{\mathbb S^2}u_n|^2\d \omega, \qquad \|u\|_{1,\kappa}^2:=\|u\|_2^2 + A(u),\ \ A(u):=\int_{\mathbb S^2}\kappa|\nabla u|^2.
$$
Then \eqref{eq: bound u_n L2 norm sphere} can be rewritten as
$$
\|u_n\|_2^2 \leq C_1 A_n + C_2,
$$
and therefore
\begin{equation}
\|u_n\|_{1,\kappa}^2 = \|u_n\|_2^2 + A_n \leq (C_1+1)A_n + C_2,
\qquad\text{i.e.}\qquad
A_n \geq \frac{1}{C_1+1}\left(\|u_n\|_{1,\kappa}^2 - C_2\right).
\label{eq: lower bound A_n}
\end{equation}

On the other hand, integrating in space \eqref{eq: bound with L1 potential sphere}, we obtain that there exist constants $C',C''>0$ such that
$$
G_\varepsilon(u_n)\geq -C' - C''\|u_n\|_{1} + \frac{1}{2} A_n
\geq -C' - C''\|u_n\|_{1,\kappa} + \frac{1}{2} A_n,
$$
where in the second inequality we have used that $  \| u \|_1 \lesssim  \| u \|_2 \lesssim \| u \|_{1,\kappa}$. Combining the previous inequality with \eqref{eq: lower bound A_n}, we arrive at
$$
G_\varepsilon(u_n)\geq -C' - C''\|u_n\|_{1,\kappa} + \frac{1}{2(C_1+1)}\left(\|u_n\|_{1,\kappa}^2 - C_2\right).
$$
Thus there is a quadratic polynomial $P(X)=\alpha X^2 - \beta X - \gamma$, with $\alpha>0$, such that
$$
G_\varepsilon(u_n)\geq P(\|u_n\|_{1,\kappa}).
$$
Since $\{G_\varepsilon(u_n)\}_n$ is bounded, the values $\|u_n\|_{1,\kappa}$ must remain bounded. Therefore $\{u_n\}$ is bounded in $H^1(\mathbb{S}^2)$.

\textit{Step 2: compactness.}\\
From boundedness in $\rH^1(\mathbb{S}^2)$ and the compact embeddings
$\rH^1(\mathbb{S}^2)\hookrightarrow \rL^2(\mathbb{S}^2)$ and $\rH^1(\mathbb {S}^2)\hookrightarrow \rL^5(\mathbb{S}^2)$,
extract a subsequence, that we do not rename, such that
\begin{equation}
\begin{aligned}
u_n       \rightharpoonup  u \;   \text{in } \rH^1(\mathbb{S}^2),    \qquad
u_n \to    u   \;    \text{in } \rL^2 ( \mathbb{S}^2)\; \text{and in } \rL^5( \mathbb{S}^2), \quad
u_n \to u \ \text{a.e.}
\end{aligned}
\label{eq: convergence norms PS proof sphere}
\end{equation}

\textit{Step 3: strong convergence in $\rH^1$.}\\
For $w,\varphi\in \rH^1 ( \mathbb{S}^2)$, the Gateaux derivative of $G_\varepsilon$ is
$$
\langle G_\varepsilon'(w),\varphi\rangle
=\int_{\mathbb{S}^2}\left(|w|^3w - Q_\varepsilon(\omega) \beta(w) - q\right)\varphi\d\omega
+\int_{\mathbb{S}^2}\kappa(\omega) \nabla_{\mathbb {S}^2}w\cdot \nabla_{\mathbb{S}^2}\varphi \d \omega.
$$
Since $G_\varepsilon'(u_n)\to 0$ in $(H^1)^*$, choosing $\varphi=u_n-u$ gives
\begin{equation}\label{eq:PS sphere aux1}
\int_{\mathbb S^2}\kappa \nabla u_n\cdot(\nabla u_n-\nabla u)\d\omega =
-\int_{\mathbb{S}^2}\left(\vert u_n \vert ^3u_n - Q_\varepsilon \beta (u_n) - q\right)(u_n-u)\d\omega + o(1).
\end{equation}
We show that the right-hand side tends to $0$.
For the leading term, using H\"older inequality, the boundedness from \eqref{eq:Palais smale sphere L5 bound} and \eqref{eq: convergence norms PS proof sphere}, we have
$$
\left|\int_{\mathbb{S}^2} \vert u_n \vert ^3u_n\,(u_n-u)\d\omega\right|
\leq \|u_n\|_{5}^4 \|u_n-u\|_{5} \to  0.
$$
For the lower order term, $\beta $ is bounded and $Q_\varepsilon\in \rL^\infty$, hence
$Q_\varepsilon \beta(u_n)+q$ is bounded in $\rL^\infty(\mathbb S^2)$ and
$$
\left| \int_{ \mathbb{S}^2}(Q_\varepsilon  \beta ( u_n ) + q)(u_n-u) \d \omega \right| \leq C \| u_n - u \|_{ 1}
\leq C \vert \mathbb{S}^2 \vert ^{1/2} \| u_n - u \|_{2} \to 0,
$$
where the convergence follows again by \eqref{eq: convergence norms PS proof sphere}. Therefore the right-hand side of \eqref{eq:PS sphere aux1} converges to $0$, and we obtain
$$
\int_{\mathbb{S}^2} \kappa \nabla u_n \cdot( \nabla u_n- \nabla u)\d\omega \to 0.
$$
But
$$
\int_{\mathbb{S}^2}\kappa \nabla u_n\cdot(\nabla u_n-\nabla u)\d\omega
=\|\sqrt{\kappa} \nabla u_n\|_2^2 - \int_{\mathbb S^2}\kappa\,\nabla u_n\cdot\nabla u\d\omega.
$$
Using the weak convergence $\sqrt{\kappa}    \nabla u_n\rightharpoonup \sqrt{\kappa} \nabla u$ in $\rL^2$ (which is a consequence of \eqref{eq: convergence norms PS proof sphere} and $\kappa \in \rL^\infty$, $\kappa \geq \kappa_* >0$),
$$
\int_{\mathbb{S}^2}\kappa \nabla u_n\cdot\nabla u\d \omega
\to  \int_{\mathbb {S}^2}\kappa \vert\nabla u\vert ^2\d\omega=\|\sqrt{\kappa}\nabla u\|_2^2.
$$
Hence
$$
\limsup_{n\to\infty}\|\sqrt{\kappa} \nabla u_n\|_2^2 \leq \|\sqrt {\kappa} \nabla u\|_2^2.
$$
By weak lower semicontinuity we also have
$$
\|\sqrt{\kappa} \nabla u\|_2^2 \leq \liminf_{n\to\infty}\|\sqrt{\kappa} \nabla u_n \|_2^2,
$$
so $\| \sqrt{\kappa } \nabla u_n \|_2 \to \| \sqrt{\kappa } \nabla u \|_2$. Combining this with $ \sqrt{\kappa } \nabla u_n  \rightharpoonup  \sqrt{\kappa } \nabla u $ in $L^2$, we arrive at
$$
\sqrt{\kappa}  \nabla u_n \to \sqrt{\kappa}  \nabla u \qquad\text{in }\rL^2(\mathbb{S}^2).
$$
Since $\kappa\ge k_*>0$, this implies
$$
\| \nabla_{\mathbb S^2}(u_n-u) \|_{2}^2 \leq \frac{1}{k_*}
\|  \sqrt{\kappa}\,     \nabla_{\mathbb{S}^2}(u_n-u)\|_{2}^2 \to 0.
$$
Together with $u_n \to u$ in $\rL^2( \mathbb{S}^2 )$ from \eqref{eq: convergence norms PS proof sphere}, this gives
$u_n \to u$ strongly in $\rH^1(\mathbb{S}^2)$.
Thus every Palais-Smale sequence admits a strongly convergent subsequence.
\end{proof}

\subsubsection{Perturbative saddle and mountain-pass critical point}
\noindent
We first analyse the homogeneous problem with constant insolation \(Q_*>0\),
for which the reference equilibria are constant. We then study their
persistence under the heterogeneous perturbation of the insolation.

We recall that in \autoref{subsec: main results} we have introduced the elliptic operator
$$
A_\kappa v:=
-\div_{\mathbb{S}^2}\left(\kappa ( \omega) \nabla_{\mathbb{S}^2}v \right).
$$
Moreover, we denote by $\lambda_1^\kappa$ the first positive eigenvalue of $A_\kappa$,
namely
\begin{equation}
    \lambda_1^\kappa =
\inf \left\lbrace  \frac{\int_{\mathbb{S}^2}\kappa(\omega) \abs{ \nabla_{\mathbb{S}^2}v}^2 \d \omega}
{\int_{\mathbb{S}^2}v^2 \d \omega}
: v \in \rH^1(\mathbb{S}^2) , \ v \neq 0,\ 
\int_{\mathbb{S}^2}v \d\omega = 0
\right \rbrace .
\label{eq: definition lambda 1 kappa sphere}
\end{equation}

The following lemma translates the scalar curvature and spectral conditions
in \autoref{ass:homogeneous-three-well} into the variational classification
of the three homogeneous equilibria on the sphere.
\begin{lem}[Homogeneous equilibria and Morse indices]
\label{lem:homogeneous-equilibria-morse-sphere}
Assume \autoref{ass:homogeneous-three-well}.  Then the constant functions
$$
u_S^0 \equiv \theta_S , \qquad
u_M^0 \equiv\theta_M , \qquad
  u_W^0\equiv \theta_W
$$
are critical points of $G_{0}$. Moreover, $u_S^0$ and $u_W^0$
are strict local minimisers, while $u_M^0$ is non-degenerate and has Morse
index one.
\end{lem}
\begin{proof}
Assume for simplicity that $\mathfrak{e} =1$. Since $R'(\theta_i)=0$, we have
$$
\theta_i\abs{\theta_i}^3-Q_* \beta(\theta_i)-q=0.
$$
Because the functions $u_i^0\equiv\theta_i$ are constant, their gradients
vanish. Therefore each $u_i^0$ satisfies the stationary equation
$$
-\div_{\mathbb S^2} \left( \kappa\nabla_{\mathbb S^2}u_i^0\right)
+
u_i^0\abs{u_i^0}^3 - Q_*\beta(u_i^0) - q = 0.
$$
Thus $u_i^0$ is a critical point of $ G_{0}$. The second variation at $u_i^0$ is
$$
D^2  G_{0}(u_i^0)[v,v] = \int_{\mathbb S^2}\kappa(\omega)\abs{\nabla_{\S^2}v}^2\d \omega + \mu_i\int_{\mathbb S^2}v^2\d \omega.
$$
If $i=S$ or $i=W$, then $\mu_i>0$. Hence
$$
D^2 G_{0}(u_i^0)[v,v]
\geq c_i \| v \|_{\rH^1(\mathbb{S}^2)}^2
$$
for some $c_i>0$. Therefore $u_S^0$ and $u_W^0$ are strict local
minimisers.

It remains to analyse $u_M^0$. Decompose
$$
v=\bar v+v_0, \qquad \bar v:=\frac{1}{\abs{\mathbb S^2}}\int_{\mathbb S^2}v\d \omega, \qquad \int_{\mathbb S^2}v_0\d \omega=0.
$$
On the constant direction $v\equiv1$, the second variation is negative,
because $\mu_M<0$. On the other hand, if $\int_{\mathbb S^2}v\d \omega=0$,
then by the definition of $\lambda_1^\kappa$ in \eqref{eq: definition lambda 1 kappa sphere}, it holds
$$
\int_{\mathbb S^2}\kappa\abs{\nabla_{\mathbb S^2}v}^2\d \omega + \mu_M\int_{\mathbb S^2}v^2\d \omega
\geq (\lambda_1^\kappa+\mu_M) \int_{\mathbb S^2}v^2\d \omega.
$$
Since $-\lambda_1^\kappa<\mu_M$, the right-hand side is strictly positive for
any $v \neq 0$, with $\int_{\mathbb{S}^2} v \, d \omega  =0$. Thus the negative spectral subspace is one-dimensional and is generated by the constant functions. It follows that $u_M^0$ has Morse
index one. The same argument also shows that the kernel of the linearised
operator is trivial, so $u_M^0$ is non-degenerate.
\end{proof}
\noindent 
We now perturb the insolation. Let $Q_1\in \rC(\mathbb S^2)$, and for
$\varepsilon \in \mathbb{R}$ set
$$
Q_\varepsilon(\omega):=Q_*+\varepsilon Q_1(\omega).
$$
For $\abs{\varepsilon}$ sufficiently small, $Q_\varepsilon$ remains
strictly positive on $\mathbb S^2$. The implicit function theorem and continuity of the Hessians allow us to prove the
persistence of the three equilibria and their Morse indices.

\begin{thm}[Persistence of the three-climate structure]
\label{thm: three climate sphere}
Assume \autoref{ass:homogeneous-three-well}. Let
$Q_1\in \rC(\mathbb{S}^2)$, and consider
$
Q_\varepsilon( \omega ):=Q_*+\varepsilon Q_1(\omega).
$ Then there exists $\varepsilon_0>0$ such that, for every
$| \varepsilon |<\varepsilon_0$, the stationary EBM
$$
-\div_{\mathbb{S}^2} \left( \kappa( \omega ) \nabla_{\mathbb S^2}u\right) + \mathfrak{e}u|u|^3 - Q_\varepsilon ( \omega )\beta(u) - q = 0
$$
admits three solutions
$$
u_S^\varepsilon, \qquad u_M^\varepsilon, \qquad u_W^\varepsilon
$$
belonging to $D(A_\kappa) = H^2(\mathbb{S}^2)$.  Moreover, for each $i \in \{S,M,W\}$,
$$
u_i^\varepsilon \to \theta_i \qquad
\text{in }D(A_\kappa)
\quad\text{as }\varepsilon\to0.
$$
The equilibria $ u_S^\varepsilon$ and $u_W^\varepsilon $ are strict local
minimisers of the energy functional $G_\varepsilon$, while
$u_M^\varepsilon$ is non-degenerate and has Morse index one. In particular,
for every sufficiently small $| \varepsilon | $, the spherical EBM has two
stable variational equilibria and one unstable Morse-index-one transition
equilibrium.
\end{thm}
\begin{proof}
Assume for simplicity that $\mathfrak{e} = 1$ and fix $i\in\{S,M,W\}$. Let $A_\kappa$ denote the $\rL^2(\mathbb S^2)$-realisation of
$
-\div_{\S^2} \left(\kappa\nabla_{\mathbb{S}^2}\cdot\right).
$
We recall that, see \autoref{rem: H2 = dom A_k sphere} for more details, $
D(A_\kappa)=\rH^2(\mathbb S^2),
$
with equivalent graph and $\rH^2$ norms. Define
$$
\mathcal{F}: \mathcal{U} \subset \mathbb{R} \times D(A_\kappa) \to \rL^2(\mathbb{S}^2), \quad \mathcal F(\varepsilon, u) := A_\kappa u + u|u|^3 - Q_\varepsilon(\omega) \beta(u) - q.
$$
where $\mathcal U$ is an open neighbourhood of
$(0,\theta_i)$, and $D(A_\kappa)$ is endowed with the graph norm. Since
$
D(A_\kappa)=\rH^2(\S^2) \hookrightarrow \rL^\infty(\mathbb S^2),
$
the non-linear terms are well-defined in $\rL^2(\mathbb S^2)$. Moreover,
because $Q_1\in \rC(\mathbb S^2)\subset \rL^\infty(\mathbb S^2)$ and
$\beta \in\rC_b^1(\mathbb R)$, the map $\mathcal F$ is of class
$C^1$ near $(0,\theta_i)$.

\textit{Step 1: Linearisation and invertibility of the operator.}
By the definition of $\theta_i$, namely $R'(\theta_i)=0$, we have
$$
\theta_i|\theta_i|^3-Q_* \beta(\theta_i)-q=0.
$$
Hence
$\mathcal{F}(0, \theta_i ) = 0. $ The derivative with respect to $u$ at $(0,\theta_i)$ is
$$
D_u\mathcal F(0,\theta_i)v
=
A_\kappa v+\mu_i v,
$$
where
$
\mu_i=4\theta_i^3-Q_* \beta'(\theta_i).
$ We claim that $A_\kappa+\mu_i:
D(A_\kappa)\to \rL^2(\mathbb S^2)$ is an isomorphism. Indeed, $A_\kappa$ is self-adjoint and non-negative on
$\rL^2(\mathbb S^2)$, with compact resolvent. Its kernel consists of the
constant functions, and its first positive eigenvalue is
$\lambda_1^\kappa$.

If $i=S$ or $i=W$, then $\mu_i>0$. Hence
$A_\kappa+\mu_i$ is strictly positive and therefore invertible. If
$i=M$, then $ -\lambda_1^\kappa<\mu_M<0.$ Thus $-\mu_M$ is not an eigenvalue of $A_\kappa$, and
$A_\kappa+\mu_M$ is invertible as an operator from $D(A_\kappa)$ to
$\rL^2(\mathbb S^2)$.

\textit{Step 2: implicit function theorem and second variation.}
The implicit function theorem gives, for each $i\in\{S,M,W\}$, a unique
$C^1$ branch
$$
\varepsilon \mapsto u_i^\varepsilon\in D(A_\kappa),
$$
defined for $\vert \varepsilon \vert $ sufficiently small, such that
$$
u_i^0\equiv\theta_i, \qquad
\mathcal{F}(\varepsilon, u_i^\varepsilon)=0.
$$
Therefore $u_i^\varepsilon$ solves the stationary equation
$$
-\div _{\mathbb{S}^2} (\kappa(\omega) \nabla_{\mathbb{S}^2} u_i^\varepsilon ) + u_i^\varepsilon |u_i^\varepsilon|^3 - Q_\varepsilon( \omega )\beta(u_i^\varepsilon) -q = 0.
$$
Moreover
$$
u_i^\varepsilon\to\theta_i \qquad \text{in }D(A_\kappa) \quad\text{as }\varepsilon\to0.
$$
Since $D(A_\kappa)=\rH^2(\mathbb S^2)$ with equivalent norms, this is
equivalent to convergence in $H^2(\mathbb S^2)$. In particular, by the
embedding $\rH^2(\mathbb S^2)\hookrightarrow \rC^0(\mathbb S^2)$, we also have
$
u_i^\varepsilon\to\theta_i$ in $\rC^0(\mathbb{S}^2).$

The second variation of $G_{\varepsilon}$ at
$u_i^\varepsilon$ is
$$
D^2  G_{\varepsilon}(u_i^\varepsilon)[v,v] =
\int_{\mathbb{S}^2} \kappa|\nabla_{\mathbb{S}^2}v|^2 \, d\omega + \int_{\mathbb{S}^2}
a_i^\varepsilon(\omega)v^2\,d\omega,
$$
where
$
a_i^\varepsilon(\omega) := 4| u_i^\varepsilon ( \omega)|^3 - Q_\varepsilon(\omega) \beta'\left( u_i^\varepsilon ( \omega ) \right).
$
Since $u_i^\varepsilon\to\theta_i$ in $\rC^0(\mathbb{S}^2)$ and 
$Q_\varepsilon\to Q_*$ in $\rL^\infty(\mathbb S^2)$, while
$\beta'$ is continuous and bounded, we have
$$
a_i^\varepsilon \to \mu_i \qquad \text{in }\rL^\infty(\mathbb{S}^2)
\quad\text{as }\varepsilon\to0.
$$
Therefore, by H\"older inequality we have, for every $v\in H^1(\mathbb S^2)$,
$$
\left| \int_{\mathbb{S}^2} \left( a_i^\varepsilon- \mu_i\right)v^2\d \omega \right| \leq \|a_i^\varepsilon-\mu_i\|_{\infty} \| v \|_{2}^2.
$$
Thus the perturbed second variations are small quadratic-form perturbations
of the homogeneous second variations.

\textit{Step 3: Persistence of the variational character.}
For $i= S , W$, the homogeneous second variation is coercive on
$\rH^1(\mathbb{S}^2)$ thanks to \autoref{ass:homogeneous-three-well}. Hence, thanks to second part of the previous step, after possibly reducing $\varepsilon_0$, there
exists $c_i>0$ such that
$$
D^2 G_{\varepsilon}(u_i^\varepsilon)[v,v] \geq c_i \| v \|_{\rH^1(\mathbb{S}^2)}^2 \qquad \forall v\in \rH^1(\S^2).
$$
Therefore $u_S^\varepsilon$ and $u_W^\varepsilon$ are strict local
minimisers of ${G}_{ \varepsilon}$.

For $i=M$, the homogeneous quadratic form has exactly one negative
direction and no kernel by \autoref{lem:homogeneous-equilibria-morse-sphere}. Since
$\|a_M^\varepsilon-\mu_M\|_\infty  \to 0$, the minimax
principle implies that the perturbed quadratic form also has exactly one
negative direction and no kernel for all sufficiently small
$\vert \varepsilon \vert  $. Hence $u_M^\varepsilon$ is non-degenerate and has Morse
index one.
\end{proof}

\begin{cor}[Mountain-pass critical point between the stable branches on $\mathbb S^2$]
\label{cor:MP point between min sphere}
Assume the hypotheses of \autoref{thm: three climate sphere}. Then, for every sufficiently small
$ \abs{\varepsilon} $, the functional $G_{\varepsilon}$ admits a
mountain-pass critical point $u_{MP}^\varepsilon\in \rH^1(\mathbb S^2)$ between the two strict local minimisers $u_S^\varepsilon$ and
$u_W^\varepsilon$.

More precisely, let $
\Gamma_\varepsilon := \left\lbrace 
 \gamma \in \rC([0,1]; \rH^1(\mathbb{S}^2)):
\gamma(0) = u_S^\varepsilon,\  \gamma(1) = u_W^\varepsilon
\right\rbrace ,$ and define the minimax level $
c_\varepsilon := \inf_{\gamma \in \Gamma_\varepsilon} \sup_{t\in[0,1]}
G_{  \varepsilon}(\gamma(t)).$ Then there exists $u_{MP}^\varepsilon\in H^1(\mathbb S^2)$ such that
$$
G_{\varepsilon}'(u_{MP}^\varepsilon)=0, \qquad  G_{\varepsilon}(u_{MP}^\varepsilon)=c_\varepsilon.
$$
Moreover $ c_\varepsilon> \max \left\lbrace  G_{\varepsilon}(u_S^\varepsilon), G_{  \varepsilon}(u_W^\varepsilon) \right\rbrace ,$ and therefore $u_{MP}^\varepsilon \notin \{ u_S^\varepsilon, u_W^\varepsilon\}.$
\end{cor}
\begin{proof}
By \autoref{thm: three climate sphere}, the second variation at $u_S^\varepsilon$ and $u_W^\varepsilon$ is coercive. Hence, by the
$\rC^2$ regularity of the functional in neighbourhoods of these points, there
exist $r_S,r_W>0$ such that
$$
B_{r_S}(u_S^\varepsilon)\cap B_{r_W}(u_W^\varepsilon) = \emptyset,
$$
and constants $\alpha_S,\alpha_W>0$ such that
$$
G_{\varepsilon}(u) \geq  G_{\varepsilon}(u_S^\varepsilon)+\alpha_S \qquad
\text{for every }u\in\partial B_{r_S}(u_S^\varepsilon),
$$
and
$$
G_{\varepsilon}(u) \geq G_{\varepsilon}(u_W^\varepsilon)+\alpha_W \qquad \text{for every }u\in\partial B_{r_W}(u_W^\varepsilon).
$$
Every continuous path connecting $u_S^\varepsilon$ to $u_W^\varepsilon$ must cross
both boundaries $\partial B_{r_S}(u_S^\varepsilon)$ and
$\partial B_{r_W}(u_W^\varepsilon)$. Then it must hold
$$
c_\varepsilon> \max \left\lbrace  G_{\varepsilon}(u_S^\varepsilon), G_{\varepsilon}(u_W^\varepsilon)
\right\rbrace .
$$

For $\abs{\varepsilon}$ sufficiently small, $Q_\varepsilon$ is continuous and
strictly positive. Therefore $G_{\varepsilon}$ satisfies the
Palais-Smale condition by \autoref{prop: Palais Smale sphere}. The
Mountain-Pass Theorem, applied in its two-local-minima formulation, gives a
critical point $u_{MP}^\varepsilon\in \rH^1(\mathbb S^2)$ at the level
$c_\varepsilon$. The strict inequality above shows that
$u_{MP}^\varepsilon$ is distinct from the two local minimisers.
\end{proof}

\subsection{Identification of the spherical perturbative saddle with the mountain-pass point}
\label{subsec: three climate sphere}
\noindent
The perturbative middle branch and the mountain-pass point are constructed by different methods, so they do not automatically coincide. To identify them, we use, as in the one-dimensional case, the \autoref{ass:homogeneous-mountain-pass-isolation} on the homogeneous problem. It states the middle (homogeneous) equilibrium $u^0_M$attains the minimax level and is the only critical point at that level.

The proof of identification has then three steps: convergence of the minimax
levels, convergence of the corresponding mountain-pass points in $\rH^1(\mathbb{S}^2)$, and an upgrade to the operator-domain topology.
\begin{lem}[Continuity of the spherical mountain-pass level]
\label{lem: MP levele continuity}
Assume the hypotheses of \autoref{thm: three climate sphere}. Let $c_\varepsilon$ be the
mountain-pass level defined in \autoref{cor:MP point between min sphere}, and let $c_0$ be the homogeneous
mountain-pass level defined in \eqref{eq:main-results-unified-homogeneous-minimax-level}. Then
$$
c_\varepsilon \to c_0 \qquad \text{as } \varepsilon \to 0.
$$
\end{lem}
\begin{proof}
We first prove that $\limsup_{\varepsilon\to0}c_\varepsilon\leq c_0.$ Let $\eta>0$. By the definition of $c_0$,
there exists $\gamma_0\in\Gamma_0$ such that
$$
\sup_{t\in[0,1]} G_{ 0}(\gamma_0(t)) \leq c_0 + \eta.
$$
Since, thanks to \autoref{thm: three climate sphere}, it holds
$$
u_S^\varepsilon\to u_S^0, \qquad u_W^\varepsilon\to u_W^0 \qquad
\text{in } \rH^1(\mathbb{S}^2),
$$
we can concatenate $\gamma_0$ with the short line segments joining
$u_S^\varepsilon$ to $u_S^0$ and $u_W^0$ to $u_W^\varepsilon$. This gives a
path $\gamma_\varepsilon \in \Gamma_\varepsilon$. Along the fixed path $\gamma_0$, we have
\begin{equation}
    G_{  \varepsilon} ( u ) -G_{0}(u) = -\varepsilon \int_{\mathbb{S}^2} Q_1( \omega ) B(u (\omega))\d\omega.
\label{eq: difference G eps 0 and non zero}
\end{equation}
Since $B$ has at most linear growth and $\gamma_0([0,1])$ is compact in
$\rH^1(\mathbb{S}^2)$, the right-hand side is $o(1)$ uniformly for
$u\in\gamma_0([0,1])$. Hence
$$
\sup_{t\in[0,1]} G_{\varepsilon}(\gamma_0(t)) \leq c_0 + \eta+o(1).
$$
On the two short endpoint segments, the energies converge to the corresponding
endpoint energies. Since $u_S^0$ and $u_W^0$ are strict local minimisers and
$$
c_0> \max\{G_{0}(u_S^0), G_{0}(u_W^0)\},
$$
the maximum energy on these short segments is strictly below $c_0+\eta$ for
$\abs{\varepsilon}$ sufficiently small. Therefore
$$
\limsup_{\varepsilon \to 0} c_\varepsilon \leq c_0 + \eta.
$$
Letting $\eta\downarrow0$, we obtain $
\limsup_{\varepsilon\to0}c_\varepsilon\leq c_0.$

We now prove that $\liminf_{\varepsilon\to0}c_\varepsilon\geq c_0.$ Let $\varepsilon_n \to 0$, and choose
$\gamma_n \in \Gamma_{\varepsilon_n}$ such that
$$
\sup_{t \in [0,1]} G_{ \varepsilon_n} (\gamma_n(t)) \leq  c_{\varepsilon_n} + \frac{1}{n}.
$$
The upper bound just proved implies that $(c_{\varepsilon_n})_n$ is bounded
from above. Hence
$$
\sup_{t\in[0,1]} G_{ \varepsilon_n }(\gamma_n(t)) \leq C
$$
for some constant $C>0$ independent of $n$.

We claim that the paths $\gamma_n$ are uniformly bounded in
$H^1(\mathbb S^2)$. Indeed, since
$Q_{\varepsilon_n}=Q_*+\varepsilon_nQ_1$ and $Q_1\in \rC(\mathbb{S}^2)$, the
quantities $\| Q_{\varepsilon_n} \|_{\infty }$ are uniformly
bounded for $n$ large. Using the growth bound
$\vert B(s) \vert \leq |B_0|+|s|$, there exist constants $c,C>0$, independent of
$n$, such that
$$
G_{  \varepsilon_n}(u) \geq \frac{1}{2} \int_{ \mathbb{S}^2} \kappa \vert  \nabla_{ \mathbb{S}^2} u \vert ^2\d \omega + c \| u \|_{5}^5 - C
$$
for every $u \in \rH^1(\mathbb{S}^2)$. Applying this to
$ u = \gamma_n(t)$ gives uniform bounds on
$$
\int_{\mathbb{S}^2}  \kappa| \nabla_{\mathbb{S}^2} \gamma_n(t)|^2 \d \omega
\qquad \text{and} \qquad \| \gamma_n(t) \|_{5}.
$$
Since $\kappa \geq k_*>0$ and
$L^5(\mathbb{S}^2) \hookrightarrow L^2(\mathbb{S}^2)$, there exists
$R>0$ such that
$$
\|\gamma_n(t)\|_{H^1(\mathbb{S}^2)} \leq R,
$$
for every $t \in [0,1]$ and every $n$. Further, on the ball $\{u\in \rH^1(\mathbb S^2):\|u\|_{\rH^1}\leq R\}$, it holds that, thanks to \eqref{eq: difference G eps 0 and non zero} and H\"older inequality
$$
\left|  G_{\varepsilon_n}(u)- G_{0}(u) \right| \leq |\varepsilon_n| \|Q_1\|_{\infty} \int_{\mathbb{S}^2}|B(u)|\d \omega
\leq C_R| \varepsilon_n|.
$$
Consequently, as $n \to +\infty$, we have
$$
\sup_{t\in[0,1]} \left| G_{  \varepsilon_n}(\gamma_n(t)) - G_{0}(\gamma_n(t)) \right| \to 0.
$$

Now concatenate $\gamma_n$ with the short line segments joining
$u_S^0$ to $u_S^{\varepsilon_n}$ and
$u_W^{\varepsilon_n}$ to $u_W^0$. This produces a path
$\tilde{\gamma}_n\in \Gamma_0 $. The contribution of the short endpoint segments,
computed with the homogeneous functional $ G_{0}$, has maximum
strictly below $c_0$ for $n$ sufficiently large, because the endpoints converge
to the strict local minimisers $u_S^0$ and $u_W^0$. Therefore
$$
c_0 \leq \sup_{t\in[0,1]} G_{0}(\tilde\gamma_n(t)) \leq
\max\left\{ c_{\varepsilon_n}+o(1),
c_0-\delta \right\}
$$
for some $\delta>0$ and all sufficiently large $n$. It follows that
$
c_0\leq c_{\varepsilon_n}+o(1).
$
Thus
$
c_0 \leq   \liminf_{n\to\infty}c_{\varepsilon_n}.
$
Since the sequence $\varepsilon_n \to0$ is arbitrary, we conclude that
$
\liminf_{\varepsilon\to0} c_\varepsilon \geq c_0.
$
\end{proof}
\noindent
We next use the continuity of the minimax levels and the Palais-Smale
condition to identify the possible limits of the perturbed mountain-pass
points.
\begin{lem}[Convergence of spherical mountain-pass critical points]
\label{lem:mp-points-converge-sphere}
Assume the hypotheses of \autoref{thm: three climate sphere} and \autoref{ass:homogeneous-mountain-pass-isolation}. Let
$\varepsilon_n\to0$, and let
$u_n:=u_{MP}^{\varepsilon_n}$ be mountain-pass critical points satisfying
$$
G _{ \varepsilon_n}'(u_n)=0, \qquad G_{ \varepsilon_n}(u_n) = c_{\varepsilon_n}.
$$
Then, up to a subsequence,
$$
u_n \to u_M^0 \qquad \text{strongly in }\rH^1(\mathbb S^2).
$$
\end{lem}
\begin{proof}
By \autoref{lem: MP levele continuity}, we have $
c_{\varepsilon_n}\to c_0.$ Since $
G_{\varepsilon_n}(u_n)=c_{\varepsilon_n},
$ the energy levels $G_{\varepsilon_n}(u_n)$ are bounded. The uniform
coercive estimate used in the proof of \autoref{lem: MP levele continuity} therefore gives that $(u_n)_n$ is bounded
in $\rH^1(\mathbb S^2)$.

We first show that $(u_n)_n$ is a Palais-Smale sequence for the homogeneous
functional $G_{0}$ at level $c_0$. Since
$$
G_{0}( u_n )-G_{  \varepsilon_n}( u_n) = \varepsilon_n \int_{\mathbb{ S}^2} Q_1(\omega) B(u_n(\omega)) \d \omega,
$$
and since $B$ has at most linear growth while $(u_n)_n$ is bounded in
$\rH^1(\mathbb{S}^2)$, we get
$
G_{ 0}( u_n)-G_{ \varepsilon_n} (u_n) \to 0.
$ Therefore $ G_{0}( u_n ) \to c_0.
$ It remains to check the derivative. For every $v\in \rH^1(\mathbb S^2)$,
$$
\left \langle G_{  0}'( u_n)  - G_{ \varepsilon_n}'(u_n) , v \right \rangle
= \varepsilon_n \int_{\mathbb{S}^2} Q_1(\omega) \beta(u_n)v\d \omega.
$$
Since $\beta$ is bounded and $Q_1\in \rL^\infty( \mathbb{S}^2 )$, we have
$$
\left| \left \langle G_{0}'(u_n) -  G_{\varepsilon_n}'(u_n), v \right \rangle \right|
\leq C | \varepsilon_n| \| v \|_{2} \leq C| \varepsilon_n |\| v \|_{\rH^1(\mathbb{S}^2)}.
$$
Since $G_{  \varepsilon_n}'(u_n)=0$, it follows that
$$
G_{0}'( u_n) \to 0 \qquad \text{in }(\rH^1(\mathbb S^2))^*.
$$
Thus $(u_n)_n$ is a Palais-Smale sequence for $G_{0}$ for the level $c_0$.

By the Palais-Smale condition for $ G_{0}$ on $\rH^1(\mathbb S^2)$, after passing to a subsequence there exists $u_* \in \rH^1 (\mathbb{S}^2)$ such
that
$
u_n\to u_*$ in $\rH^1(\mathbb S^2)$. Since $G_{  0 }$ is $\rC^1$ on $\rH^1(\mathbb S^2)$, we obtain
$$
G_{0}'(u_*)=0, \qquad G_{ 0}( u_*) = c_0.
$$
By \autoref{ass:homogeneous-mountain-pass-isolation}, the only critical point of $ G_{0}$ at level $c_0$ is $u_M^0$. Hence $ u_*=u_M^0. $ In conclusion $ u_n\to u_M^0 $
in $\rH^1(\mathbb S^2).$
\end{proof}
\noindent
To apply the local uniqueness supplied by the implicit function theorem, we
upgrade the convergence from $\rH^1(\mathbb{S}^2)$ to $D(A_\kappa)$.
\begin{lem}[Upgrade of critical-point convergence on $\mathbb S^2$]
\label{lem:critical-convergence-upgrade-sphere}
Let $\kappa\in \rW^{1,\infty}(\mathbb S^2)$, with
$\kappa\ge k_*>0$. Assume that $\varepsilon_n \to 0$, and let $u_n$ be critical points of
$G_{\varepsilon_n}$. Assume that
$$
u_n\to u_M^0 \qquad \text{in } \rH^1(\mathbb{S}^2).
$$
Then $
u_n\to u_M^0$ in $ D( A_\kappa) = \rH^2(\mathbb{S}^2).$
\end{lem}
\begin{proof}
Assume for simplicity that $\mathfrak{e} = 1$. Since $u_n$ is a critical point of $G_{\varepsilon_n}$, it satisfies
$$
A_\kappa u_n = Q_{ \varepsilon_n } \beta ( u_n ) + q - u_n |u_n|^3 \qquad \text{in }\rL^2(\mathbb{S}^2).
$$
The homogeneous equilibrium $u_M^0$ satisfies
$$
A_\kappa u_M^0 = Q_* \beta (u_M^0)+q-u_M^0|u_M^0|^3 \qquad \text{in }\rL^2(\mathbb{S}^2).
$$
We show that the right-hand sides converge strongly in $\rL^2(\mathbb{S}^2)$.

First, since $Q_{\varepsilon_n}\to Q_*$ in $\rL^\infty(\S^2)$,
$\beta$ is bounded and continuous, and $u_n\to u_M^0$ strongly in
$\rH^1(\mathbb S^2)$, we have
$$
Q_{\varepsilon_n}\beta (u_n) \to
Q_* \beta(u_M^0) \qquad \text{in }\rL^2(\mathbb{S}^2).
$$
Indeed, the term involving $Q_{\varepsilon_n} - Q_*$ is controlled by $\| Q_{\varepsilon_n}-Q_*\|_{\infty}$ and the boundedness of
$\beta$, while the term
$\beta(u_n)-\beta(u_M^0)$ converges to zero in $\rL^2$ by dominated
convergence.

Second, since $\rH^1(\mathbb{S}^2)\hookrightarrow \rL^p(\mathbb{S}^2)$ for any $p < \infty$, we have
$
u_n\to u_M^0
$
in $\rL^8(\mathbb{S}^2)$, which implies
$
u_n|u_n|^3 \to u_M^0|u_M^0|^3 $ in $ \rL^2(\mathbb{S}^2).$ Consequently $
A_\kappa u_n \to A_\kappa u_M^0$ in $\rL^2(\S^2).$ Together with $u_n\to u_M^0$ in $\rL^2(\mathbb S^2)$, this gives
$$
\|u_n-u_M^0\|_{D(A_\kappa)} = \|u_n-u_M^0\|_{L^2(\mathbb S^2)} + \|A_\kappa(u_n-u_M^0)\|_{L^2(\mathbb S^2)} \to 0.
$$
Thus $u_n\to u_M^0$ strongly in $D(A_\kappa)$. Since
$D(A_\kappa)=\rH^2(\mathbb S^2)$, see \autoref{rem: H2 = dom A_k sphere}, this concludes the proof.
\end{proof}
\noindent
The preceding lemmas now allow us to identify the mountain-pass point with
the perturbative middle branch.
\begin{prop}[Identification of the perturbative saddle with the mountain-pass saddle on $\mathbb S^2$]
\label{prop:identification-mp-morse-sphere}
Assume the hypotheses of \autoref{thm: three climate sphere} and \autoref{ass:homogeneous-mountain-pass-isolation}. Then,
for every
sufficiently small $|\varepsilon|$, it holds
$$
u_{MP}^\varepsilon=u_M^\varepsilon,
$$
where $u_{MP}^\varepsilon$ is the mountain-pass point from \autoref{cor:MP point between min sphere} and $u_M^\varepsilon$ is the Morse-index one equilibrium from \autoref{thm: three climate sphere}.
\end{prop}
\begin{proof}
We argue by contradiction. Suppose that there exist $\varepsilon_n\to0$ and
mountain-pass critical points $u_{MP}^{\varepsilon_n}$ such that $
u_{MP}^{\varepsilon_n}\neq u_M^{\varepsilon_n} $
for every $n$. Set $u_n:=u_{MP}^{\varepsilon_n}. $
Then
$$
 G_{\varepsilon_n}'(u_n)=0,
\qquad  G_{\varepsilon_n}(u_n)=c_{\varepsilon_n}.
$$
By \autoref{lem:mp-points-converge-sphere}, after passing to a subsequence,
we have
$
u_n\to u_M^0$ in $\rH^1(\S^2).$ Since each $u_n$ is a critical point of $ G_{\varepsilon_n}$, \autoref{lem:critical-convergence-upgrade-sphere} gives the stronger convergence
$$
u_n \to u_M^0 \qquad \text{strongly in }D(A_\kappa).
$$
On the other hand, by the implicit function theorem applied in the proof of \autoref{thm: three climate sphere} also guarantees that there exist a
neighbourhood $\mathcal{U}$ of $u_M^0$ in $D(A_\kappa)$ and a number
$\varepsilon_1 > 0$ such that, for every $\abs{\varepsilon} < \varepsilon_1$, the
only critical point of $G_{\varepsilon}$ in $\mathcal U$ is $u_M^\varepsilon$. Since $u_n \to u_M^0$ in $ D(A_\kappa)$, we have
$u_n\in\mathcal{U}$ for all sufficiently large $n$. Therefore
$$
u_{MP}^{\varepsilon_n} = u_M^{\varepsilon_n}
$$
for all sufficiently large $n$, contradicting the choice of the sequence.
\end{proof}

\begin{rem}
\autoref{ass:homogeneous-mountain-pass-isolation} is a global
assumption. The Morse-index-one property of $u_M^0$ is local and does not, by
itself, imply that $u_M^0$ is the unique critical point at the homogeneous
mountain-pass level. Thus the identification
$u_{MP}^\varepsilon=u_M^\varepsilon$ requires this additional minimax
isolation input.
\end{rem}

\begin{cor}[Perturbative tri-stability on $\mathbb{S}^2$]
Under the assumptions of \autoref{thm: three climate sphere}, for every sufficiently small
$\abs{\varepsilon}$ the functional $ G_{\varepsilon}$ has at least
three distinct critical points in $\rH^1(\mathbb S^2)$. In particular, it has
two strict local minimisers $
u_S^\varepsilon, u_W^\varepsilon,$ and one non-degenerate
Morse-index-one critical point
$
u_M^\varepsilon.
$
In addition, by the
Palais-Smale condition and the Mountain-Pass Theorem, there exists a
mountain-pass critical point
$
u_{MP}^\varepsilon \notin \{ u_S^\varepsilon, u_W^\varepsilon \}
$
If \autoref{ass:homogeneous-mountain-pass-isolation} also holds, then, for all
sufficiently small $\abs{\varepsilon}$,
$
u_{MP}^\varepsilon=u_M^\varepsilon.
$
\end{cor}

\section{Numerical experiments}
\label{sec:numerical section}
\noindent
In this section, we present a numerical illustration of the one-dimensional degenerate problem studied in \autoref{sec: 1D degenerate}. The stationary equation is
\begin{equation}
-\nu \bigl((1-x^2)u'(x)\bigr)'
+\mathfrak e u(x)\abs{u(x)}^3
-Q_\varepsilon(x)\beta(u(x))
-q=0,
\qquad x\in(-1,1),
\label{eq:numerical-dimensional-stationary}
\end{equation}
where $\nu>0$ is the meridional heat-diffusion coefficient. The analytical
part corresponds to the normalised choice $\nu=1$. For any fixed
$\nu>0$, all the preceding arguments remain valid after replacing the
quadratic form $a_\kappa$ and the associated operator $A_\kappa$ by
$\nu a_\kappa$ and $\nu A_\kappa$, respectively. In particular, the
first positive eigenvalue is rescaled from $\lambda_1^\kappa$ to
$\nu\lambda_1^\kappa$. 

In this section, we assume the heterogeneous insolation of the form
\begin{equation}
Q_\varepsilon(x) = Q_* + \varepsilon Q_1(x), \qquad Q_1(x) = \frac{1}{2}(1-3x^2)
 = -P_2 (x),
\label{eq:numerical-insolation-parametrisation}
\end{equation}
where $P_2$ is the second Legendre polynomial, see \autoref{subsec:numerical-method-1d} for more properties. The use of the second
Legendre polynomial to describe the principal equator-to-pole variation of
annual-mean insolation is standard in 1D EBMs, see \cite{North1975,Bonetti2019}. The co-albedo is parametrised by the smooth logistic function (see 
\cite{Kooloth2024})
\begin{equation}
\beta(\theta  )
=
\beta_{\mathrm{low}}
+
\frac{\beta_{\mathrm{high}}-\beta_{\mathrm{low}}}
{1+\exp\!\left(-k_\beta(\theta  -\theta _\beta)\right)}.
\label{eq:numerical-coalbedo-parametrisation}
\end{equation}
 We choose the primitive of $\beta $ satisfying $B(0)=0$, namely
$$
B(\theta )
=
\beta_{\mathrm{low}}\theta 
+
\frac{\beta_{\mathrm{high}}-\beta_{\mathrm{low}}}{k_\beta}
\biggl[
\log\!\left(1+\exp\!\left(k_\beta(\theta -\theta _\beta)\right)\right)
-
\log\!\left(1+\exp(-k_\beta \theta _\beta)\right)
\biggr].
$$
The functional associated to \eqref{eq:numerical-dimensional-stationary} is
\begin{equation}
 G_{\varepsilon}(u) = \int_{-1}^1 \left( \frac{\mathfrak e}{5}\abs{u}^5 -Q_\varepsilon(x)B(u) -qu
\right)\d x + \frac \nu 2\int_{-1}^1(1-x^2)(u')^2 \d x.
\label{eq:numerical-dimensional-energy}
\end{equation}
In the following subsections, we describe how to approximate numerically the three stationary solutions $u_i^\varepsilon$, $i \in \{S,M,W\}$, whose existence and qualitative properties have been established in \autoref{sec: 1D degenerate}. The physical and model parameters used in the numerical experiments in this section are described in
\autoref{tab:numerical-model-parameters}.

\begin{table}[t]
\centering
\small
\renewcommand{\arraystretch}{1.15}
\begin{tabularx}{\textwidth}{@{}l X l l X@{}}
\hline
Symbol & Meaning & Value & Unit  \\
\hline
\(\nu\)
& Meridional diffusivity
& \(0.65\)
& \(\mathrm{W\,m^{-2}K^{-1}}\)
\\

\(\sigma_{\mathrm{SB}}\)
& Stefan--Boltzmann constant
& \(5.67\times10^{-8}\)
& \(\mathrm{W\,m^{-2}K^{-4}}\)
\\

\(\varepsilon_{\mathrm{em}}\)
& Emissivity constant
& \(0.61\)
& None
\\

\(Q_*\)
& Global-mean incoming solar flux
& \(340\)
& \(\mathrm{W\,m^{-2}}\)
\\

\(\bar\varepsilon\)
& Final heterogeneity amplitude
& \(20\)
& \(\mathrm{W\,m^{-2}}\)
\\

\(q\)
& Additional radiative forcing
& \(0\)
& \(\mathrm{W\,m^{-2}}\)
\\

\(\beta_{\mathrm{low}}\)
& Low-temperature co-albedo
& \(0.30\)
& None
\\

\(\beta_{\mathrm{high}}\)
& High-temperature co-albedo
& \(0.72\)
& None
\\

\(k_\beta\)
& Co-albedo transition steepness
& \(0.11\)
& \(\mathrm{K^{-1}}\)
\\

\(\theta _\beta\)
& Co-albedo transition temperature
& \(263\)
& \(\mathrm K\)
\\
\hline
\end{tabularx}
\caption{Physical and model parameters used in the numerical experiments.}
\label{tab:numerical-model-parameters}
\end{table}

\subsection{Legendre-Galerkin discretisation}
\label{subsec:numerical-method-1d}

Let $P_n$ denote the Legendre polynomial of degree $n$ defined recursively by
\begin{equation}
\left \lbrace 
\begin{aligned}
P_0 (x) & = 1, \\
P_1(x)&=x,\\
P_{n+1}(x)
&=
\frac{2n+1}{n+1}xP_n(x)-\frac{n}{n+1}P_{n-1}(x),
\qquad n\geq1.
\label{eq:legendre-three-term-recurrence}
\end{aligned}
\right.
\end{equation}
For the basic properties of Legendre polynomials and their use in spectral
Galerkin methods, we refer to \cite{Shen2011}.
Let $N \geq 1$. We approximate the stationary temperature by
\begin{equation}
u_N(x)=\sum_{n=0}^N c_nP_n(x).
\label{eq:numerical-legendre-expansion}
\end{equation}
The Legendre basis is particularly well suited to the degenerate diffusion
operator in \eqref{eq:numerical-dimensional-stationary}, since
\begin{equation}
-\bigl((1-x^2)P_n'(x)\bigr)'
=
n(n+1)P_n(x),
\qquad n\geq0,
\label{eq:legendre-degenerate-eigenfunctions}
\end{equation}
and
$$
\int_{-1}^1P_n(x)P_m(x)\d x
=
M_n\delta_{nm},
\qquad
M_n=\frac{2}{2n+1}.
$$
Consequently, the diffusion term is diagonal in Legendre coefficient space.
Substituting \eqref{eq:numerical-legendre-expansion} into
\eqref{eq:numerical-dimensional-stationary} and testing the resulting
equation against $P_m$, for $m=0,\ldots,N$, gives the following system of $N+1$
non-linear equations for the $N+1$ coefficients
$c=(c_0,\ldots,c_N)^\top$,
\begin{equation}
\begin{split}
0=R_m(c) :={}& \nu  m(m+1)M_m c_m + \int_{-1}^1 \left(
\mathfrak e u_N\abs{u_N}^3 -Q_\varepsilon(x)\beta(u_N) -q \right)P_m(x)\d x, \qquad 0\leq m\leq N.
\end{split}
\label{eq:numerical-galerkin-system}
\end{equation}
Writing
$
\mathbf R(c) := \bigl(R_0(c),\ldots,R_N(c)\bigr)^\top \in \R^{N+1},
$
the discrete stationary problem in \eqref{eq:numerical-galerkin-system} is therefore reduced to the finite-dimensional
non-linear system
\begin{equation}
\mathbf R(c)=0
\qquad\text{in }\mathbb R^{N+1}.
\label{eq:numerical-vector-residual}
\end{equation}
We approximate the integrals in
\eqref{eq:numerical-galerkin-system} and in \eqref{eq:numerical-dimensional-energy}, by the
$N_q$-point Gauss-Legendre quadrature rule, see \cite[Chapter 1.4]{Gautschi2004}. This rule is exact for polynomials of degree at most $2N_q -1$. More precisely, denoting by
$((x_\ell,w_\ell)_{\ell=1}^{N_q})$ the Gauss-Legendre nodes and weights, then
$$
\int_{-1}^1 f(x)\d x \approx \sum_{\ell=1}^{N_q}w_\ell f(x_\ell).
$$
The values of the Legendre basis functions at the quadrature nodes are precomputed in the 
Vandermonde matrix
$$
V_{\ell n}=P_n(x_\ell), \qquad \ell=1,\ldots,N_q, \quad n=0,\ldots,N.
$$
Hence, if $c=(c_0,\ldots,c_N)^\top$, the vector of nodal values of the
Galerkin approximation is
$$
\bigl(u_N(x_1),\ldots,u_N(x_{N_q})\bigr)^\top=Vc.
$$
To write the quadrature-based residual compactly, introduce
$$
M:=\operatorname{diag}(M_0,\ldots,M_N), \qquad \Lambda:=\operatorname{diag}\bigl(0,2,\ldots,N(N+1)\bigr), \qquad 
W:=\operatorname{diag}(w_1,\ldots,w_{N_q}).
$$
For a nodal vector $T=(T_1,\ldots,T_{N_q})^\top$, define the non-linear
radiative imbalance
$$
\mathcal F_\varepsilon(T)_\ell := \mathfrak e T_\ell\abs{T_\ell}^3
-Q_\varepsilon(x_\ell)\beta(T_\ell)-q.
$$
Thus \eqref{eq:numerical-vector-residual} can be compactly rewritten as
\begin{equation}
\mathbf R(c) = \nu  M\Lambda c + V^\top W\mathcal F_\varepsilon(Vc) =0.
\label{eq:numerical-residual-matrix-form}
\end{equation}

The non-linear system \eqref{eq:numerical-residual-matrix-form} is solved by
a Newton method with backtracking, see for instance \cite[Chapter 6]{Dennis1996}.
Given a current approximation \(c^{(k)}\), the Newton correction
\(\delta^{(k)}\) is obtained by solving
\begin{equation}
J\bigl(c^{(k)}\bigr)\delta^{(k)}
=
-\mathbf R\bigl(c^{(k)}\bigr),
\label{eq:numerical-newton-linear-system}
\end{equation}
where
$
J(c):=D \mathbf{R}(c) = \nu M \Lambda + V^\top W D_\varepsilon(V c) V
$ is the Jacobian matrix of the discrete residual. Here $D_\varepsilon(T)$ is the diagonal matrix defined by
$$
\bigl[ D_\varepsilon(T) \bigr ]_{ \ell \ell} = 4 \mathfrak{e} |T_\ell|^3 - Q_\varepsilon(x_\ell)\beta'(T_\ell).
$$
The
coefficient vector is then updated according to
\begin{equation}
c^{(k+1)} = c^{(k)}+\alpha_k\delta^{(k)}, \qquad 0<\alpha_k\leq1.
\label{eq:numerical-damped-newton-update}
\end{equation}
In the implementation, $\alpha_k$ is selected by a backtracking line search:
the full Newton step $\alpha_k=1$ is tried first and, if it does not decrease
the Euclidean residual norm, the step length is successively halved until
$$
\left| \mathbf R\bigl(c^{(k)}+\alpha_k\delta^{(k)}\bigr) \right|_2 < \left|\mathbf R\bigl(c^{(k)}\bigr)\right|_2.
$$

\subsection{Continuation, discrete energy, and Morse-index computation}
\label{subsec:numerical-continuation-energy-morse}
At $\varepsilon=0$, the insolation is spatially homogeneous,
$Q_\varepsilon\equiv Q_*$, and the constant stationary states are determined
by the scalar radiative-balance equation
$$
\mathfrak e\theta\abs{\theta}^3 -Q_*\beta(\theta)-q=0.
$$
For the parameters considered in \autoref{tab:numerical-model-parameters}, this equation has three positive roots $
\theta_S<\theta_M<\theta_W,$ corresponding to the cold, middle, and warm homogeneous states. Since
$P_0\equiv1$, their Legendre coefficient vectors are
$$
c_i^{(0)}=(\theta_i,0,\ldots,0)^\top, \qquad i\in\{S,M,W\}.
$$
Starting from these states, we perform a parameter continuation in
$\varepsilon$. Given a final value $\bar\varepsilon>0$, we introduce the
uniform mesh
$
\varepsilon_j=\frac{j}{K}\bar\varepsilon,$ for $ j=0,\ldots,K.$ For each branch $i\in\{S,M,W\}$, the solution $c_i^{(j-1)}$ computed at
$\varepsilon_{j-1}$ is used as the initial guess for the damped Newton
method applied to
\begin{equation}
\mathbf R_{\varepsilon_j}(c)=0.
\label{eq:numerical-continuation-corrector}
\end{equation}
The resulting solution $c_i^{(j)}$ is then used at the next continuation
step. Thus, the previous branch state acts as a zeroth-order predictor and
Newton's method as a corrector. If the corrector fails, the continuation mesh
must be refined.

For $c=(c_0,\ldots,c_N)^\top$, the discrete  energy is
\begin{equation}
\begin{split}
 G_{\varepsilon,N}(c) :={}& \frac{\nu}{2} \sum_{n=0}^N n(n+1)M_n c_n^2 + \sum_{\ell=1}^{N_q}w_\ell
\left[ \frac{\mathfrak e}{5}\abs{(Vc)_\ell}^5-Q_\varepsilon(x_\ell)B\bigl((Vc)_\ell\bigr)
-q(Vc)_\ell \right].
\end{split}
\label{eq:numerical-discrete-energy}
\end{equation}
The diffusion contribution, equivalently written as
$\frac{\nu}{2}c^\top M\Lambda c$, is evaluated exactly in coefficient
space, while the potential term is computed by Gauss--Legendre quadrature.
The energy is evaluated at the three final stationary states in order to
compare their energy levels.

For a stationary coefficient vector $c$, let $T=Vc$ denote the
corresponding values of $u_N$ at the quadrature nodes. The Hessian matrix
of the discrete energy \eqref{eq:numerical-discrete-energy} is
\begin{equation}
H_\varepsilon(c) = \nu M\Lambda + V^\top W D_\varepsilon(T)V,
\label{eq:numerical-discrete-hessian}
\end{equation}
where $D_\varepsilon(T)$ is the diagonal matrix with entries
$$
\bigl[D_\varepsilon(T)\bigr]_{\ell\ell} = 4\mathfrak e\abs{T_\ell}^3 -Q_\varepsilon(x_\ell)\beta'(T_\ell).
$$
To determine the stability of the stationary state, we consider the second
variation of the energy on the finite-dimensional space
$
X_N:=\operatorname{span}\{P_0,\ldots,P_N\}.$ More precisely, we look for $v_N\in X_N\setminus\{0\}$ and
$\lambda\in\mathbb R$ such that
\begin{equation}
\bigl( G_{\varepsilon}\bigr)'' (u_N)[v_N,w_N] = \lambda (v_N,w_N)_{\mathrm L^2(-1,1)} \qquad \text{for every }w_N\in X_N.
\label{eq:numerical-variational-eigenproblem}
\end{equation}
Writing
$
v_N=\sum_{n=0}^N z_nP_n,$ and $ w_N=\sum_{n=0}^N y_nP_n,
$ then \eqref{eq:numerical-variational-eigenproblem} takes the form
$$
y^\top H_\varepsilon(c)z =  \lambda y^\top Mz,
$$
Since this identity must hold for every coefficient vector $y$, the
spectral problem reduces to
\begin{equation}
H_\varepsilon(c)z = \lambda Mz.
\label{eq:numerical-generalised-hessian}
\end{equation}
The numerical Morse index is the number of eigenvalues of
\eqref{eq:numerical-generalised-hessian} satisfying
$
\lambda<-\tau_{\mathrm{Morse}}.
$

\subsection{Numerical results}
\label{subsec:numerical-results-1d}
In this section we illustrate the continuation of the three homogeneous
stationary states under the introduction of spatially heterogeneous
insolation. In addition to the final temperature profiles, we examine the departure
of each branch from its homogeneous state, the spectrum of the discrete
Hessian, and the geometry of the Galerkin energy on a two-dimensional
coefficient-space section.

\autoref{fig:numerics-model} collects the main ingredients of the model.
Panel (a) shows the heterogeneous insolation
$Q_\varepsilon(x)=Q_* + \varepsilon Q_1(x)$, which is larger near the equator
and smaller near the poles for the choice of $Q_1$ considered here.
Panel (b) displays the smooth co-albedo function $\beta$, whose transition
between its low- and high-temperature regimes produces the non-monotone
radiative balance shown in panel~(c). The intersections of the outgoing and
absorbed radiation curves are precisely the three homogeneous equilibria
$
\theta_S<\theta_M<\theta_W.
$
These states are the critical points of the homogeneous scalar
potential
$
R(\theta ) = \frac{\mathfrak e}{5}\abs{\theta }^5 -Q_* B(\theta )-q\theta ,
$
displayed in panel~(d). The cold and warm equilibria correspond to its two
local minima, while the middle equilibrium lies at the intervening
local maximum.

\begin{figure}[t]
    \centering

    \begin{subfigure}[t]{0.35\textwidth}
        \centering
        \includegraphics[width=\linewidth]
        {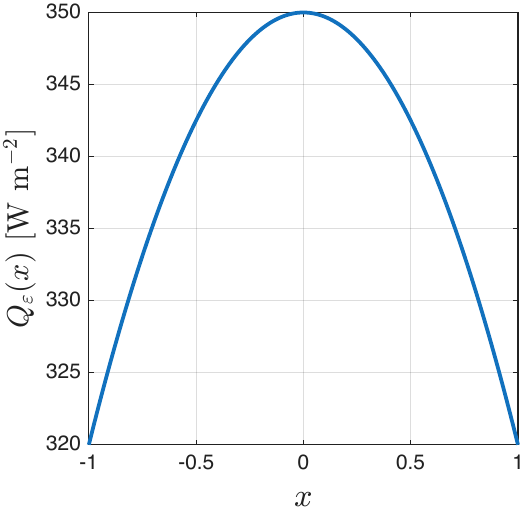}
        \caption{Insolation $Q_\varepsilon (x)$}
        \label{fig:numerics-model-insolation}
    \end{subfigure}
    \begin{subfigure}[t]{0.35\textwidth}
        \centering
        \includegraphics[width=\linewidth]
        {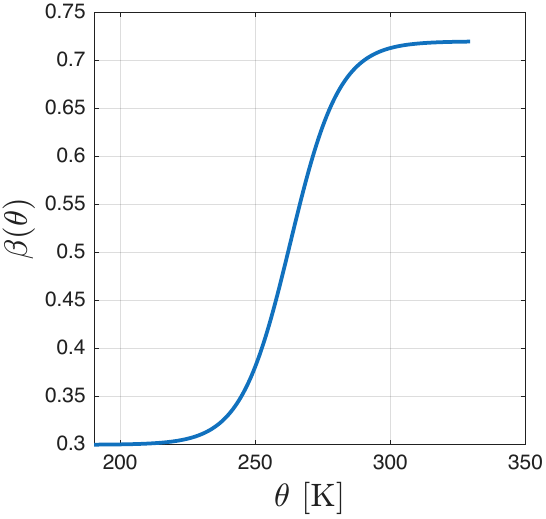}
        \caption{Co-albedo $\beta(\theta )$}
        \label{fig:numerics-model-coalbedo}
    \end{subfigure}

    \begin{subfigure}[t]{0.35\textwidth}
        \centering
        \includegraphics[width=\linewidth]
        {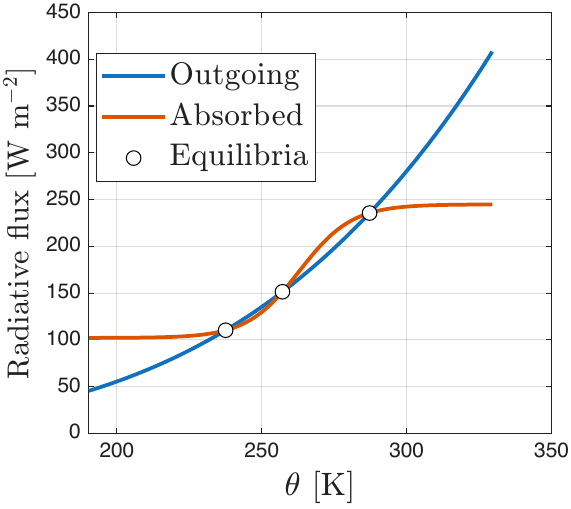}
        \caption{Radiation balance}
        \label{fig:numerics-model-balance}
    \end{subfigure}
    \begin{subfigure}[t]{0.35\textwidth}
        \centering
        \includegraphics[width=\linewidth]
        {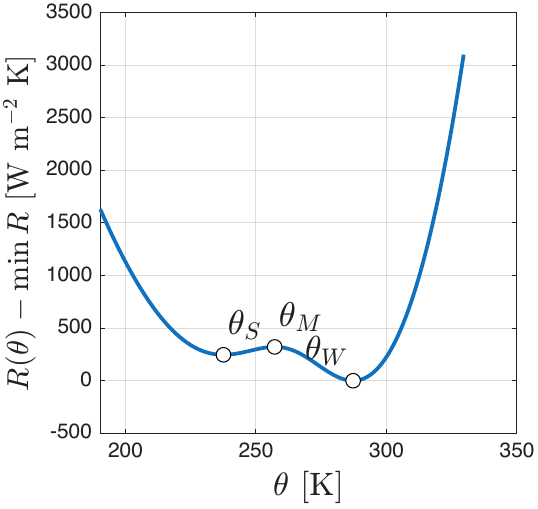}
        \caption{Shifted scalar potential $R$}
        \label{fig:numerics-model-potential}
    \end{subfigure}
    \caption{Model parametrisation and homogeneous three-climate structure. (A) Heterogeneous insolation $Q_\varepsilon (x)$ (B) Co-albedo $\beta (\theta )$ (C) Outgoing and absorbed radiation in the homogeneous problem. The intersections of the curves determine the three homogeneous
    equilibria $\theta_S<\theta_M<\theta_W$, which correspond to the critical
    points of the potential (D) Homogeneous scalar potential $R$, shifted by its minimum }
    \label{fig:numerics-model}
\end{figure}

Starting from the coefficient vectors associated with
$\theta_S$, $\theta_M$, and $\theta_W$, the continuation procedure described
in \autoref{subsec:numerical-continuation-energy-morse} produces the three
heterogeneous stationary states displayed in
\autoref{fig:numerics-profiles-and-graph-norm}\subref{fig:numerics-profiles}.
The profiles
$u_S^{\overline\varepsilon}$, $u_M^{\overline\varepsilon}$, and $u_W^{\overline\varepsilon}$
represent, respectively, the continued cold, middle, and warm climate
configurations corresponding to the three stationary solutions obtained in
\autoref{cor: corllary 1D tri-stability}.

\autoref{fig:numerics-graph-norm} shows the plot of $\varepsilon \mapsto \| u_i^{\varepsilon} - \theta_i\|_{D(A_\kappa)}$ for $i = \{S,M,W\}.$  The approximately linear growth observed is
consistent with the implicit-function construction underlying
\autoref{thm: existence three climate degenerate small eps}. Indeed, the
resulting $\rC^1$ branches admit the expansion
$$
u_i^\varepsilon = \theta_i+\varepsilon\phi_i+o(\varepsilon) \quad\text{in }D(A_\kappa),
$$
and hence $
\|u_i^\varepsilon-\theta_i\|_{D(A_\kappa)} = \varepsilon \|\phi_i\|_{D(A_\kappa)}+o(\varepsilon)
$

\begin{figure}[t]
    \centering

    \begin{subfigure}[t]{0.48\textwidth}
        \centering
        \includegraphics[width=\linewidth]
        {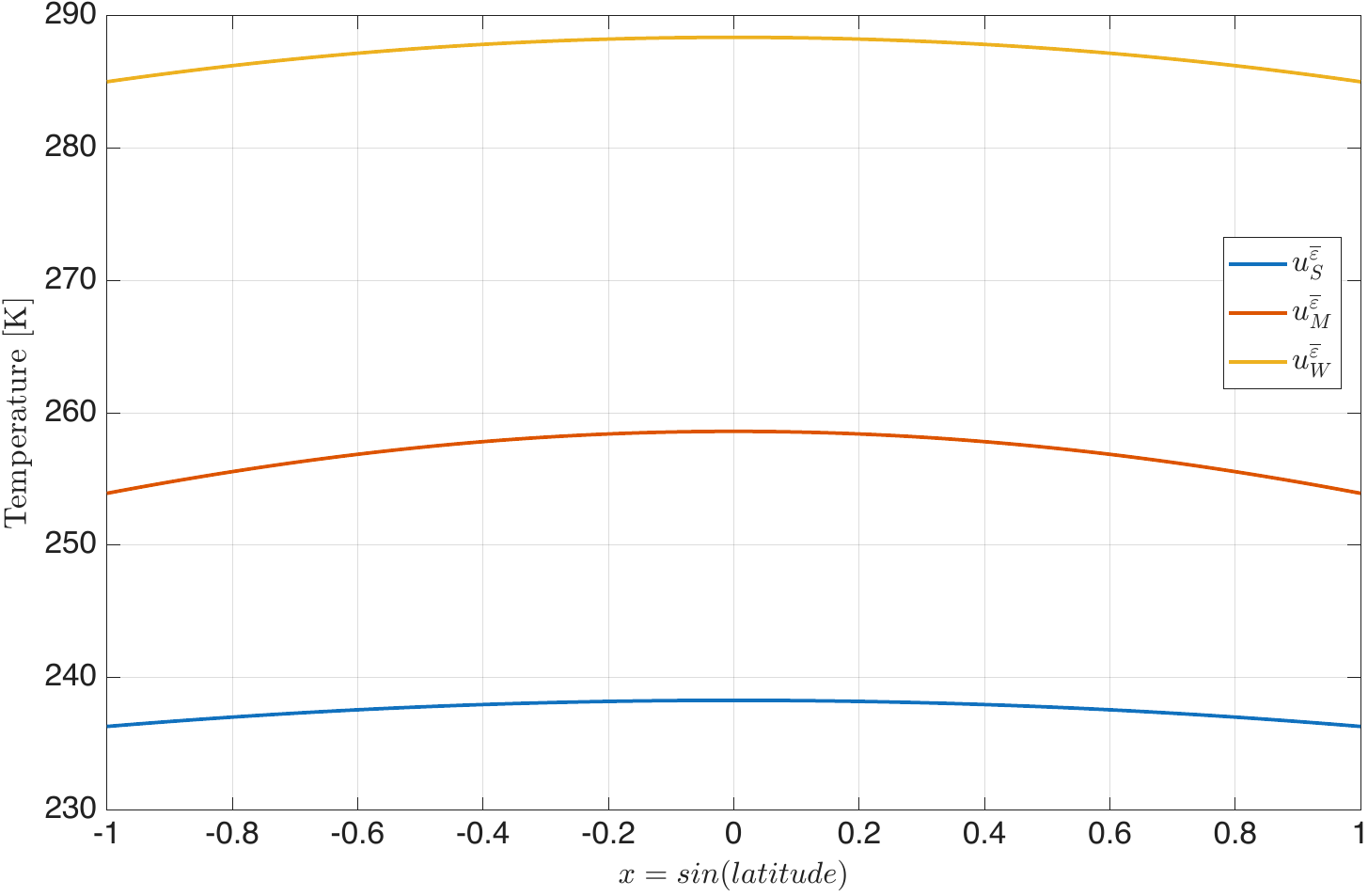}
        \caption{Stationary EBM solutions for $\overline \varepsilon = 20$}
        \label{fig:numerics-profiles}
    \end{subfigure}
    \hfill
    \begin{subfigure}[t]{0.48\textwidth}
        \centering
        \includegraphics[width=\linewidth]
        {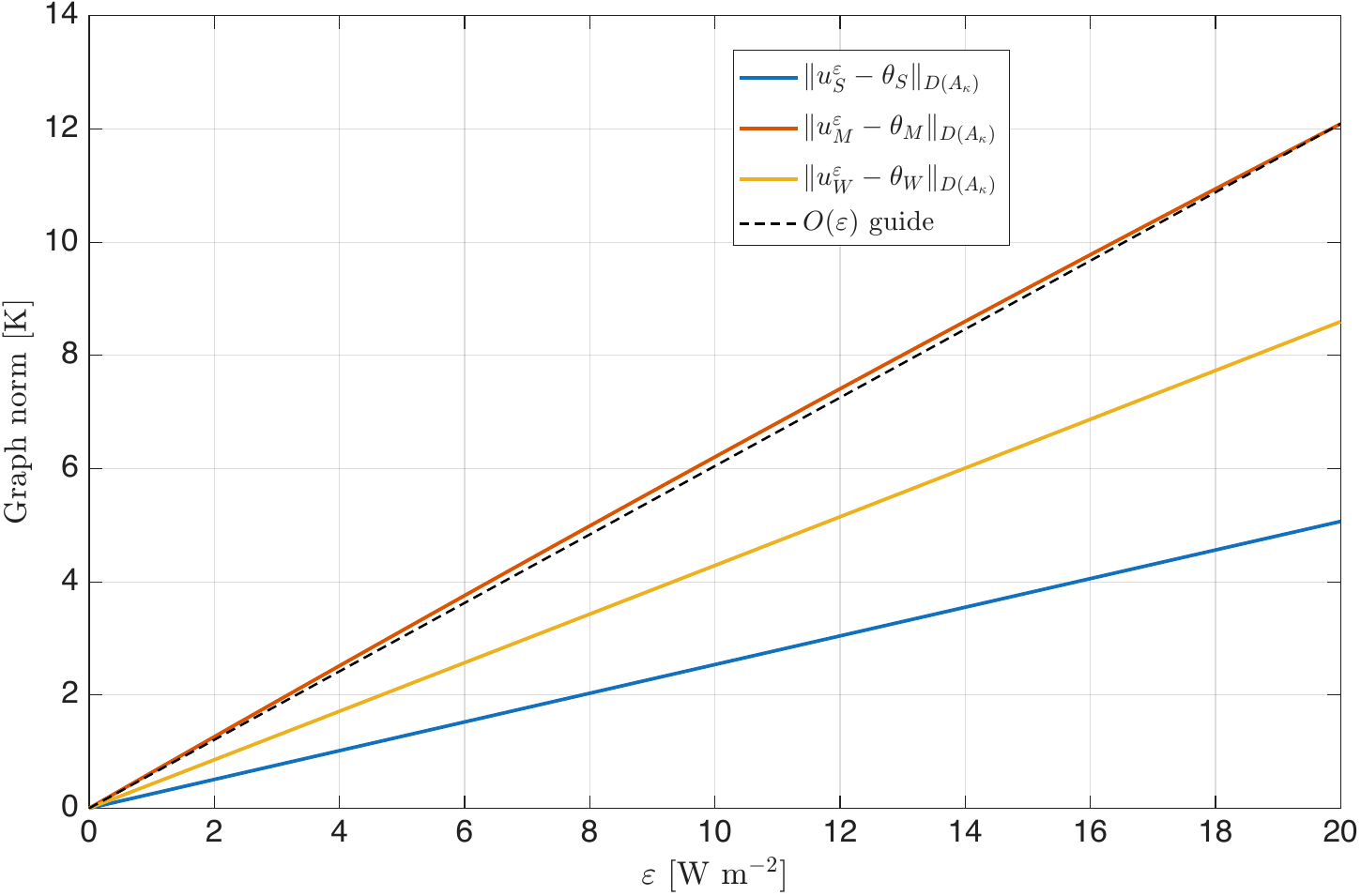}
        \caption{Map $\varepsilon \mapsto \| u_{i}^\varepsilon -\theta_i \|_{D(A_\kappa)}$, $i \in \{S,M,W\}$}
        \label{fig:numerics-graph-norm}
    \end{subfigure}

    \caption{Continuation of the three homogeneous climate states under
    heterogeneous insolation. (A) Stationary temperature profiles for $\varepsilon = 20$ (B) Graph norm distance of $u_i^\varepsilon$ from the corresponding homogeneous equilibrium $\theta_i$, i.e. $\| u_i^\varepsilon - \theta_i \|_{D(A_\kappa)} $ for $i \in \{S,M, W\}$. The dashed
    line is an
    $\mathcal O(\varepsilon)$ visual guide.}
    \label{fig:numerics-profiles-and-graph-norm}
\end{figure}

\autoref{fig:numerics-hessian} shows the lowest eigenvalues of the
discrete Hessian along the three branches, computed through
\eqref{eq:numerical-generalised-hessian} and the procedure described in \autoref{subsec:numerical-continuation-energy-morse}. Along the cold and warm branches,
the lowest eigenvalue remains positive throughout the displayed
continuation interval. Along the middle branch, the first eigenvalue
remains negative and the second remains positive. The resulting discrete
Morse-indices are thus $0,$ $1$ and $0$, for the cold, middle, and warm states, respectively. 

\begin{figure}[t]
    \centering

    \begin{subfigure}[t]{0.28\textwidth}
        \centering
        \includegraphics[width=\linewidth]
        {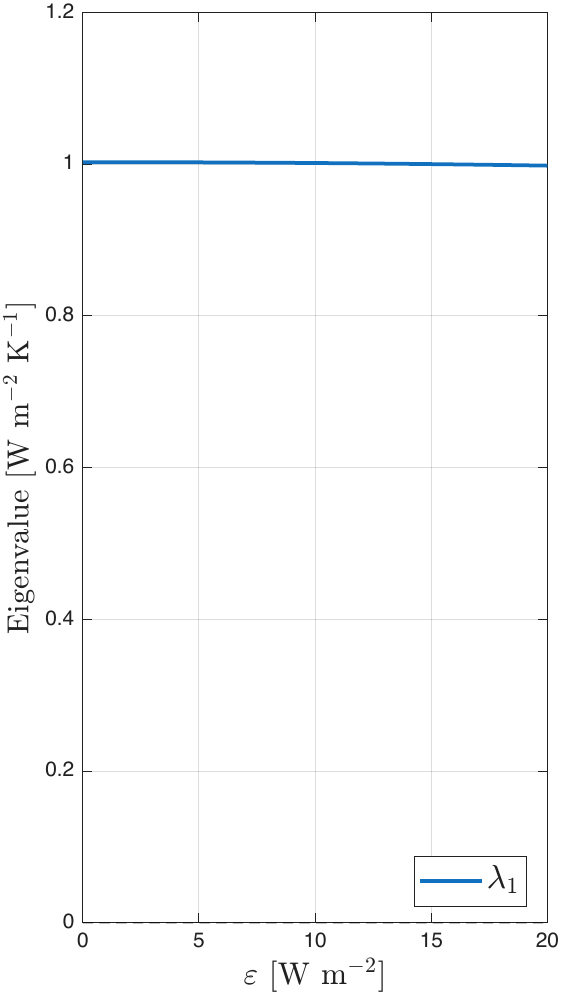}
        \caption{Cold branch.}
        \label{fig:numerics-hessian-cold}
    \end{subfigure}
    \begin{subfigure}[t]{0.27\textwidth}
        \centering
        \includegraphics[width=\linewidth]
        {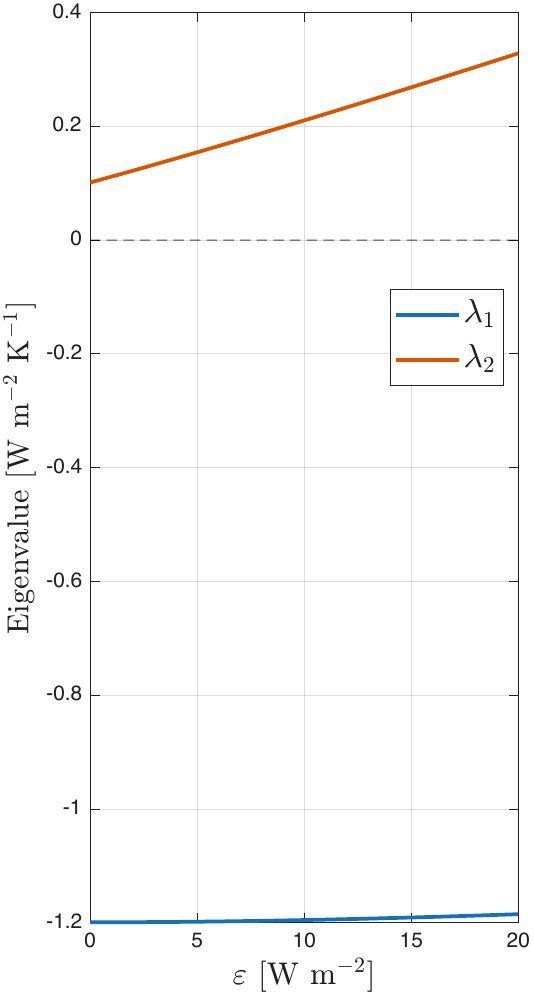}
        \caption{Middle branch.}
        \label{fig:numerics-hessian-middle}
    \end{subfigure}
    \begin{subfigure}[t]{0.27\textwidth}
        \centering
        \includegraphics[width=\linewidth]
        {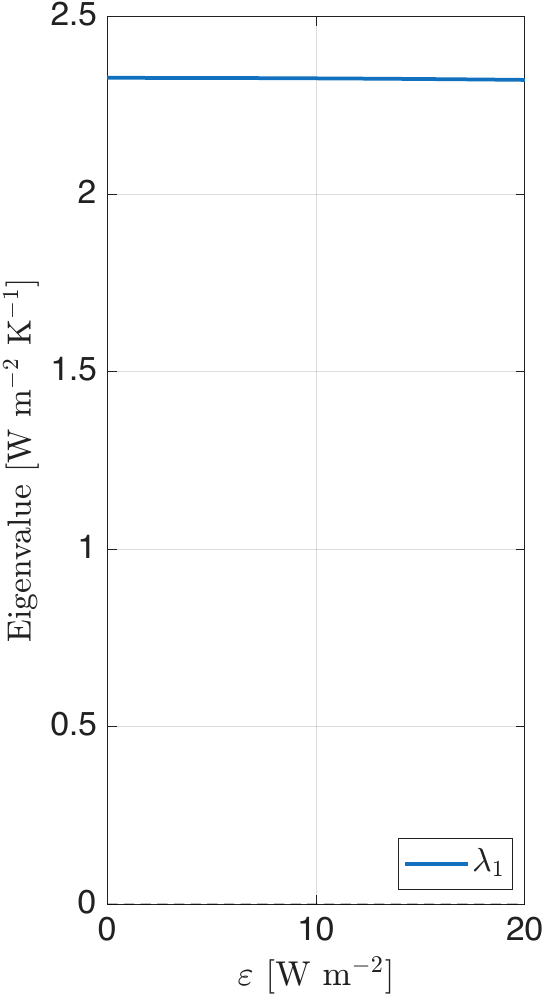}
        \caption{Warm branch.}
        \label{fig:numerics-hessian-warm}
    \end{subfigure}

    \caption{Smallest eigenvalues of the discrete Hessian along
    the three continued stationary branches. Their signs give the numerical check that the Morse-indices are $(0,1,0)$ along the branches $\varepsilon \mapsto (u_S^\varepsilon, u_M^\varepsilon, u_W^\varepsilon)$ (A) Smallest eigenvalue along the cold branch (B) Two smallest eigenvalues along the middle branch (C) Smallest eigenvalue along the warm branch.}
    \label{fig:numerics-hessian}
\end{figure}

Finally, \autoref{fig:numerics-energy-abstract} displays the complete
Galerkin energy \eqref{eq:numerical-discrete-energy} for $\varepsilon = {\overline\varepsilon}$ restricted to the
two-dimensional affine section
$$
c(a,b)=c_S^{\overline\varepsilon}+a e_1+b e_2.
$$
Let $
\langle c,d\rangle_M:=c^\top Md,$ $
\|c\|_M:=\langle c,c\rangle_M^{1/2},$ where we recall that $M=\mathrm{diag}(M_0,\ldots,M_N)$ is the Legendre mass matrix.
The first direction is the $M$-normalised cold-to-warm displacement,
$$
e_1 := \frac{c_W^{\overline\varepsilon}-c_S^{\overline\varepsilon}} {\|c_W^{\overline\varepsilon}-c_S^{\overline\varepsilon}\|_M}.
$$
Writing $v_M:=c_M^{\overline\varepsilon}-c_S^{\overline\varepsilon}$, the second direction is
obtained by subtracting from $v_M$ its $M$-orthogonal projection onto
$e_1$:
$$
\widetilde{e}_2 := v_M-\langle v_M,e_1\rangle_M e_1, \qquad e_2 := \frac{\widetilde e_2}{\|\widetilde{e}_2\|_M}.
$$
Thus $e_1$ and $e_2$ are $M$-orthonormal, and, whenever
$\widetilde e_2\neq0$, the section coincides with the affine plane
$$
c_S^{\overline\varepsilon} +\mathrm{span}\{ c_M^{\overline\varepsilon}-c_S^{\overline\varepsilon},\,
c_W^{\overline\varepsilon}-c_S^{\overline\varepsilon} \}
$$
generated by the three computed stationary states.

\begin{figure}[t]
\centering
\includegraphics[width=.7\textwidth]{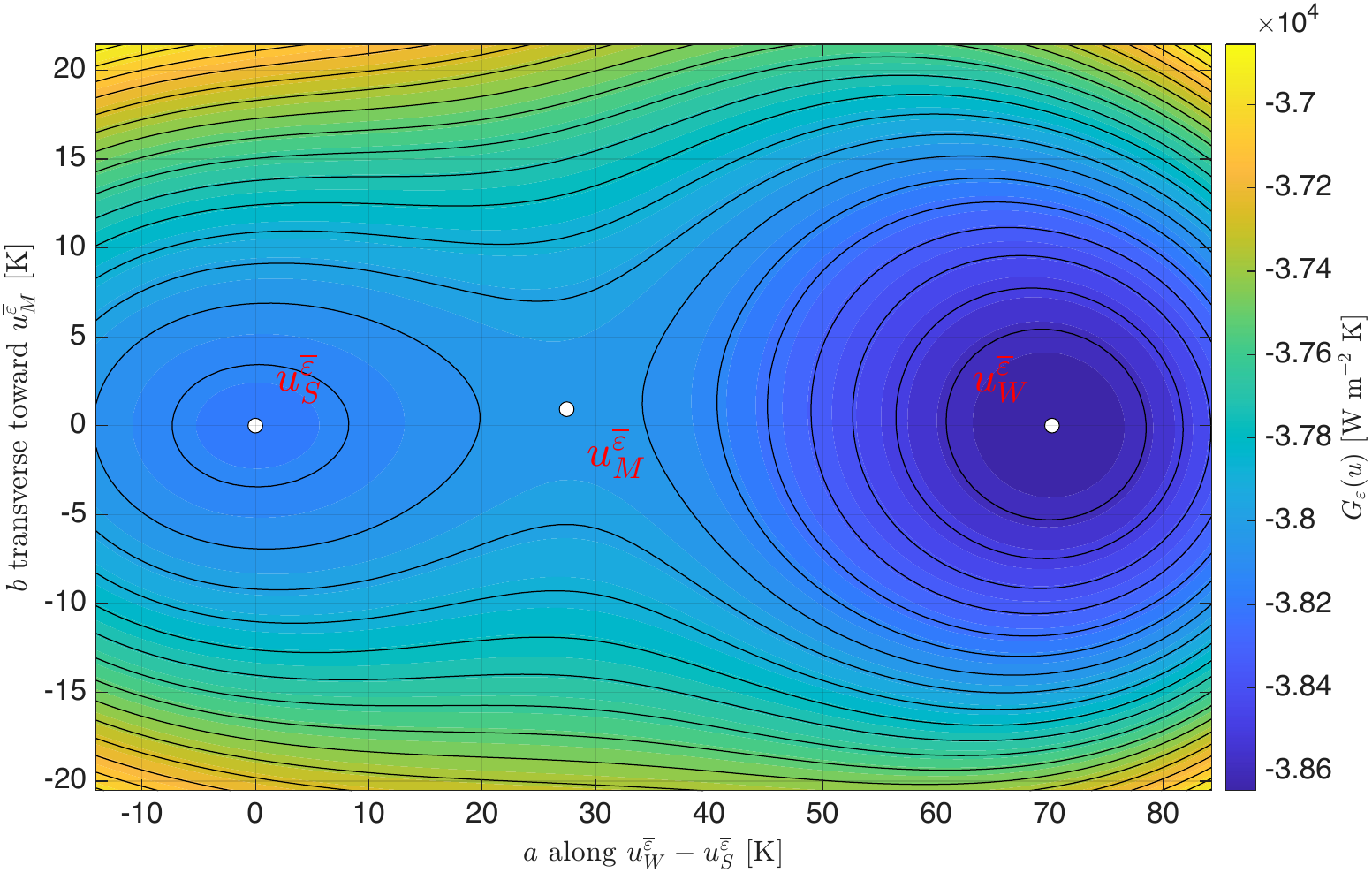}
\caption{The discrete Galerkin energy $ G_{\overline \varepsilon,N}$ restricted to the abstract affine
section $c_S^{\overline\varepsilon}+\operatorname{span}\{e_1,e_2\}$. The marked points are the cold, middle, and warm stationary states.}
\label{fig:numerics-energy-abstract}
\end{figure}

{\bf Data availability.} The MATLAB code (version R2026a) used to perform the numerical experiments
presented in \autoref{sec:numerical section} is publicly available on Zenodo at \url{https://doi.org/10.5281/zenodo.23622497}.

{\bf Acknowledgements. }{\small G.D.S. acknowledges support from the DFG project FOR~5528. Part of this research was carried out during visits of G.D.S. to La Sapienza in Rome, whose generous support and hospitality are gratefully acknowledged.
}

\bibliographystyle{plain}
\bibliography{biblio}

\end{document}